\documentclass[11pt]{amsart}
\usepackage[T1]{fontenc}
\usepackage{lmodern}
\usepackage{microtype}
\usepackage{mathtools}
\usepackage[margin=1.15in]{geometry}
\usepackage{amsmath}
\usepackage{amssymb}
\usepackage{amsthm}
\usepackage{mathrsfs}
\usepackage{longtable}
\usepackage{pdflscape}
\usepackage{diagbox}
\usepackage{xcolor}
\usepackage[colorlinks,linkcolor=blue,anchorcolor=blue,citecolor=blue,backref=page]{hyperref}
\usepackage{hyperref}
\usepackage{comment}
\newtheorem{thm}{Theorem}

\newtheorem{prop}[thm]{Proposition}
\newtheorem{coro}[thm]{Corollary}
\newtheorem{conj}[thm]{Conjecture}
\newtheorem{rmk}[thm]{Remark}

\newcommand{\Z}{\mathbb{Z}}

\def\ker{\mathrm{ker}}

\newcommand{\GL}{\mathrm{GL}}

 \def\det {\mathop{\mathrm{det}}\nolimits}

\newcommand{\SL}{\mathrm{SL}}

\newcommand{\rank}{\mathrm{rank}}
\newcommand{\vcd}{\mathrm{vcd}}
\newcommand{\cd}{\mathrm{cd}}

\title[Highest-Weight Euler Characteristics of $\mathrm{SL}_4(\mathbb{Z})$ and $\mathrm{GL}_4(\mathbb{Z})$]{Euler Characteristics of $\mathrm{SL}_4(\mathbb{Z})$ and $\mathrm{GL}_4(\mathbb{Z})$, and their cohomological consequences}

\author{Jitendra Bajpai}
\address{Department of Mathematics, Christian-Albrechts-Universit\"at zu Kiel, Heinrich-Hecht-Platz 6, 24118 Kiel, Germany and Einstein Institute of Mathematics, Edmond J. Safra Campus, The Hebrew University of Jerusalem, Givat Ram, Jerusalem 9190401, Israel}
\email{jitendra@math.uni-kiel.de}
\author{Taiwang Deng}
\address{Beijing Institute of Mathematical Sciences and\allowbreak{} Applications (BIMSA), Huairou District, Beijing, China}
\email{dengtaiw@bimsa.cn}

\begin{document}
\date{\today}


\begin{abstract}
We compute the homological Euler characteristics of $\mathrm{SL}_4(\mathbb{Z})$ and $\mathrm{GL}_4(\mathbb{Z})$ with coefficients in arbitrary irreducible rational highest-weight representations. Applying Wall’s formula, we combine the orbifold Euler characteristics of centralizers of torsion elements with traces computed using the Jacobi–Trudi identity to derive explicit formulas and rational generating functions. Consequently, these Euler characteristics are quasi-polynomial functions of the highest-weight parameters, of total degree at most two.

We further derive degreewise vanishing results and parity-sensitive lower bounds for the dimensions of  cohomology groups. The two extensions of an $\mathrm{SL}_4(\mathbb{Z})$-coefficient system to $\mathrm{GL}_4(\mathbb{Z})$ yield sharper bounds, which grow linearly or quadratically in explicit infinite families. For symmetric powers, combining our formulas with Horozov’s calculation of the determinant-twisted summand yields exact identities and lower bounds for the untwisted summand, together with a conjectural degreewise description of its cohomology.
\end{abstract}

\maketitle

\tableofcontents

\section{Introduction}\label{se:intro}
Euler characteristics are fundamental invariants connecting the cohomology of arithmetic groups with the geometry of their associated locally symmetric spaces. Let $\mathbf G$ be a connected reductive algebraic group defined over $\mathbb Q$, let $\Gamma\subseteq\mathbf G(\mathbb Q)$ be an arithmetic subgroup, and let
\[X=\mathbf G(\mathbb R)/K_\infty,\]
where $K_\infty\subseteq\mathbf G(\mathbb R)$ is a maximal compact subgroup. Since $\Gamma$ generally contains torsion elements, its action on $X$ need not be free, and the quotient $\Gamma\backslash X$ naturally has an orbifold structure. This leads to the orbifold Euler characteristic $\chi_{\mathrm{orb}}(\Gamma)$. On the other hand, for a finite-dimensional rational $\Gamma$-module $M$, the homological Euler characteristic is defined by
\[\chi_h(\Gamma,M)=\sum_{i\geq 0}(-1)^i \dim_{\mathbb{Q}}H^i(\Gamma,M),\]
whenever these cohomology groups are finite-dimensional and vanish outside a finite range. The purpose of this article is to compute the latter invariant for $\Gamma=\mathrm{SL}_4(\mathbb{Z})$ and $\mathrm{GL}_4(\mathbb{Z})$ with arbitrary irreducible highest weight coefficient systems.

Historically, Euler-characteristic computations for arithmetic groups have been closely connected with volume formulas for locally symmetric spaces. Harder's extension of the Gauss--Bonnet theorem to noncompact arithmetic quotients identifies the orbifold Euler characteristic of an arithmetic group with the volume of the corresponding quotient, measured using the appropriately normalized Euler--Poincar\'e measure~\cite{Harder71}. For certain arithmetic groups, the explicit evaluation of this measure yields formulas involving special values of zeta functions. In the present work, such orbifold Euler characteristics occur as the centralizer weights in Wall's formula for the homological Euler characteristic with coefficients.

The homological Euler characteristic records only the alternating sum of cohomological dimensions, rather than the individual cohomology groups. Nevertheless, it is often the first global invariant available when the coefficient system varies, and explicit formulas can reveal vanishing, congruence, and growth patterns that are difficult to see from isolated cohomology computations. For $\mathrm{SL}_4(\mathbb{Z})$, the associated symmetric space
\[ 
X_4=\mathrm{SL}_4(\mathbb{R})/\mathrm{SO}(4) 
\] 
has dimension $9$. Correspondingly, the Euler--Poincar\'e measure vanishes, and hence $$\chi_{\mathrm{orb}}(\mathrm{SL}_4(\mathbb{Z}))=0.$$ This is the value for the trivial coefficient system, but it does not imply the vanishing of the homological Euler characteristic for arbitrary coefficients. The present article studies precisely how the homological Euler characteristic varies with the highest weight.

\subsection{Background and previous work}

The homological and orbifold Euler characteristics of arithmetic groups were developed in the work of Serre~\cite{Serre71} and Harder~\cite{Harder71}; see also Brown~\cite{Brown94} for the general cohomological framework. The formula of Wall~\cite{Wall} expresses the homological Euler characteristic with coefficients as a sum over conjugacy classes of torsion elements, weighted by the orbifold Euler characteristics of their centralizers. Related fixed-point formulas and extensions were developed by Brown~\cite{Brown1982} and Chiswell~\cite{Chiswell76}.

For the groups considered here, computations with nontrivial coefficients were carried out by Horozov~\cite{Horozov2005,Horozov2014} for families consisting of symmetric powers of the standard representation, with determinant twists in the $\mathrm{GL}_4$ case. Our aim is to treat arbitrary irreducible highest weight coefficient systems for both $\mathrm{SL}_4(\mathbb{Z})$ and $\mathrm{GL}_4(\mathbb{Z})$ and to organize the resulting formulas by means of rational generating functions.

\subsection{Method and main results}

Let $\mathcal{M}_\lambda$ denote the irreducible representation of highest weight $\lambda$. Wall's formula takes the form
\[ \chi_h(\Gamma,\mathcal{M}_\lambda)=\sum_{(T)}\chi_{\mathrm{orb}}\bigl(C_\Gamma(T)\bigr)\operatorname{Tr}\bigl(T^{-1}\mid\mathcal{M}_\lambda\bigr),\]
where $(T)$ runs over the $\Gamma$-conjugacy classes of torsion elements. Thus the computation separates into three tasks: determining the conjugacy classes that can contribute, computing the orbifold Euler characteristics of their centralizers, and evaluating the traces of their representatives on $\mathcal{M}_\lambda$.

We first determine the possible cyclotomic characteristic polynomials of torsion elements and organize the corresponding conjugacy-class contributions using Horozov's compressed formula. We then compute the traces in Wall's formula by applying the Jacobi--Trudi identity to the eigenvalues of each representative. Combining these trace calculations with the centralizer weights gives explicit formulas for the homological Euler characteristics. The formulas exhibit a parity phenomenon in the highest weight parameters, governed more precisely by their residue classes modulo $12$. Throughout, when we say that a multivariable quasi-polynomial has period $p$, we mean that $p$ is its least positive common scalar period; equivalently, $p$ is the least positive integer for which the function is polynomial on every residue class modulo $p$ in all variables simultaneously.

For $\mathrm{SL}_4(\mathbb{Z})$, the highest weights are parameterized by three nonnegative integers $(\ell_1,\ell_2,\ell_3)$. We collect the Euler characteristics into a rational generating function $P(x,y,z)$. For $\mathrm{GL}_4(\mathbb{Z})$, the determinant component of the highest weight reduces, for the purpose of these Euler characteristics, to its parity. The corresponding untwisted and determinant-twisted contributions are encoded by rational generating functions $P_1(x,y,z)$ and $P_2(x,y,z)$. These generating functions give a uniform description of the computed invariants. Their denominators also provide an interpretation in terms of vector partition functions, from which quasi-polynomiality follows.

\subsection{Outline of the article}

Section~\ref{se:basic} recalls the homological and orbifold Euler characteristics, Wall's formula, and the resultant factors used in the computations. Section~\ref{se:torsion} determines the possible characteristic polynomials and the relevant conjugacy classes of torsion elements in $\mathrm{SL}_4(\mathbb{Z})$ and $\mathrm{GL}_4(\mathbb{Z})$. Section~\ref{se:traces} computes the traces of representatives of these classes on arbitrary highest weight representations using the Jacobi--Trudi identity. Section~\ref{se:Euler-Sl4} evaluates Wall's formula for $\mathrm{SL}_4(\mathbb{Z})$, derives the generating function $P(x,y,z)$, and explains the resulting quasi-polynomiality. Section~\ref{se:Euler-GL4} carries out the corresponding calculation for $\mathrm{GL}_4(\mathbb{Z})$ and obtains the generating functions $P_1(x,y,z)$ and $P_2(x,y,z)$. Section~\ref{se:cohomological-consequences} derives cohomological vanishing results and parity-sensitive lower bounds, studies infinite families with growing cohomology, and establishes exact identities and a degreewise conjecture for symmetric powers. Finally, Section\ref{se:conclusion} discusses cohomological applications and further directions involving boundary, Eisenstein, and cuspidal cohomology, residue-polynomial nonvanishing, and higher-rank generalizations.

\section{Background and Notation}\label{se:basic}

Let $\Gamma$ be a group and let $\mathcal{V}$ be a finite-dimensional rational $\Gamma$-module. Assume that the cohomology groups $H^i(\Gamma,\mathcal{V})$ are finite-dimensional and vanish for all sufficiently large $i$, the \emph{homological Euler characteristic} of $\Gamma$ with coefficients in $\mathcal{V}$ is
\begin{equation}\label{eq:hec}
  \chi_h(\Gamma,\mathcal{V}) =\sum_{i\geq 0}(-1)^i\dim_{\mathbb{Q}}H^i(\Gamma,\mathcal{V}).
\end{equation}

The arithmetic groups considered in this article are finitely generated linear groups over a field of characteristic zero. Consequently, Selberg's lemma~\cite{Selberg60} ensures that each such group $\Gamma$ contains a torsion-free subgroup $\Gamma'\subseteq\Gamma$ of finite index, and the Borel--Serre compactification~\cite{BorelSerre73} shows that $\Gamma'$ admits a finite classifying space. Hence, the required finiteness hypotheses are satisfied: for every finite-dimensional rational $\Gamma$-module $\mathcal{V}$, the groups $H^i(\Gamma,\mathcal{V})$ are finite-dimensional and vanish in degrees above the virtual cohomological dimension of $\Gamma$. See~\cite{Brown94,Serre71} for the general cohomological framework.

Suppose now that $\Gamma$ admits a torsion-free subgroup $\Gamma'$ of finite index for which $\chi_h(\Gamma',\mathbb{Q})$ is defined. The \emph{orbifold Euler characteristic} of $\Gamma$ is
\begin{equation}\label{eq:orbeuler}
  \chi_{\mathrm{orb}}(\Gamma)=\frac{\chi_h(\Gamma',\mathbb{Q})}{[\Gamma:\Gamma']}.
\end{equation}
This number is independent of the choice of $\Gamma'$. In particular, if $\Gamma$ itself is torsion-free, then $\chi_{\mathrm{orb}}(\Gamma)=\chi_h(\Gamma,\mathbb{Q})$. Consequently, their orbifold Euler characteristics are defined.

Wall's formula~\cite{Wall} relates these two Euler characteristics. In the setting used here, it takes the form
\begin{equation}\label{eq:hecT}
  \chi_h(\Gamma,\mathcal{V})=\sum_{(T)}\chi_{\mathrm{orb}}\bigl(C_\Gamma(T)\bigr)\operatorname{Tr}\bigl(T^{-1}\mid\mathcal{V}\bigr),
\end{equation}
where $(T)$ runs over the $\Gamma$-conjugacy classes of torsion elements and
\[
  C_\Gamma(T)=\{\gamma\in\Gamma: \gamma T=T\gamma\}
\]
is the centralizer of $T$ in $\Gamma$. Thus the orbifold Euler characteristics of the centralizers occur as weights in the computation of the homological Euler characteristic with coefficients.

We shall use the following standard properties whenever the Euler characteristics involved are defined:
\begin{enumerate}
\item If $\Gamma$ is torsion-free, then $\chi_{\mathrm{orb}}(\Gamma)=\chi_h(\Gamma,\mathbb{Q})$.
\item If $\Gamma$ is finite, then $\chi_{\mathrm{orb}}(\Gamma)=1/|\Gamma|$.
\item For a short exact sequence
  \[
    1\longrightarrow\Gamma_1\longrightarrow\Gamma \longrightarrow\Gamma_2\longrightarrow 1
  \]
  of groups satisfying the appropriate finiteness hypotheses, one has
  \[
    \chi_{\mathrm{orb}}(\Gamma) =\chi_{\mathrm{orb}}(\Gamma_1) \chi_{\mathrm{orb}}(\Gamma_2).
  \]
\end{enumerate}

We next introduce the resultant factor occurring in the specialized form of Wall's formula used below. If $f_1(x)=\prod_i(x-\alpha_i)$ and
$f_2(x)=\prod_j(x-\beta_j)$, set
\[
  \operatorname{Res}(f_1,f_2) =\prod_{i,j}(\alpha_i-\beta_j).
\]
Let $f=f_1\cdots f_d$, where the $f_i$ are powers of pairwise distinct irreducible polynomials over $\mathbb{Q}$. We define
\[
  \operatorname{Res}(f)  =\prod_{i<j}\operatorname{Res}(f_i,f_j),
\]
with the convention that $\operatorname{Res}(f)=1$ when $d=1$. Only the absolute value $\left|\operatorname{Res}(f)\right|$ will occur below, so the resulting factor is independent of the ordering of the factors $f_i$.

For later use, put
\[
T_3=\begin{pmatrix}0&1\\-1&-1\end{pmatrix},\qquad
T_4=\begin{pmatrix}0&1\\-1&0\end{pmatrix},\qquad
T_6=\begin{pmatrix}0&-1\\1&1\end{pmatrix}.
\]
Horozov's specialization of Wall's formula~\cite{Horozov2005,Horozov2014} states that, for $\Gamma=\mathrm{GL}_n(\mathbb{Z})$, or for $\Gamma=\mathrm{SL}_n(\mathbb{Z})$ with $n$ odd,
\begin{equation}\label{eq:hnthor}
  \chi_h(\Gamma,\mathcal{V})=\sum_A \bigl\lvert\operatorname{Res}(f_A)\bigr\rvert\chi_{\mathrm{orb}}\bigl(C_\Gamma(A)\bigr) \operatorname{Tr}\bigl(A^{-1}\mid\mathcal{V}\bigr),
\end{equation}
where $f_A$ is the characteristic polynomial of $A$. The sum is over the block diagonal matrices $A\in\Gamma$ satisfying the following conditions, with matrices that differ only in the ordering of their blocks counted once:
\begin{itemize}
\item every diagonal block belongs to $\{1,-1,T_3,T_4,T_6\}$;
\item each of $T_3,T_4,T_6$ occurs at most once;
\item each of $1$ and $-1$ occurs at most twice.
\end{itemize}
Here, two copies of $1$, respectively of $-1$, correspond to the blocks $I_2$, respectively $-I_2$, in Horozov's notation. In particular, no such matrix exists when $n>10$, and hence the sum in~\eqref{eq:hnthor} is then empty. More precisely, the scalar blocks contribute at most four dimensions, while the three blocks $T_3,T_4,T_6$ contribute at most six. Hence no matrix satisfying these block conditions exists when $n>10$, and the sum in~\eqref{eq:hnthor} is empty in that range.

Every matrix $A$ occurring in \eqref{eq:hnthor} has the same multiset of eigenvalues as $A^{-1}$, so $A$ and $A^{-1}$ are conjugate in $\mathrm{GL}_n(\mathbb{C})$. The coefficient systems considered here are restrictions of rational algebraic representations; their characters are constant on $\mathrm{GL}_n(\mathbb{C})$-conjugacy classes. Therefore
\[
  \operatorname{Tr}\bigl(A^{-1}\mid\mathcal{V}\bigr)=\operatorname{Tr}\bigl(A\mid\mathcal{V}\bigr).
\]
We will use this simplification throughout. Related extensions of Wall's formula to other classes of groups were developed by Chiswell~\cite{Chiswell76}.

Note that for $n=4$, we apply~\eqref{eq:hnthor} directly to $\mathrm{GL}_4(\mathbb Z)$; the corresponding formula for $\mathrm{SL}_4(\mathbb Z)$ will be obtained in Subsection~\ref{ss:passage-GLtoSL} from the index-two relation between these two groups.

\section{Torsion Elements in \texorpdfstring{$\mathrm{SL}_4(\mathbb{Z})$}{SL4(Z)} and \texorpdfstring{$\mathrm{GL}_4(\mathbb{Z})$}{GL4(Z)}}\label{se:torsion}

Wall's formula~\eqref{eq:hecT} expresses the homological Euler characteristic as a sum indexed by conjugacy classes of torsion elements. As a first step toward determining the classes that can occur, we classify the possible characteristic polynomials of torsion elements in $\mathrm{GL}_4(\mathbb{Z})$ and identify those satisfying the determinant-one condition.

Let $T\in\mathrm{GL}_4(\mathbb Z)$ be an element of finite order $k$. For $d\geq 1$, let $\zeta_d=e^{2\pi i/d}$. The $d$-th cyclotomic polynomial is
\[
    \Phi_d(x)=\prod_{\substack{1\leq j\leq d\\ \gcd(j,d)=1}}\left(x-\zeta_d^j\right)\in\mathbb Z[x].
\]
Its roots are precisely the primitive $d$-th roots of unity, and
\[
    x^k-1=\prod_{d\mid k}\Phi_d(x).
\]
Writing
\[
    \varphi(d)=\#\{1\leq j\leq d:\gcd(j,d)=1\}
\]
for Euler's totient function, we have $\deg\Phi_d=\varphi(d)$.

Since $T^k=I$, its minimal polynomial divides $x^k-1$. In characteristic zero, $x^k-1$ is separable. Hence the minimal polynomial of $T$ is a product of distinct cyclotomic polynomials, and $T$ is semisimple over $\mathbb C$. Its characteristic polynomial therefore has the form
\[
    f_T(x)=\det(xI-T)=\prod_{d\geq 1}\Phi_d(x)^{m_d},\qquad m_d\in\mathbb Z_{\geq 0},
\]
where only finitely many $m_d$ are nonzero. Since $\deg f_T=4$, these multiplicities satisfy
\[
    \sum_{d\geq 1}m_d\varphi(d)=4.
\]

In particular, only those $d$ for which $\varphi(d)\leq 4$ can occur; namely,
\[
  d\in\{1,2,3,4,5,6,8,10,12\}.
\]
Enumerating the solutions of the degree equation gives the $24$ possible characteristic polynomials listed in Table~\ref{characteristic-gl4}. Each listed polynomial is realized by an integral matrix of finite order: for $\prod_d\Phi_d^{m_d}$, take the block diagonal matrix containing $m_d$ copies of the companion matrix of $\Phi_d$, which has order $d$.

It remains to determine which of these polynomials can occur for elements of $\mathrm{SL}_4(\mathbb{Z})$. Since $f_T(x)=\det(xI-T)$ has degree four, its constant term is $f_T(0)=\det(T)$. Moreover,
\[
  \Phi_1(0)=-1\qquad\text{and}\qquad \Phi_d(0)=1\quad(d>1).
\]
Consequently,
\[
  \det(T)=f_T(0)=(-1)^{m_1}.
\]
Thus a torsion element of $\mathrm{GL}_4(\mathbb{Z})$ belongs to $\mathrm{SL}_4(\mathbb{Z})$ if and only if the multiplicity of $\Phi_1(x)=x-1$ in its characteristic polynomial is even. Of the $24$ possible characteristic polynomials, exactly $19$ satisfy this condition.

We emphasize that Table~\ref{characteristic-gl4} classifies possible characteristic polynomials, rather than integral conjugacy classes. The characteristic polynomial determines the rational conjugacy type of a finite-order matrix, but it need not determine its integral conjugacy class. Indeed, Yang's classification~\cite{Yang2015} gives $45$ $\mathrm{GL}_4(\mathbb{Z})$-conjugacy classes of torsion elements, although only the $24$ characteristic polynomials listed below occur. Thus several integral conjugacy classes may have the same characteristic polynomial. This distinction is incorporated in Horozov's compressed form of Wall's formula~\eqref{eq:hnthor}, where the absolute resultant accounts for the combined contribution of the integral conjugacy classes corresponding to a given block type (that is sharing the same characteristic polynomial). The present section therefore provides only the preliminary classification by characteristic polynomial.

{\begin{center}
\scriptsize\renewcommand{\arraystretch}{2.0}
\begin{longtable}{|c|c|c|c|}
\caption{Torsion characteristic polynomials in $\mathrm{GL}_4(\mathbb{Z})$ and $\mathrm{SL}_4(\mathbb{Z})$.}
\label{characteristic-gl4}\\
\hline
No. & Polynomial & Expanded form & Occurs In $\mathrm{SL}_4(\mathbb{Z})$ \\
\hline
\hline
1 & $\Phi_1^4$ & $(x-1)^4$ & Yes \\
\hline
2 & $\Phi_1^3\Phi_2$ & $(x-1)^3(x+1)$ & No \\
\hline
3 & $\Phi_1^2\Phi_2^2$ & $(x-1)^2(x+1)^2$ & Yes \\
\hline
4 & $\Phi_1^2\Phi_3$ & $(x-1)^2(x^2+x+1)$ & Yes \\
\hline
5 & $\Phi_1^2\Phi_4$ & $(x-1)^2(x^2+1)$ & Yes \\
\hline
6 & $\Phi_1^2\Phi_6$ & $(x-1)^2(x^2-x+1)$ & Yes \\
\hline
7 & $\Phi_1\Phi_2^3$ & $(x-1)(x+1)^3$ & No \\
\hline
8 & $\Phi_1\Phi_2\Phi_3$ & $(x-1)(x+1)(x^2+x+1)$ & No \\
\hline
9 & $\Phi_1\Phi_2\Phi_4$ & $(x-1)(x+1)(x^2+1)$ & No \\
\hline
10 & $\Phi_1\Phi_2\Phi_6$ & $(x-1)(x+1)(x^2-x+1)$ & No \\
\hline
11 & $\Phi_2^4$ & $(x+1)^4$ & Yes \\
\hline
12 & $\Phi_2^2\Phi_3$ & $(x+1)^2(x^2+x+1)$ & Yes \\
\hline
13 & $\Phi_2^2\Phi_4$ & $(x+1)^2(x^2+1)$ & Yes \\
\hline
14 & $\Phi_2^2\Phi_6$ & $(x+1)^2(x^2-x+1)$ & Yes \\
\hline
15 & $\Phi_3^2$ & $(x^2+x+1)^2$ & Yes \\
\hline
16 & $\Phi_3\Phi_4$ & $(x^2+x+1)(x^2+1)$ & Yes \\
\hline
17 & $\Phi_3\Phi_6$ & $(x^2+x+1)(x^2-x+1)$ & Yes \\
\hline
18 & $\Phi_4^2$ & $(x^2+1)^2$ & Yes \\
\hline
19 & $\Phi_4\Phi_6$ & $(x^2+1)(x^2-x+1)$ & Yes \\
\hline
20 & $\Phi_5$ & $x^4+x^3+x^2+x+1$ & Yes \\
\hline
21 & $\Phi_6^2$ & $(x^2-x+1)^2$ & Yes \\
\hline
22 & $\Phi_8$ & $x^4+1$ & Yes \\
\hline
23 & $\Phi_{10}$ & $x^4-x^3+x^2-x+1$ & Yes \\
\hline
24 & $\Phi_{12}$ & $x^4-x^2+1$ & Yes \\
\hline
\end{longtable}
\end{center}
}

\section{Traces of Torsion Elements}\label{se:traces}

Let
\[
  \Gamma=\mathrm{SL}_4(\mathbb{Z}),\qquad
  \widetilde{\Gamma}=\mathrm{GL}_4(\mathbb{Z}),
\]
and let \(V\) denote the standard four-dimensional representation.  

Let \(\mathbf T\subseteq \mathrm{GL}_4\) be the diagonal maximal
torus and let \(\mathbf B\subseteq \mathrm{GL}_4\) be the Borel
subgroup of upper-triangular matrices. For \(1\leq i\leq 4\), let
\(\epsilon_i\in X^{*}(\mathbf T)\) be the coordinate character
defined by
\[
\epsilon_i\bigl(\operatorname{diag}(t_1,t_2,t_3,t_4)\bigr)=t_i.
\]
For \(\mathrm{SL}_4\), we use the restrictions of these characters
to \(\mathbf T\cap\mathrm{SL}_4\), and dominance is understood with
respect to \(\mathbf B\cap\mathrm{SL}_4\). Thus the fundamental
dominant weights are
\[
\omega_1=\epsilon_1,\qquad
\omega_2=\epsilon_1+\epsilon_2,\qquad
\omega_3=\epsilon_1+\epsilon_2+\epsilon_3.
\]
For \(\mathrm{GL}_4\), the additional character
\[
\omega_4=\epsilon_1+\epsilon_2+\epsilon_3+\epsilon_4
\]
is the determinant character.

A dominant integral highest weight for \(\mathrm{SL}_4\) will be written
\[
  \lambda=\ell_1\epsilon_1+\ell_2(\epsilon_1+\epsilon_2)
  +\ell_3(\epsilon_1+\epsilon_2+\epsilon_3),
  \qquad \ell_1,\ell_2,\ell_3\in\mathbb{Z}_{\geq 0}.
\]
In partition coordinates this is
\[
  (\lambda_1,\lambda_2,\lambda_3,\lambda_4)
  =(\ell_1+\ell_2+\ell_3,\ell_2+\ell_3,\ell_3,0),
\]
and the corresponding irreducible coefficient system is denoted by
\(\mathcal{M}_\lambda\).  For \(\mathrm{GL}_4\), the same convention
uses \(\ell_1,\ell_2,\ell_3\in\mathbb{Z}_{\geq0}\) and
\(\ell_4\in\mathbb{Z}\), with
\[
  (\lambda_1,\lambda_2,\lambda_3,\lambda_4)
  =(\ell_1+\ell_2+\ell_3+\ell_4,
    \ell_2+\ell_3+\ell_4,\ell_3+\ell_4,\ell_4).
\]

For a torsion element \(T\), put
\[
  H_\lambda(T)
  =\operatorname{Tr}\bigl(T^{-1}\mid\mathcal{M}_\lambda\bigr).
\]
In particular, the standard character formula for symmetric powers~\cite[\S6.1]{FH-book} gives
\begin{equation}\label{eq:hnt1}
  H_m(T)
  =\operatorname{Tr}\bigl(T^{-1}\mid\operatorname{Sym}^m V\bigr)
  =\sum_{\substack{a,b,c,d\geq0\\a+b+c+d=m}}
     \mu_1^a\mu_2^b\mu_3^c\mu_4^d,
\end{equation}
where \(\mu_1,\ldots,\mu_4\) are the eigenvalues of \(T^{-1}\).
Equivalently, using the generating series for complete symmetric
polynomials~\cite[App.~A.1]{FH-book},
\begin{equation}\label{eq:sym-trace-series}
  \sum_{m\geq0}H_m(T)t^m
  =\frac{1}{\det(I-tT^{-1})}.
\end{equation}
For the block representatives used below, $T$ and $T^{-1}$ have the same multiset of eigenvalues. Indeed, the characteristic
polynomial of $T$ is a product of cyclotomic polynomials, and inversion permutes the primitive $d$-th roots of unity occurring
as the roots of each factor $\Phi_d$. Consequently, inversion preserves the eigenvalue multiset, in agreement with the
simplification at the end of Section~\ref{se:basic}.

\subsection{Passage from $\mathrm{GL}_4(\mathbb{Z})$ to $\mathrm{SL}_4(\mathbb{Z})$}\label{ss:passage-GLtoSL}

The compressed formula~\eqref{eq:hnthor} cannot be applied directly to $\Gamma$, because its $\mathrm{SL}_n(\mathbb{Z})$ version was stated
only for odd $n$.  Instead, let $\widetilde{\mathcal{V}}$ be a rational $\widetilde{\Gamma}$-representation whose restriction to $\Gamma$ is $\mathcal{V}$.  Shapiro's lemma and the decomposition
\[\operatorname{Ind}_{\Gamma}^{\widetilde{\Gamma}}\mathbf{1}\simeq \mathbf{1}\oplus\det\]
give
\begin{equation}\label{eq:index-two-sl4}
  \chi_h(\Gamma,\mathcal{V}) =\chi_h(\widetilde{\Gamma},\widetilde{\mathcal{V}}) +\chi_h(\widetilde{\Gamma},\widetilde{\mathcal{V}}\otimes\det).
\end{equation}
Indeed,
\[
  \operatorname{Ind}_\Gamma^{\widetilde{\Gamma}}\mathcal{V}
  \simeq \widetilde{\mathcal{V}}\otimes
         \operatorname{Ind}_\Gamma^{\widetilde{\Gamma}}\mathbf{1},
\]
and the asserted equality~\eqref{eq:index-two-sl4} follows by taking Euler characteristics;
see also~\cite[Chapter~III]{Brown94}.

Applying~\eqref{eq:hnthor} to the two terms on the right-hand side of
\eqref{eq:index-two-sl4} shows that a block type represented by
\(A\in \widetilde{\Gamma}\) is multiplied by
\[
  1+\det(A).
\]
Thus all determinant \(-1\) terms cancel, whereas all determinant-one terms are doubled.  This gives the required even-dimensional passage from $\widetilde{\Gamma}$ to $\Gamma$.

There is a corresponding conjugacy-class interpretation.  If $A\in \Gamma$ and $C_{\widetilde{\Gamma}}(A)$ contains an element of determinant \(-1\),
then its $\widetilde{\Gamma}$-conjugacy class remains a single \(\Gamma\)-conjugacy class
and
\[
  [C_{\widetilde{\Gamma}}(A):C_{\Gamma}(A)]=2,\qquad
  \chi_{\mathrm{orb}}(C_{\Gamma}(A))
  =2\chi_{\mathrm{orb}}(C_{\widetilde{\Gamma}}(A)).
\]
If $C_{\widetilde{\Gamma}}(A)\subseteq \Gamma$, then the $\widetilde{\Gamma}$-conjugacy class splits into two $\Gamma$-conjugacy classes and $C_{\widetilde{\Gamma}}(A)=C_{\Gamma}(A)\).  Among the block types below, precisely  \(K,L,N\) undergo this splitting, and their two families are distinguished by the signs \(+\) and \(-\).

The rows below are the block types occurring in Horozov's compressed resultant formula, rather than a complete list of integral conjugacy classes. Each row records the combined contribution associated with the displayed block type.

There are \(19\) determinant-one characteristic-polynomial types in Section~\ref{se:torsion}.  The four irreducible degree-four types  
\(\Phi_5,\Phi_8,\Phi_{10},\Phi_{12}\) do not occur in the compressed sum, leaving \(15\) block types.  After the splittings of \(K,L,N\), these give \(18\) displayed rows.  The five rows \(A,F,J,M,P\) have zero weight, so \(13\) rows contribute nontrivially.

Put \(C_n=\mathbb{Z}/n\mathbb{Z}\) and
\[
  w=\begin{pmatrix}1&0\\0&-1\end{pmatrix},\qquad
  \Gamma_B
  =\bigl\{(g,h)\in\mathrm{GL}_2(\mathbb{Z})^2:
          \det(g)\det(h)=1\bigr\}.
\]
For each displayed representative \(A\), set
\[
  R_A=\bigl\lvert\operatorname{Res}(f_A)\bigr\rvert,\qquad
  \chi_A=\chi_{\mathrm{orb}}\bigl(C_\Gamma(A)\bigr),\qquad
  W(A)=R_A\chi_A.
\]

\begingroup
\scriptsize
\setlength{\tabcolsep}{3pt}
\renewcommand{\arraystretch}{1.25}
\begin{longtable}{@{}c c c c c c c@{}}
\caption{Block types and weights for
\(\mathrm{SL}_4(\mathbb{Z})\).}
\label{torsion-sl4-weights}\\
\hline
Case & \(A\) & \(f_A\) & \(C_\Gamma(A)\) & \(\chi_A\) & \(R_A\) & \(W(A)\)\\
\hline
\endfirsthead
\multicolumn{7}{c}{\tablename~\thetable\ (continued)}\\
\hline
Case & \(A\) & \(f_A\) & \(C_H(A)\) & \(\chi_A\) & \(R_A\) & \(W(A)\)\\
\hline
\endhead
\(A\) & \(I_4\) & \(\Phi_1^4\) & \(\Gamma\) & \(0\) & \(1\) & \(0\)\\
\(B\) & \(\operatorname{diag}(I_2,-I_2)\)
 & \(\Phi_1^2\Phi_2^2\) & \(\Gamma_B\)
 & \(1/288\) & \(16\) & \(1/18\)\\
\(C\) & \(\operatorname{diag}(I_2,T_3^\top)\)
 & \(\Phi_1^2\Phi_3\) & \(\mathrm{SL}_2(\mathbb{Z})\times C_6\)
 & \(-1/72\) & \(9\) & \(-1/8\)\\
\(D\) & \(\operatorname{diag}(I_2,T_4)\)
 & \(\Phi_1^2\Phi_4\) & \(\mathrm{SL}_2(\mathbb{Z})\times C_4\)
 & \(-1/48\) & \(4\) & \(-1/12\)\\
\(E\) & \(\operatorname{diag}(I_2,T_6^\top)\)
 & \(\Phi_1^2\Phi_6\) & \(\mathrm{SL}_2(\mathbb{Z})\times C_6\)
 & \(-1/72\) & \(1\) & \(-1/72\)\\
\(F\) & \(-I_4\) & \(\Phi_2^4\) & \(\Gamma\)
 & \(0\) & \(1\) & \(0\)\\
\(G\) & \(\operatorname{diag}(-I_2,T_3^\top)\)
 & \(\Phi_2^2\Phi_3\) & \(\mathrm{SL}_2(\mathbb{Z})\times C_6\)
 & \(-1/72\) & \(1\) & \(-1/72\)\\
\(H\) & \(\operatorname{diag}(-I_2,T_4)\)
 & \(\Phi_2^2\Phi_4\) & \(\mathrm{SL}_2(\mathbb{Z})\times C_4\)
 & \(-1/48\) & \(4\) & \(-1/12\)\\
\(I\) & \(\operatorname{diag}(-I_2,T_6^\top)\)
 & \(\Phi_2^2\Phi_6\) & \(\mathrm{SL}_2(\mathbb{Z})\times C_6\)
 & \(-1/72\) & \(9\) & \(-1/8\)\\
\(J\) & \(\operatorname{diag}(T_3^\top,T_3^\top)\)
 & \(\Phi_3^2\) & \(\mathrm{GL}_2(\mathbb{Z}[\zeta_3])\)
 & \(0\) & \(1\) & \(0\)\\
\(K_+\) & \(\operatorname{diag}(T_3^\top,T_4)\)
 & \(\Phi_3\Phi_4\) & \(C_6\times C_4\)
 & \(1/24\) & \(1\) & \(1/24\)\\
\(K_-\) & \(\operatorname{diag}(T_3^\top,-T_4)\)
 & \(\Phi_3\Phi_4\) & \(C_6\times C_4\)
 & \(1/24\) & \(1\) & \(1/24\)\\
\(L_+\) & \(\operatorname{diag}(T_3^\top,T_6^\top)\)
 & \(\Phi_3\Phi_6\) & \(C_6\times C_6\)
 & \(1/36\) & \(4\) & \(1/9\)\\
\(L_-\) & \(\operatorname{diag}(T_3^\top,wT_6^\top w^{-1})\)
 & \(\Phi_3\Phi_6\) & \(C_6\times C_6\)
 & \(1/36\) & \(4\) & \(1/9\)\\
\(M\) & \(\operatorname{diag}(T_4,T_4)\)
 & \(\Phi_4^2\) & \(\mathrm{GL}_2(\mathbb{Z}[\zeta_4])\)
 & \(0\) & \(1\) & \(0\)\\
\(N_+\) & \(\operatorname{diag}(T_4,T_6^\top)\)
 & \(\Phi_4\Phi_6\) & \(C_4\times C_6\)
 & \(1/24\) & \(1\) & \(1/24\)\\
\(N_-\) & \(\operatorname{diag}(T_4,wT_6^\top w^{-1})\)
 & \(\Phi_4\Phi_6\) & \(C_4\times C_6\)
 & \(1/24\) & \(1\) & \(1/24\)\\
\(P\) & \(\operatorname{diag}(T_6^\top,T_6^\top)\)
 & \(\Phi_6^2\) & \(\mathrm{GL}_2(\mathbb{Z}[\zeta_3])\)
 & \(0\) & \(1\) & \(0\)\\
\hline
\end{longtable}
\endgroup

For each case $X$ in Table~\ref{torsion-sl4-weights}, let $T_X$ denote the displayed representative in the second column. When no confusion can arise,
we abbreviate $W(T_X)$, $H_m(T_X)$, and $H_\lambda(T_X)$ by $W(X)$, $H_m(X)$, and $H_\lambda(X)$, respectively. The corresponding symmetric-power traces are recorded separately in Table~\ref{torsion-sl4-traces}. Equalities such as
\[
H_m(T_G)=(-1)^mH_m(T_E)
\]
make the pairing of these rows transparent. 

\subsection{Symmetric-power traces}
\begingroup
\small
\setlength{\tabcolsep}{5pt}
\renewcommand{\arraystretch}{1.35}
\begin{longtable}{@{}c p{0.82\textwidth}@{}}
\caption{Symmetric-power traces for $\mathrm{SL}_4(\mathbb{Z})$ block representatives.}
\label{torsion-sl4-traces}\\
\hline
Case & \(H_m(T)=\operatorname{Tr}
  (T^{-1}\mid\operatorname{Sym}^mV)\)\\
\hline
\endfirsthead
\multicolumn{2}{c}{\tablename~\thetable\ (continued)}\\
\hline
Case & \(H_m(T)=\operatorname{Tr}
  (T^{-1}\mid\operatorname{Sym}^mV)\)\\
\hline
\endhead
\(A\) & \(\displaystyle \binom{m+3}{3}\)\\
\(B\) & \(\displaystyle
  \begin{cases}
    m/2+1,&m\equiv0\pmod2,\\
    0,&m\equiv1\pmod2;
  \end{cases}\)\\
\(C\) & \(\displaystyle \lfloor m/3\rfloor+1\)\\
\(D\) & \(\displaystyle
  \begin{cases}
    m/2+1,&m\equiv0\pmod2,\\
    2+2\lfloor m/4\rfloor,&m\equiv1\pmod2;
  \end{cases}\)\\
\(E\) & \(\displaystyle
  \begin{cases}
    m+1,&m\equiv0,5\pmod6,\\
    m+2,&m\equiv1,4\pmod6,\\
    m+3,&m\equiv2,3\pmod6;
  \end{cases}\)\\
\(F\) & \(\displaystyle (-1)^m\binom{m+3}{3}\)\\
\(G\) & \(\displaystyle (-1)^mH_m(T_E)\)\\
\(H\) & \(\displaystyle (-1)^mH_m(T_D)\)\\
\(I\) & \(\displaystyle (-1)^mH_m(T_C)\)\\
\(J\) & \(\displaystyle
  \begin{cases}
    m/3+1,&m\equiv0\pmod3,\\
    -2(m+2)/3,&m\equiv1\pmod3,\\
    (m+1)/3,&m\equiv2\pmod3;
  \end{cases}\)\\
\(K_\pm\) & \(\displaystyle
  \begin{cases}
    1,&m\equiv0,6,7\pmod{12},\\
    -1,&m\equiv1,2,8\pmod{12},\\
    0,&m\equiv4,9,10,11\pmod{12},\\
    2,&m\equiv3\pmod{12},\\
    -2,&m\equiv5\pmod{12};
  \end{cases}\)\\
\(L_\pm\) & \(\displaystyle
  \begin{cases}
    1,&m\equiv0\pmod6,\\
    -1,&m\equiv2\pmod6,\\
    0,&m\equiv1,3,4,5\pmod6;
  \end{cases}\)\\
\(M\) & \(\displaystyle
  \begin{cases}
    (-1)^{m/2}(m/2+1),&m\equiv0\pmod2,\\
    0,&m\equiv1\pmod2;
  \end{cases}\)\\
\(N_\pm\) & \(\displaystyle (-1)^mH_m(T_{K_\pm})\)\\
\(P\) & \(\displaystyle
  \begin{cases}
    (-1)^{m/3}(m/3+1),&m\equiv0\pmod3,\\
    (-1)^{(m-1)/3}\,2(m+2)/3,&m\equiv1\pmod3,\\
    (-1)^{(m-2)/3}\,(m+1)/3,&m\equiv2\pmod3.
  \end{cases}\)\\
\hline
\end{longtable}
\endgroup

Let
\[
  \mathcal{S}
  =\{B,C,D,E,G,H,I,K_+,K_-,L_+,L_-,N_+,N_-\}.
\]
Equations~\eqref{eq:index-two-sl4} and~\eqref{eq:hnthor}, together with
Table~\ref{torsion-sl4-weights}, now give
\begin{equation}\label{eq:sl4-sym-weighted-sum}
  \chi_h\bigl(\Gamma,\operatorname{Sym}^mV\bigr)
  =\sum_{A\in\mathcal{S}}W(A)H_m(A).
\end{equation}

\begin{prop}\label{prop-Sym-SL4}
Let \(V\) be the standard four dimensional representation of
\(\Gamma=\mathrm{SL}_4(\mathbb{Z})\).  Then, for every $m\in\mathbb Z_{\geq 0}$,
\begin{equation}\label{chisl4}
  \chi_h\bigl(\Gamma,\operatorname{Sym}^mV\bigr)
  =
  \begin{cases}
    -m/6,&m\equiv0\pmod6,\\
    -m/6-2/3,&m\equiv2\pmod6,\\
    -m/6-1/3,&m\equiv4\pmod6,\\
    0,&m\text{ is odd}.
  \end{cases}
\end{equation}
\end{prop}

\begin{proof}
Suppose first that \(m\) is odd. Write
\[
Z=\langle -I_4\rangle,\qquad
\overline{\Gamma}=\Gamma/Z.
\]
The Lyndon--Hochschild--Serre spectral sequence associated with
\[
1\longrightarrow Z\longrightarrow\Gamma
 \longrightarrow\overline{\Gamma}\longrightarrow1
\]
has
\[
E_2^{p,q}
 =
H^p\!\left(\overline{\Gamma},
H^q\!\left(Z,\operatorname{Sym}^m V\right)\right)
\Longrightarrow
H^{p+q}\!\left(\Gamma,\operatorname{Sym}^m V\right);
\]
see~\cite[Chapter~VII]{Brown94}.

Since \(m\) is odd, the element \(-I_4\) acts on
\(\operatorname{Sym}^m V\) as \(-1\). Hence
\[
H^0\!\left(Z,\operatorname{Sym}^m V\right)
 =
\left(\operatorname{Sym}^m V\right)^Z
 =0.
\]
Moreover, \(Z\) is finite of order \(2\), and the coefficient module is
a rational vector space. The invariants functor is therefore exact,
by averaging over \(Z\), and consequently
\[
H^q\!\left(Z,\operatorname{Sym}^m V\right)=0
\qquad(q>0).
\]
Thus, every term \(E_2^{p,q}\) of this particular spectral sequence is zero and it follows that
\[
H^n\left(\Gamma, \operatorname{Sym}^m V\right)=0 \qquad(n\geq0),
\]
and hence
\[
\chi_h\left(\Gamma, \operatorname{Sym}^m V\right)=0.
\]

Now let \(m=6k+r\), where \(r\in\{0,2,4\}\).  Substitution of the
weights in Table~\ref{torsion-sl4-weights} and the traces in
Table~\ref{torsion-sl4-traces} into
\eqref{eq:sl4-sym-weighted-sum} gives
\[
\begin{array}{c|ccc}
r&0&2&4\\ \hline
\displaystyle\sum_{A\in\mathcal{S}}W(A)H_{6k+r}(A)
 &-k&-k-1&-k-1.
\end{array}
\]
For example, when \(r=0\), the left-hand side is
\[
 \frac{3k+1}{18}-\frac{2k+1}{8} -\frac{3k+1}{12}-\frac{6k+1}{72}-\frac{6k+1}{72}-\frac{3k+1}{12}-\frac{2k+1}{8} +\frac{2}{24}+\frac{2}{9}+\frac{2}{24} =-k.
\]

Replacing \(k\) by \((m-r)/6\) in the three cases gives precisely~\eqref{chisl4}.
\end{proof}

The equivalent generating series is
\begin{equation}\label{gensl4}
\sum_{m\geq0}\chi_h\bigl(\Gamma,\operatorname{Sym}^mV\bigr)t^m = -\frac{t^2+t^4}{1-t^4-t^6+t^{10}}.
\end{equation}

\subsection{Traces for arbitrary highest weights}
Finally, the traces for arbitrary $\mathrm{SL}_4$-highest weights are obtained from
the complete symmetric traces \(H_m(T)\).  With
\[
  (\lambda_1,\lambda_2,\lambda_3,\lambda_4)
  =(\ell_1+\ell_2+\ell_3,\ell_2+\ell_3,\ell_3,0),
\]
the Jacobi--Trudi identity~\cite[Theorem~7.16.1]{St-book} gives
\begin{equation}\label{eqn-Jacobi-Trudi}
  H_\lambda(T)
  =\det\begin{pmatrix}
    H_{\lambda_1}(T)&H_{\lambda_1+1}(T)&H_{\lambda_1+2}(T)\\
    H_{\lambda_2-1}(T)&H_{\lambda_2}(T)&H_{\lambda_2+1}(T)\\
    H_{\lambda_3-2}(T)&H_{\lambda_3-1}(T)&H_{\lambda_3}(T)
  \end{pmatrix},
\end{equation}
where \(H_0(T)=1\) and \(H_j(T)=0\) for \(j<0\). Since $\lambda_4=0$, the displayed $3\times 3$ determinant is obtained from the usual $4\times 4$ Jacobi--Trudi determinant by expanding along its last row. Thus~\eqref{eqn-Jacobi-Trudi}, together with the complete symmetric traces $H_m(T)$, determines the trace
of every irreducible $\mathrm{SL}_4$ highest-weight coefficient
system on each torsion representative.

\section{Euler Characteristics of \texorpdfstring{$\mathrm{SL}_4(\mathbb{Z})$} {SL4(Z)}}\label{se:Euler-Sl4}

For
\[
  \lambda=\ell_1\epsilon_1
  +\ell_2(\epsilon_1+\epsilon_2)
  +\ell_3(\epsilon_1+\epsilon_2+\epsilon_3),
  \qquad \ell_1,\ell_2,\ell_3\in\mathbb{Z}_{\geq0},
\]
write
\[
  \chi_{(\ell_1,\ell_2,\ell_3)}
  =\chi_h\bigl(\mathrm{SL}_4(\mathbb{Z}),\mathcal{M}_\lambda\bigr).
\]
We use the residue notation
\begin{equation}\label{eq:sl4-residue-notation}
  \ell_i=12m_i+r_i,
  \qquad m_i\in\mathbb{Z}_{\geq0},\quad 0\leq r_i<12.
\end{equation}
The formulas below are indexed by the residue triple
$(r_1,r_2,r_3)$ modulo $12$, not merely by the eight possible parity
patterns.

\begin{prop}[Parity vanishing]\label{prop:sl4-parity-vanishing}
One has $\chi_{(\ell_1,\ell_2,\ell_3)}=0$
whenever
\[
  \ell_1\not\equiv\ell_3\pmod2 \qquad\text{or}\qquad \ell_1\equiv\ell_2\equiv\ell_3\equiv1\pmod2.
\]
Consequently, a nonzero value is possible only when
\[
  (\ell_1,\ell_3)\equiv(0,0)\pmod2,
\]
with arbitrary $\ell_2$, or when
\[
  (\ell_1,\ell_2,\ell_3)\equiv(1,0,1)\pmod2.
\]
\end{prop}

\begin{proof}
The central element $-I_4$ acts on $\mathcal{M}_\lambda$ by
\[
  (-1)^{\lambda_1+\lambda_2+\lambda_3}
  =(-1)^{\ell_1+2\ell_2+3\ell_3}
  =(-1)^{\ell_1+\ell_3}.
\]
If $\ell_1\not\equiv\ell_3\pmod2$, this action is nontrivial. More precisely, applying the argument from the proof of Proposition~\ref{prop-Sym-SL4}, with \(\mathcal M_\lambda\) in place of \(\operatorname{Sym}^m V\), gives $H^n\left(\Gamma,\mathcal M_\lambda\right)=0\quad(n\geq0).$
Therefore $\chi_h\left(\Gamma,\mathcal M_\lambda\right)=0.$ It remains to treat the all-odd parity pattern.  Write $|\lambda|=\lambda_1+\lambda_2+\lambda_3$.  If two block representatives satisfy $T'=-T$, then
\[
  H_m(T')=(-1)^mH_m(T)\quad\Longrightarrow\quad H_\lambda(T')=(-1)^{|\lambda|}H_\lambda(T).
\]
For odd $\ell_1,\ell_2,\ell_3$, the partition coordinates have parity $(\lambda_1,\lambda_2,\lambda_3)\equiv(1,0,1)\pmod2$ and $|\lambda|$ is even.  The $B$ and $L_\pm$ determinants vanish because their odd complete-symmetric traces vanish.  For the remaining rows, substitution of the four trace series
\[
\begin{aligned}
 \sum_{m\geq0}H_m(C)t^m
   &=\frac1{(1-t)^2(1+t+t^2)},\\
 \sum_{m\geq0}H_m(D)t^m
   &=\frac1{(1-t)^2(1+t^2)},\\
 \sum_{m\geq0}H_m(E)t^m
   &=\frac1{(1-t)^2(1-t+t^2)},\\
 \sum_{m\geq0}H_m(K_\pm)t^m
   &=\frac1{(1+t+t^2)(1+t^2)}
\end{aligned}
\]
into the Jacobi--Trudi determinant gives the identity
\begin{equation}\label{eq:all-odd-trace-identity}
  6H_\lambda(K_\pm) =9H_\lambda(C)+6H_\lambda(D)+H_\lambda(E).
\end{equation}
The paired rows $C,I$, $D,H$, $E,G$, and $K_\pm,N_\pm$ have equal
traces in this parity pattern.  Their total weighted contribution is
therefore
\[
  -\frac14H_\lambda(C)-\frac16H_\lambda(D) -\frac1{36}H_\lambda(E)+\frac16H_\lambda(K_\pm),
\]
which is zero by~\eqref{eq:all-odd-trace-identity}.  This proves the second vanishing assertion.
\end{proof}

\begin{thm}[Residue-class formulas]\label{thm:sl4-residue-formulas}
For every residue triple allowed by Proposition~\ref{prop:sl4-parity-vanishing}, there is a polynomial
\[
  Q_{r_1,r_2,r_3}(m_1,m_2,m_3)\in\mathbb{Q}[m_1,m_2,m_3]
\]
of total degree at most two such that
\[
  \chi_{(\ell_1,\ell_2,\ell_3)}=Q_{r_1,r_2,r_3}(m_1,m_2,m_3)
\]
under~\eqref{eq:sl4-residue-notation}.  There are $432$ potentially nonzero even--arbitrary--even residue triples and $216$ potentially
nonzero odd--even--odd residue triples, for a total of $648$.  The complete formulas are recorded in Appendix~\ref{app:sl4-residue-tables}.
\end{thm}

\begin{proof}
For each contributing row $A$, combine the weight $W(A)$ from Table~\ref{torsion-sl4-weights} with the Jacobi--Trudi trace~\eqref{eqn-Jacobi-Trudi}:
\begin{equation}\label{eq:sl4-general-weighted-sum}
  \chi_{(\ell_1,\ell_2,\ell_3)}=\sum_{A\in\mathcal{S}}W(A)H_\lambda(A).
\end{equation}
Every sequence $H_m(A)$ in Table~\ref{torsion-sl4-traces} is affine on residue classes modulo $12$.  Substituting~\eqref{eq:sl4-residue-notation} into the six terms of the determinant therefore gives a polynomial on each residue triple.  Simplification of this finite list gives the polynomials of degree at most two in the appendix.

For example: suppose $r_1=r_2=r_3=0$.  Then $\lambda_1,\lambda_2,\lambda_3$ are even, and the row $B$ gives
\[
\begin{aligned}
H_\lambda(B)=(36m_2+36m_3+6)m_1+36m_2^2+36m_2m_3+12m_2+6m_3+1.
\end{aligned}
\]
The remaining nonzero traces reduce to
\[
\begin{aligned}
& H_\lambda(C)=H_\lambda(I)=4m_1+4m_2+4m_3+1,\\
& H_\lambda(D)=H_\lambda(H)=6m_1+6m_3+1,\\
& H_\lambda(E)=H_\lambda(G)=12m_1-12m_2+12m_3+1,\\
& H_\lambda(K_\pm)=H_\lambda(L_\pm)=H_\lambda(N_\pm)=1.
\end{aligned}
\]
Substitution in~\eqref{eq:sl4-general-weighted-sum} yields
\begin{equation}\label{eq:sl4-zero-residue-example}
  \chi_{(12m_1,12m_2,12m_3)}=(2m_2+2m_3-2)m_1+2m_2^2+2m_2m_3-2m_3.
\end{equation}
All other entries follow from the same finite determinant calculation.
\end{proof}

\begin{coro}\label{cor:sl4-nonvanishing}
For every integer \(t\geq 1\), one has
\[
H^\bullet\!\left(
\mathrm{SL}_4(\mathbb Z),
\mathcal M_{(12t,12t,12t)}
\right)\neq 0.
\]
More precisely,
\[
\sum_{q\geq 0}
\dim H^q\!\left(
\mathrm{SL}_4(\mathbb Z),
\mathcal M_{(12t,12t,12t)}
\right)
\geq 8t^2-4t.
\]
In particular, the coefficient systems
\(\mathcal M_{(12t,12t,12t)}\), \(t\geq1\), give an infinite
family of irreducible highest-weight coefficient systems with
nonvanishing rational cohomology.
\end{coro}

\begin{proof}
Setting \(m_1=m_2=m_3=t\) in
\eqref{eq:sl4-zero-residue-example} gives
\[
\chi_h\!\left(
\mathrm{SL}_4(\mathbb Z),
\mathcal M_{(12t,12t,12t)}
\right)
=8t^2-4t.
\]
Since \(t\geq1\), this number is positive. On the other hand, the
triangle inequality applied to the alternating sum defining the
homological Euler characteristic gives
\[
\begin{aligned}
8t^2-4t
&=
\left|
\chi_h\!\left(
\mathrm{SL}_4(\mathbb Z),
\mathcal M_{(12t,12t,12t)}
\right)
\right| 
\leq
\sum_{q\geq0}
\dim H^q\!\left(
\mathrm{SL}_4(\mathbb Z),
\mathcal M_{(12t,12t,12t)}
\right).
\end{aligned}
\]
The asserted lower bound follows, and therefore at least one of these
cohomology groups is nonzero.
\end{proof}

\begin{prop}[Duality]\label{prop:sl4-duality}
For all $\ell_1,\ell_2,\ell_3\geq0$, $\chi_{(\ell_1,\ell_2,\ell_3)}=\chi_{(\ell_3,\ell_2,\ell_1)}.$

\end{prop}

\begin{proof}
The contragredient of $\mathcal{M}_{(\ell_1,\ell_2,\ell_3)}$ has
highest weight $(\ell_3,\ell_2,\ell_1)$, and
$H_{\lambda^\vee}(T)=H_\lambda(T^{-1})$.  Inversion permutes the
torsion terms in Wall's formula and leaves both the centralizer and
its orbifold Euler characteristic unchanged.  The asserted equality
follows.
\end{proof}

\subsection{Rational generating function}

Define
\begin{align*}
  P(x,y,z)={}\sum_{a,b,c\geq0}\chi_{(2a,b,2c)}x^{2a}y^bz^{2c} +\sum_{a,b,c\geq0}\chi_{(2a+1,2b,2c+1)}x^{2a+1}y^{2b}z^{2c+1}.
\end{align*}
By Proposition~\ref{prop:sl4-parity-vanishing}, this is the generating series over all potentially nonzero highest weights.  Let
\[
  \mathcal{I} =\{(a,b)\in\{0,1,\ldots,10\}^2:a\equiv b\pmod2\}.
\]
The polynomials $f_{a,b}(y)$ are specified in Table~\ref{tab:sl4-numerator} for $a\leq b$ and extended by $f_{b,a}(y)=f_{a,b}(y)$.  We put $f_{a,b}(y)=0$ when
$(a,b)\notin\mathcal{I}$.

{ \begin{center}
\tiny\renewcommand{\arraystretch}{2.4}
{\begin{longtable}{@{}c p{0.84\textwidth}@{}}
\caption{Numerator polynomials for the
$\mathrm{SL}_4(\mathbb{Z})$ generating function.}\label{tab:sl4-numerator}\\
\hline
$(a,b)$&$f_{a,b}(y)$\\
\hline
\endfirsthead
\multicolumn{2}{c}{\tablename~\thetable\ (continued)}\\
\hline
$(a,b)$&$f_{a,b}(y)$\\
\hline
\endhead
(0,0)&$y^2 + y^4 - 4y^9 + y^{14} + y^{16} $\\
\hline
(0,2)&$-1 - y - y^2 - y^3 + 2y^6 + 4y^9 + 2y^{10} + y^{12} - y^{13} - y^{14} - y^{15} - 2y^{16}$\\
\hline 
(0,4)& $-y^2 - y^5 - y^6 + y^7 + 2y^9 + 2y^{13} - y^{14} + 2y^{15} - y^{17} - y^{18} - y^{19}$\\
\hline
(0,6)&$1 + y + y^2 - 2y^4 + y^5 - 2y^6 - 2y^{10} - y^{12} + y^{13} + y^{14} + y^{17}$\\
\hline
(0,8)&$2y^3 + y^4 + y^5 + y^6 - y^7 - 2y^9 - 4y^{10} - 2y^{13} + y^{16} + y^{17} + y^{18} + y^{19}$ \\
\hline
(0,10)& $-y^3 - y^5 + 4y^{10} - y^{15} - y^{17}$\\
\hline
(1,1)& $-1 + y^4 + 2y^6 - 4y^8 + 2y^{10} + y^{12} - y^{16}$\\
\hline
(1,3)& $y^4 - 2y^6 + 3y^8 - 2y^{12} - 2y^{14} + y^{16} + y^{20}$\\
\hline
(1,5)& $1 + y^2 - 3y^4 - y^6 + 2y^8 - 2y^{10} + y^{12} + 3y^{14} - y^{16} - y^{18}$\\
\hline
(1,7)& $-2y^2 + y^4 + 2y^6 + y^8 - 4y^{10} + 2y^{12} + y^{16} - y^{20}$\\
\hline
(1,9)& $y^2 - y^6 - 2y^8 + 4y^{10}  - 2y^{12} - y^{14} + y^{18}$\\
\hline
(2,2)& $2 + 2y + 2y^2 + 2y^3 + y^4 - 4y^6 - 4y^7 - y^8 - 4y^9 - 2y^{10} - 2y^{12} + 2y^{13} + 2y^{15} + 3y^{16} + y^{20}$\\
\hline
(2,4)& $1 + y  + 2y^2 + y^3 - 2y^4 - y^6 + 2y^7 - 2y^9 - 2y^{10} - 2y^{11} - y^{12} - 3y^{13} + 2y^{14} - y^{15} + 2y^{17} + y^{18} + 2y^{19}$\\
\hline
(2,6)& $-2 - y
 - 3y^2
 - y^3
 + 2y^4
 + 4y^6
 + 2y^7
 + y^8
 - 2y^{10}
 + 2y^{11}
 + 2y^{12}
 + y^{13}
 - y^{14}
 - y^{15}
 - 2y^{17}
 - y^{20}$\\
\hline
(2,8)&$
-2y
 - y^2
 - 2y^3
 - 2y^4
 - y^6
 + 2y^7
 + 2y^9
 + 8y^{10}
 + 2y^{11}
 + 2y^{13}
 - y^{14}
 - 2y^{16}
 - 2y^{17}
 - y^{18}
 - 2y^{19}
$\\
\hline
(2,10)&$
y
 + y^2
 + y^3
 + y^4
 - 2y^7
 - 4y^{10}
 - 2y^{11}
 - y^{13}
 + y^{14}
 + y^{15}
 + y^{16}
 + 2y^{17}
$\\
\hline
(3,3)& $y^2
 - 2y^4
 - y^6
 - 2y^8
 + 4y^{12}
 + 5y^{14}
 - 2y^{16}
 - y^{18}
 - 2y^{20}$\\
\hline
(3,5)& $-1
 - 2y^2
 + 2y^4
 + 6y^6
 - y^8
 - 4y^{10}
 - y^{12}
 - 2y^{14}
 + 2y^{16}
 + 2y^{18}
 - y^{20}$\\
\hline
(3,7)&$
2
 + y^2
 - 2y^4
 - 3y^6
 - 2y^8
 + 8y^{10}
 - 2y^{12}
 - 3y^{14}
 - 2y^{16}
 + y^{18}
 + 2y^{20}
$\\
\hline
(3,9)&$
-1
 + y^4
 + 2y^8
 - 4y^{10}
 + y^{12}
 + 2y^{14}
 + y^{16}
 - 2y^{18}
$\\
\hline
(4,4)& $-1
 - y^2
 + 2y^4
 + 4y^6
 - 4y^7
 - 2y^{10}
 + 4y^{11}
 + y^{12}
 - y^{14}
 - 4y^{15}
 + 2y^{18}$\\
\hline
(4,6)& $1
 - y
 + 2y^5
 - 2y^6
 - y^7
 - y^8
 - 2y^9
 + 8y^{10}
 - 2y^{11}
 - y^{12}
 - y^{13}
 - 2y^{14}
 + 2y^{15}
 - y^{19}
 + y^{20}$\\
 \hline
 (4,8)&$
 -1
 - 2y^3
 - y^5
 - y^6
 + y^7
 + 2y^8
 + 2y^9
 - 2y^{10}
 + y^{12}
 + 2y^{13}
 + 4y^{14}
 + 2y^{16}
 - y^{17}
 - 3y^{18}
 - y^{19}
 - 2y^{20}
 $\\
 \hline
(4,10)&$
y^3
 + y^6
 + y^7
 - y^8
 - 2y^{10}
 - 2y^{14}
 + y^{15}
 - 2y^{16}
 + y^{18}
 + y^{19}
 + y^{20}
$\\ 
\hline
(5,5)& $1
 - 6y^6
 - y^8
 + 12y^{10}
 - y^{12}
 - 6y^{14}
 + y^{20}$\\
 \hline
 (5,7)&$
 -1
 + 2y^2
 + 2y^4
 - 2y^6
 - y^8
 - 4y^{10}
 - y^{12}
 + 6y^{14}
 + 2y^{16}
 - 2y^{18}
 - y^{20}
 $\\
 \hline
 (5,9)&$
 -y^2
 - y^4
 + 3y^6
 + y^8
 - 2y^{10}
 + 2y^{12}
 - y^{14}
 - 3y^{16}
 + y^{18}
 + y^{20}
 $\\
 \hline
 (6,6)&$
 2y^2
 - 4y^5
 - y^6
 + y^8
 + 4y^9
 - 2y^{10}
 - 4y^{13}
 + 4y^{14}
 + 2y^{16}
 - y^{18}
 - y^{20}
 $\\
\hline
(6,8)&$
2y
 + y^2
 + 2y^3
 - y^5
 + 2y^6
 - 3y^7
 - y^8
 - 2y^9
 - 2y^{10}
 - 2y^{11}
 + 2y^{13}
 - y^{14}
 - 2y^{16}
 + y^{17}
 + 2y^{18}
 + y^{19}
 + y^{20}
$\\
\hline
(6,10)&$-y
 - y^2
 - y^3
 + 2y^5
 - y^6
 + 2y^7
 + 2y^{11}
 + y^{13}
 - y^{14}
 - y^{15}
 - y^{18}$\\
\hline
(7,7)&$
-2
 - y^2
 - 2y^4
 + 5y^6
 + 4y^8
 - 2y^{12}
 - y^{14}
 - 2y^{16}
 + y^{18}
$\\
\hline
(7,9)&$
1
 + y^4
 - 2y^6
 - 2y^8
 + 3y^{12}
 - 2y^{14}
 + y^{16}
$\\
 \hline
 (8,8)&$
 1 + 3y^4+ 2y^5
 + 2y^7
 - 2y^8
 - 2y^{10}
 - 4y^{11}
 - y^{12}
 - 4y^{13}
 - 4y^{14}
 + y^{16}
 + 2y^{17}
 + 2y^{18}
 + 2y^{19}
 + 2y^{20}
 $\\
 \hline
 (8,10)&$
 -2y^4
 - y^5
 - y^6
 - y^7
 + y^8
 + 2y^{10}
 + 4y^{11}
 + 2y^{14}
 - y^{17}
 - y^{18}
 - y^{19}
 - y^{20}
 $\\
 \hline
 (9,9)&$ -y^4 + y^8 + 2y^{10} - 4y^{12} + 2y^{14} + y^{16} - y^{20}$\\
\hline
(10,10)& $y^4  + y^6 - 4y^{11} + y^{16} + y^{18}$\\
\hline
\end{longtable}}
\end{center}
}

\begin{thm}\label{teo-euler-sl4}
The generating series $P(x,y,z)$ is
\begin{equation}\label{eq:sl4-rational-generating-function}
  P(x,y,z)
  =\frac{\displaystyle
     \sum_{(a,b)\in\mathcal{I}}f_{a,b}(y)x^az^b}
    {\displaystyle
     \prod_{i=1}^3(1-x^{2i})(1-z^{2i})
     (1-y^4)(1-y^6)(1-y^{12})}.
\end{equation}
\end{thm}

\begin{proof}
For a block representative $A$, put
\[
  Q_A(t)=\sum_{m\geq0}H_m(A)t^m
        =\frac1{\det(I-tA^{-1})}
\]
and
\[
  F_A(x,y,z)
  =\sum_{\ell_1,\ell_2,\ell_3\geq0}
     H_\lambda(A)x^{\ell_1}y^{\ell_2}z^{\ell_3}.
\]
Expanding the Jacobi--Trudi determinant expresses $F_A$ as the sum of
six series whose coefficients are products of three shifted
coefficients of $Q_A(t)$.  Because every $Q_A(t)$ is a reciprocal
product of cyclotomic polynomials of orders dividing $12$, these
coefficient sequences are affine on residue classes modulo $12$.
Each of the six sums is therefore evaluated by the elementary
identities
\[
  \sum_{n\geq0}u^n=\frac1{1-u},\qquad
  \sum_{n\geq0}nu^n=\frac{u}{(1-u)^2},\qquad
  \sum_{n\geq0}n^2u^n=\frac{u(1+u)}{(1-u)^3}.
\]
Applying this calculation to the thirteen rows in $\mathcal{S}$ and
forming the weighted sum
\[
  P(x,y,z)=\sum_{A\in\mathcal{S}}W(A)F_A(x,y,z)
\]
gives, after cancellation, the common denominator
\[
  D(x,y,z)=
  \prod_{i=1}^3(1-x^{2i})(1-z^{2i})
  (1-y^4)(1-y^6)(1-y^{12}).
\]
Thus $D(x,y,z)P(x,y,z)$ is a polynomial.  The same finite calculation
shows that its $x,z$-support lies in $0\leq a,b\leq10$, that only
$a\equiv b\pmod2$ occurs, and that its coefficient of $x^az^b$ is
the polynomial $f_{a,b}(y)$ displayed in
Table~\ref{tab:sl4-numerator}.  This proves
\eqref{eq:sl4-rational-generating-function}.
\end{proof}

\begin{coro}\label{cor:sl4-quasipolynomial}
The function $(\ell_1,\ell_2,\ell_3)\mapsto \chi_{(\ell_1,\ell_2,\ell_3)}$ is a quasi-polynomial of total degree at most two and period $12$ on $\mathbb{Z}_{\geq0}^3$.
\end{coro}
\begin{proof}
Theorem~\ref{thm:sl4-residue-formulas} shows that, on each residue class modulo $12$, the function is given by one of the polynomials of
degree at most two recorded in Appendix~\ref{app:sl4-residue-tables}. Hence its period divides $12$.

It remains to show that no smaller common period is possible. Setting $x=z=0$ in~\eqref{eq:sl4-rational-generating-function} gives
\[ \sum_{b\geq0}\chi_{(0,b,0)}y^b = \frac{y^2+y^4-4y^9+y^{14}+y^{16}}{(1-y^4)(1-y^6)(1-y^{12})}.\]
Let $\zeta=e^{\pi i/6}$ be a primitive $12$-th root of unity. The first two factors in the denominator do not vanish at $\zeta$, whereas $
\zeta^2+\zeta^4-4\zeta^9+\zeta^{14}+\zeta^{16}=(4+2\sqrt{3})i\neq0.$
Thus the generating function has a pole at a primitive $12$-th root of unity. Consequently, the sequence $b\mapsto\chi_{(0,b,0)}$ cannot have period smaller than $12$. Therefore the minimal common period of the three-variable quasi-polynomial is exactly $12$.
\end{proof}

\section{Euler Characteristics of \texorpdfstring{$\mathrm{GL}_4(\mathbb{Z})$}{GL4(Z)}}\label{se:Euler-GL4}

Recall that $\widetilde{\Gamma}=\mathrm{GL}_4(\mathbb Z).$  For
$\lambda=\ell_1\epsilon_1 +\ell_2(\epsilon_1+\epsilon_2) +\ell_3(\epsilon_1+\epsilon_2+\epsilon_3) +\ell_4(\epsilon_1+\epsilon_2+\epsilon_3+\epsilon_4),$
where \(\ell_1,\ell_2,\ell_3\in\mathbb{Z}_{\geq0}\) and
\(\ell_4\in\mathbb{Z}\), let
\(\mathcal{M}_{(\ell_1,\ell_2,\ell_3,\ell_4)}\) denote the corresponding rational irreducible coefficient system.  In partition coordinates its highest weight is
$(\ell_1+\ell_2+\ell_3+\ell_4, \ell_2+\ell_3+\ell_4, \ell_3+\ell_4,\ell_4),$ and
\begin{equation}\label{eq:gl4-determinant-shift}
 \mathcal{M}_{(\ell_1,\ell_2,\ell_3,\ell_4)}\simeq \mathcal{M}_{(\ell_1,\ell_2,\ell_3,0)} \otimes \det^{\ell_4}.
\end{equation}
Since \(\det(\widetilde{\Gamma})=\{\pm1\}\), the Euler characteristic depends on \(\ell_4\) only modulo 2.  For \(\varepsilon\in\{0,1\}\), write
\begin{equation}\label{eq:gl4-chi-epsilon}
 \chi^\varepsilon_{(\ell_1,\ell_2,\ell_3)}=\chi_h\!\left(\widetilde{\Gamma},\mathcal{M}_{(\ell_1,\ell_2,\ell_3,0)}\otimes\det^\varepsilon\right).
\end{equation}
Thus \(\chi^0\) is the untwisted characteristic and \(\chi^1\) is the determinant-twisted characteristic.  In particular, the latter is
 $\chi^1_{(\ell_1,\ell_2,\ell_3)}=\chi_h\!\left(\widetilde{\Gamma},\mathcal{M}_{(\ell_1,\ell_2,\ell_3,1)}\right),$
with no shift of $\ell_1,\ell_2,\ell_3$. The index-two identity~\eqref{eq:index-two-sl4} becomes
\begin{equation}\label{eq:gl4-sl4-coefficient-relation}
 \chi_{(\ell_1,\ell_2,\ell_3)} =\chi^0_{(\ell_1,\ell_2,\ell_3)}+\chi^1_{(\ell_1,\ell_2,\ell_3)},
\end{equation}
where the left-hand side is the \(\mathrm{SL}_4(\mathbb{Z})\) characteristic defined in Section~\ref{se:Euler-Sl4}.

\subsection{Block types}

For a block representative \(A\), set
\[
 W_{\widetilde{\Gamma}}(A)=\bigl\lvert\operatorname{Res}(f_A)\bigr\rvert
        \chi_{\mathrm{orb}}\bigl(C_{\widetilde{\Gamma}}(A)\bigr).
\]
The compressed formula~\eqref{eq:hnthor} gives
\begin{equation}\label{eq:gl4-general-weighted-sum}
 \chi^\varepsilon_{(\ell_1,\ell_2,\ell_3)}=\sum_A W_{\widetilde{\Gamma}}(A)\det(A)^\varepsilon H_\lambda(A),\qquad \lambda=(\ell_1+\ell_2+\ell_3,\ell_2+\ell_3,\ell_3,0).
\end{equation}
The $18$ block types and their weights are listed in Table~\ref{torsion-gl4}.  These are the block types in Horozov's resultant formula, not a complete list of integral conjugacy classes. The notation agrees with Table~\ref{torsion-sl4-weights}: in particular, the type \(\Phi_6^2\) is denoted by \(P\).  The three additional determinant $-1$ types are denoted by \(Q,R,S\).

\begingroup
\small
\setlength{\tabcolsep}{10pt}
\renewcommand{\arraystretch}{1.15}
\begin{longtable}{@{}c c c c@{}}
\caption{Block types and weights for
\(\mathrm{GL}_4(\mathbb{Z})\).}
\label{torsion-gl4}\\
\hline
Case & Characteristic polynomial & \(\det(A)\) & \(W_{\widetilde{\Gamma}}(A)\)\\
\hline
\endfirsthead
\multicolumn{4}{c}{\tablename~\thetable\ (continued)}\\
\hline
Case & Characteristic polynomial & \(\det(A)\) & \(W_G(A)\)\\
\hline
\endhead
\(A\) & \(\Phi_1^4\) & \(1\) & \(0\)\\
\(B\) & \(\Phi_1^2\Phi_2^2\) & \(1\) & \(1/36\)\\
\(C\) & \(\Phi_1^2\Phi_3\) & \(1\) & \(-1/16\)\\
\(D\) & \(\Phi_1^2\Phi_4\) & \(1\) & \(-1/24\)\\
\(E\) & \(\Phi_1^2\Phi_6\) & \(1\) & \(-1/144\)\\
\(F\) & \(\Phi_2^4\) & \(1\) & \(0\)\\
\(G\) & \(\Phi_2^2\Phi_3\) & \(1\) & \(-1/144\)\\
\(H\) & \(\Phi_2^2\Phi_4\) & \(1\) & \(-1/24\)\\
\(I\) & \(\Phi_2^2\Phi_6\) & \(1\) & \(-1/16\)\\
\(J\) & \(\Phi_3^2\) & \(1\) & \(0\)\\
\(K\) & \(\Phi_3\Phi_4\) & \(1\) & \(1/24\)\\
\(L\) & \(\Phi_3\Phi_6\) & \(1\) & \(1/9\)\\
\(M\) & \(\Phi_4^2\) & \(1\) & \(0\)\\
\(N\) & \(\Phi_4\Phi_6\) & \(1\) & \(1/24\)\\
\(P\) & \(\Phi_6^2\) & \(1\) & \(0\)\\
\(Q\) & \(\Phi_1\Phi_2\Phi_3\) & \(-1\) & \(1/4\)\\
\(R\) & \(\Phi_1\Phi_2\Phi_4\) & \(-1\) & \(1/2\)\\
\(S\) & \(\Phi_1\Phi_2\Phi_6\) & \(-1\) & \(1/4\)\\
\hline
\end{longtable}
\endgroup

For example, substituting the complete-symmetric trace series
\(\det(I-tA^{-1})^{-1}\) into~\eqref{eq:gl4-general-weighted-sum}
gives the two symmetric-power series
\begin{align}
 \sum_{m\geq0}\chi_h\!\left(\widetilde{\Gamma},\operatorname{Sym}^mV\right)t^m &=\frac{1-t^4-t^6}{1-t^4-t^6+t^{10}},\label{gengl4}\\
 \sum_{m\geq0}\chi_h\!\left(\widetilde{\Gamma},\operatorname{Sym}^mV\otimes\det\right)t^m &=-\frac{1+t^2-t^6}{1-t^4-t^6+t^{10}}.\label{gengl4-odd}
\end{align}
Their sum is precisely the \(\mathrm{SL}_4(\mathbb{Z})\) series in~\eqref{gensl4}, as required by~\eqref{eq:index-two-sl4}.

\subsection{Vanishing and residue-class formulas}

We use the residue notation
\begin{equation}\label{eq:gl4-residue-notation}
 \ell_i=12m_i+r_i,
 \qquad m_i\in\mathbb{Z}_{\geq0},\quad 0\leq r_i<12,
 \qquad (i=1,2,3),
\end{equation}
and record \(\ell_4\) only through
\(\varepsilon\equiv\ell_4\pmod2\).

\begin{prop}[Parity vanishing]\label{prop:gl4-parity-vanishing}
For either \(\varepsilon\in\{0,1\}\), one has $\chi^\varepsilon_{(\ell_1,\ell_2,\ell_3)}=0$ whenever
\[ \ell_1\not\equiv\ell_3\pmod2 \qquad\text{or}\qquad \ell_1\equiv\ell_2\equiv\ell_3\equiv1\pmod2. \]
Thus a nonzero value is possible only for the even--arbitrary--even or the odd--even--odd parity pattern.
\end{prop}

\begin{proof}
The action of the central element \(-I_4\) on the coefficient system is multiplication by $(-1)^{\ell_1+\ell_3}.$ The determinant twist does not change this action.  If \(\ell_1\not\equiv\ell_3\pmod2\), the same Lyndon--Hochschild--Serre argument used in Proposition~\ref{prop:sl4-parity-vanishing} gives the first
vanishing assertion.

Suppose that \(\ell_1,\ell_2,\ell_3\) are all odd.  The contribution of the determinant-one terms in Table~\ref{torsion-gl4} is one half of the corresponding contribution of $\mathrm{SL}_4(\Z)$, and it vanishes by the trace identity~\eqref{eq:all-odd-trace-identity}. For each of the determinant $-1$ terms \(Q,R,S\), direct substitution of its complete-symmetric trace series into the Jacobi--Trudi determinant~\eqref{eqn-Jacobi-Trudi} gives \(H_\lambda(A)=0\). Multiplication by \(\det(A)^\varepsilon\) therefore does not alter the conclusion.
\end{proof}

\begin{thm}[Residue-class formulas]\label{thm:gl4-residue-formulas}
Fix \(\varepsilon\in\{0,1\}\).  For every residue triple allowed by Proposition~\ref{prop:gl4-parity-vanishing}, there is a polynomial
\[
 Q^\varepsilon_{r_1,r_2,r_3}(m_1,m_2,m_3)\in\mathbb{Q}[m_1,m_2,m_3]
\]
of total degree at most two such that
\[
 \chi^\varepsilon_{(\ell_1,\ell_2,\ell_3)}=Q^\varepsilon_{r_1,r_2,r_3}(m_1,m_2,m_3)
\]
under~\eqref{eq:gl4-residue-notation}.  For each determinant parity
there are \(648\) potentially nonzero residue triples, hence
\(1{,}296\) formulas in total.  They are organized into four parity
families and twenty-four \(r_3\)-cases in
Appendix~\ref{app:gl4-residue-tables}.
\end{thm}

\begin{proof}
Combine the weights in Table~\ref{torsion-gl4} with
\eqref{eq:gl4-general-weighted-sum} and the Jacobi--Trudi
formula~\eqref{eqn-Jacobi-Trudi}.  For each  row, the
sequence \(H_m(A)\) is affine on residue classes modulo 12.
After substituting~\eqref{eq:gl4-residue-notation}, each of the six
terms in the Jacobi--Trudi determinant is therefore polynomial on a
fixed residue triple.  The finite simplification gives the
degree-at-most-two formulas recorded in the appendix.  There are
\(432\) even--arbitrary--even and \(216\) odd--even--odd triples for
each value of \(\varepsilon\), which proves the stated count.
\end{proof}

\subsection{Rational generating functions}

Define
\begin{align*}
 P_1(x,y,z)
 =\sum_{a,b,c\geq0}
 \chi^0_{(2a,b,2c)}x^{2a}y^bz^{2c}
 +\sum_{a,b,c\geq0}
 \chi^0_{(2a+1,2b,2c+1)}x^{2a+1}y^{2b}z^{2c+1},\\
 P_2(x,y,z)
 =\sum_{a,b,c\geq0}
 \chi^1_{(2a,b,2c)}x^{2a}y^bz^{2c}
 +\sum_{a,b,c\geq0}
 \chi^1_{(2a+1,2b,2c+1)}x^{2a+1}y^{2b}z^{2c+1}.
\end{align*}
Let
\[
 \mathcal{I}
 =\{(a,b)\in\{0,1,\ldots,10\}^2:a\equiv b\pmod2\}.
\]
The \(61\) polynomials \(g_{a,b}(y)\), indexed by
\((a,b)\in\mathcal{I}\), are recorded in
Table~\ref{tab:gl4-numerator}.

\begin{thm}\label{teo-gl4-euler}
The untwisted generating series is
\begin{equation}\label{eq:gl4-rational-generating-function}
 P_1(x,y,z)
 =\frac{\displaystyle
   \sum_{(a,b)\in\mathcal{I}}g_{a,b}(y)x^az^b}
  {\displaystyle
   \prod_{i=1}^3(1-x^{2i})(1-z^{2i})
   (1-y^4)(1-y^6)(1-y^{12})}.
\end{equation}
Moreover,
\begin{equation}\label{eq:gl4-sl4-generating-relation}
 P(x,y,z)=P_1(x,y,z)+P_2(x,y,z),
\end{equation}
where \(P(x,y,z)\) is the \(\mathrm{SL}_4(\mathbb{Z})\) series in
Theorem~\ref{teo-euler-sl4}.
\end{thm}

\begin{proof}
For a block representative \(A\), let us write 
\[
 Q_A(t)=\sum_{m\geq0}H_m(A)t^m
       =\frac1{\det(I-tA^{-1})}.
\]
Expanding the Jacobi--Trudi determinant expresses the corresponding
three-variable trace series as a sum of six products of shifted
coefficient sequences of \(Q_A(t)\).  Since each \(Q_A(t)\) is a
reciprocal product of cyclotomic polynomials of orders dividing
12, those coefficient sequences are affine on residue classes
modulo 12.  Summing them with the elementary series for
\(1,n,n^2\), and then taking the weighted sum in
\eqref{eq:gl4-general-weighted-sum}, gives the common denominator
\[
 D(x,y,z)=
 \prod_{i=1}^3(1-x^{2i})(1-z^{2i})
 (1-y^4)(1-y^6)(1-y^{12}).
\]
The same finite calculation shows that
\(D(x,y,z)P_1(x,y,z)\) has support
\(0\leq a,b\leq10\) with \(a\equiv b\pmod2\), and its coefficient
of \(x^az^b\) is the polynomial \(g_{a,b}(y)\) in
Table~\ref{tab:gl4-numerator}.  This proves
\eqref{eq:gl4-rational-generating-function}.  Finally,
\eqref{eq:gl4-sl4-generating-relation} follows coefficientwise from
the index-two identity~\eqref{eq:gl4-sl4-coefficient-relation}.
\end{proof}

As a check, set
\[
 D_x=\prod_{i=1}^3(1-x^{2i}).
\]
Then \(y=z=0\) gives
\begin{align*}
 D_xP_1(x,0,0)
 &=g_{0,0}(0)+g_{2,0}(0)x^2+g_{4,0}(0)x^4 + g_{6,0}(0)x^6+g_{8,0}(0)x^8+g_{10,0}(0)x^{10}\\
 &=1-x^2-x^4+x^8,\\
 P_1(x,0,0)
 &=\frac{1-x^4-x^6}{1-x^4-x^6+x^{10}},
\end{align*}
which recovers~\eqref{gengl4}.  

\begin{coro}\label{cor:gl4-quasipolynomial}
For each fixed determinant parity \(\varepsilon\), the function $(\ell_1,\ell_2,\ell_3)\longmapsto\chi^\varepsilon_{(\ell_1,\ell_2,\ell_3)}$ is a quasi-polynomial of total degree at most two and period \(12\) on \(\mathbb{Z}_{\geq0}^3\).
\end{coro}
\begin{proof}
Theorem~\ref{thm:gl4-residue-formulas} shows that both functions are quasi-polynomials of total degree at most two with period dividing $12$. As in the proof of Corollary~\ref{cor:sl4-quasipolynomial}, it remains to exclude a smaller period. Let $\zeta=e^{\pi i/6}$. From the numerator polynomials of the generating functions, we have
\[
g_{0,0}(\zeta)=(2+\sqrt3)i\neq0 \quad \text{and} \quad \bigl(f_{0,0}-g_{0,0}\bigr)(\zeta) =(2+\sqrt3)i\neq0.
\]
Thus both $P_1(0,y,0)$ and $P_2(0,y,0)$ have a pole at a primitive
twelfth root of unity. Therefore, for each determinant parity, the
minimal common period is $12$.
\end{proof}

{ \begin{center}
\tiny\renewcommand{\arraystretch}{2.4}
{\begin{longtable}{@{}c p{0.84\textwidth}@{}}
\caption{Numerator polynomials for the
$\mathrm{GL}_4(\mathbb{Z})$ generating function.}
\label{tab:gl4-numerator}\\
\hline
$(a,b)$&$g_{a,b}(y)$\\
\hline
\endfirsthead
\multicolumn{2}{c}{\tablename~\thetable\ (continued)}\\
\hline
$(a,b)$&$g_{a,b}(y)$\\
\hline
\endhead
(0,0)&$1
 + y^2
 - y^6
 - 2\*y^9
 - y^{12}
 + y^{16}
 + y^{18} $\\
\hline
(0,2)&$-1
 - y^2
 + 2\*y^6
 - y^7
 + 2\*y^9
 + y^{10}
 + y^{12}
 - y^{13}
 - y^{15}
 - y^{16}
 - y^{18}
 + y^{19}$\\
\hline 
(0,4)& $
-1
 - y^2
 + y^4
 - y^5
 + y^7
 + y^9
 + y^{12}
 + y^{13}
 + y^{15}
 - y^{16}
 - y^{18}
 - y^{19}
$
\\
\hline
(0,6)&$y^2
 - y^3
 - y^4
 + y^5
 - y^6
 + y^7
 - y^{10}
 + y^{13}
 + y^{15}
 - y^{19}$\\
\hline
(0,8)&$1
 - y
 + y^3
 + y^5
 - y^9
 - 2\*y^{10}
 - y^{12}
 + y^{16}
 + y^{18}$ \\
\hline
(0,10)& $y
 - y^5
 - y^7
 + 2\*y^{10}
 - y^{13}
 - y^{15}
 + y^{19}$\\
\hline
(1,1)& $-y^8
 + y^{10}
 + y^{18}
 - y^{20}$\\
\hline
(1,3)& $y^8
 - y^{12}
 - y^{14}
 + y^{16}
 - y^{18}
 + y^{20}$\\
\hline
(1,5)& $-y^{10}
 + y^{12}
 + 2\*y^{14}
 - 2\*y^{16}
 - y^{18}
 + y^{20}$\\
\hline
(1,7)& $-1
 + y^4
 + y^6
 - 2\*y^{10}
 + 2\*y^{12}
 - y^{14}$\\
\hline
(1,9)& $1
 - y^4
 - y^6
 + 2\*y^{10}
 - 2\*y^{12}
 + y^{16}
 + y^{18}
 - y^{20}$\\
\hline
(2,0)& $-1
 - y
 - y^2
 - y^3
 + 2\*y^6
 + y^7
 + 2\*y^9
 + y^{10}
 + y^{12}
 - y^{16}
 - y^{18}
 - y^{19}$\\
\hline
(2,2)& $1
 + y
 + y^2
 + y^3
 + y^4
 - 3\*y^6
 - 2\*y^7
 - 2\*y^9
 - y^{10}
 - y^{12}
 + y^{13}
 + y^{15}
 + y^{16}
 + y^{18}$\\
\hline
(2,4)& $1
 + y
 + 2\*y^2
 + y^3
 - 2\*y^4
 - y^6
 - y^9
 - y^{10}
 - y^{11}
 - y^{12}
 - 2\*y^{13}
 - y^{15}
 + y^{16}
 + y^{17}
 + y^{18}
 + 2\*y^{19}$\\
\hline
(2,6)& $-2\*y^2
 + y^4
 + 2\*y^6
 - y^{10}
 + y^{11}
 - y^{15}
 - y^{17}
 + y^{19}$\\
\hline
(2,8)&$
-1
 - y
 - y^3
 + y^7
 - y^8
 + y^9
 + 4\*y^{10}
 + y^{11}
 + y^{12}
 + y^{13}
 - y^{14}
 - 2\*y^{16}
 - y^{17}
 - y^{18}
 - y^{19}
 + y^{20}
$\\
\hline
(2,10)&$
y^8
 - 2\*y^{10}
 - y^{11}
 + y^{14}
 + y^{15}
 + y^{16}
 + y^{17}
 - y^{19}
 - y^{20}
$\\
\hline
(3,1)& $y^8
 - y^{12}
 - y^{14}
 + y^{16}
 - y^{18}
 + y^{20}$\\
\hline
(3,3)& $-y^6
 - y^8
 + 2\*y^{12}
 + 3\*y^{14}
 - 2\*y^{16}
 - y^{20}$\\
\hline
(3,5)& $-1
 + y^4
 + 2\*y^6
 - 2\*y^{10}
 - 2\*y^{14}
 + y^{16}
 + 2\*y^{18}
 - y^{20}$\\
\hline
(3,7)&$
2
 - 2\*y^4
 - 2\*y^6
 + 4\*y^{10}
 - 2\*y^{12}
 - y^{14}
 + y^{18}
$\\
\hline
(3,9)&$
-1
 + y^4
 + y^6
 - 2\*y^{10}
 + y^{12}
 + y^{14}
 - 2\*y^{18}
 + y^{20}
$\\
\hline
(4,0)& $-1
 - y^2
 + y^4
 + y^9
 + y^{12}
 + y^{13}
 + y^{15}
 - y^{16}
 - y^{17}
 - y^{18}$\\
\hline
(4,2)& $1
 + 2\*y^2
 - 2\*y^4
 - y^6
 + 2\*y^7
 - y^9
 - y^{10}
 - y^{11}
 - y^{12}
 - y^{13}
 + y^{16}
 + y^{17}
 + y^{18}$\\
 \hline
(4,4)& $1
 - y^2
 + 2\*y^6
 - 2\*y^7
 - y^{10}
 + 2\*y^{11}
 - y^{12}
 - 2\*y^{15}
 + y^{16}
 + y^{18}$\\
\hline
(4,6)& $y^3
 + y^4
 - y^6
 - y^7
 - y^8
 - y^9
 + 4\*y^{10}
 - y^{11}
 - y^{13}
 - y^{14}
 - y^{16}
 + y^{17}
 + y^{20}$\\
 \hline
 (4,8)&$
 -1
 + y
 - y^3
 - y^5
 - y^6
 + 2\*y^8
 + y^9
 - y^{10}
 + y^{12}
 + 2\*y^{14}
 + y^{16}
 - y^{18}
 - 2\*y^{20}
 $\\
 \hline
(4,10)&$
-y
 + y^5
 + y^6
 + y^7
 - y^8
 - y^{10}
 + y^{13}
 - y^{14}
 + y^{15}
 - y^{16}
 - y^{17}
 + y^{20}
$\\ 
\hline
(5,1)& $-y^{10}
 + y^{12}
 + 2\*y^{14}
 - 2\*y^{16}
 - y^{18}
 + y^{20}$\\
 \hline
(5,3)& $-1
 + y^4
 + 2\*y^6
 - 2\*y^{10}
 - 2\*y^{14}
 + y^{16}
 + 2\*y^{18}
 - y^{20}$\\
\hline
(5,5)& $2
 - 3\*y^4
 - 3\*y^6
 + y^8
 + 6\*y^{10}
 - 2\*y^{12}
 - 3\*y^{14}
 + 3\*y^{16}
 - y^{20}$\\
 \hline
 (5,7)&$
 y^4
 - y^8
 - 2\*y^{10}
 - y^{12}
 + 4\*y^{14}
 + y^{16}
 - 2\*y^{18}
 $\\
 \hline
 (5,9)&$
 -1
 + y^4
 + y^6
 - y^{10}
 + 2\*y^{12}
 - y^{14}
 - 3\*y^{16}
 + y^{18}
 + y^{20}
 $\\
 \hline
 (6,0)&$
 y
 + y^2
 + y^3
 - y^4
 - y^6
 - y^7
 - y^{10}
 - y^{15}
 + y^{17}
 + y^{19}
 $\\
 \hline
 (6,2)&$
 -y
 - 2\*y^2
 - y^3
 + y^4
 + 2\*y^6
 + 2\*y^7
 - y^{10}
 + y^{11}
 + y^{13}
 - y^{17}
 - y^{19}
 $\\
 \hline
 (6,4)&$
 -y
 - y^3
 + y^4
 + 2\*y^5
 - y^6
 - y^8
 - y^9
 + 4\*y^{10}
 - y^{11}
 - y^{14}
 + 2\*y^{15}
 - y^{16}
 - y^{17}
 - y^{19}
 + y^{20}
 $\\
 \hline
 (6,6)&$
 y^2
 - y^4
 - 2\*y^5
 - y^6
 + 2\*y^8
 + 2\*y^9
 - y^{10}
 - 2\*y^{13}
 + 2\*y^{14}
 + 2\*y^{16}
 - 2\*y^{20}
 $\\
\hline
(6,8)&$
2\*y
 + y^3
 - y^4
 - y^5
 + 2\*y^6
 - 2\*y^7
 - y^9
 - y^{10}
 - y^{11}
 + y^{17}
 + y^{19}
$\\
\hline
(6,10)&$-y
 + y^4
 + y^5
 - y^6
 + y^7
 - y^8
 + y^{11}
 + y^{13}
 - y^{14}
 - y^{15}
 - y^{16}
 + y^{20}$\\
 \hline
(7,1)&$
-1
 + y^4
 + y^6
 - 2\*y^{10}
 + 2\*y^{12}
 - y^{14}
$\\
\hline
(7,3)&$
2
 - 2\*y^4
 - 2\*y^6
 + 4\*y^{10}
 - 2\*y^{12}
 - y^{14}
 + y^{18}
$\\
\hline
(7,5)&$
y^4
 - y^8
 - 2\*y^{10}
 - y^{12}
 + 4\*y^{14}
 + y^{16}
 - 2\*y^{18}
$\\
\hline
(7,7)&$
-1
 - y^2
 + 2\*y^6
 + 2\*y^8
 - y^{12}
 - 2\*y^{16}
 + y^{18}
$\\
\hline
(7,9)&$
y^2
 - y^6
 - y^8
 + 2\*y^{12}
 - 2\*y^{14}
 + y^{16}
$\\
\hline
 (8,0)&$
 1
 + y
 + y^3
 - y^7
 - y^9
 - 2\*y^{10}
 - y^{12}
 - 2\*y^{13}
 + y^{16}
 + y^{17}
 + y^{18}
 + y^{19}
 $\\
 \hline
 (8,2)&$
 -1
 - y
 - y^3
 + y^7
 - y^8
 + y^9
 + 4\*y^{10}
 + y^{11}
 + y^{12}
 + y^{13}
 - y^{14}
 - 2\*y^{16}
 - y^{17}
 - y^{18}
 - y^{19}
 + y^{20}
 $\\
 \hline
 (8,4)&$
 -1
 - y
 - y^3
 - y^6
 + y^7
 + 2\*y^8
 + y^9
 - y^{10}
 + y^{12}
 + 2\*y^{13}
 + 2\*y^{14}
 + y^{16}
 - y^{17}
 - y^{18}
 - y^{19}
 - 2\*y^{20}
 $\\
 \hline
 (8,6)&$
 y^3
 - y^4
 + 2\*y^6
 - y^7
 - y^9
 - y^{10}
 - y^{11}
 + 2\*y^{13}
 $\\
 \hline
 (8,8)&$
 1
 - y^2
 + 2\*y^4
 + y^5
 + y^7
 - y^8
 - y^{10}
 - 2\*y^{11}
 - y^{12}
 - 2\*y^{13}
 - y^{14}
 + y^{17}
 + y^{18}
 + y^{19}
 + y^{20}
 $\\
 \hline
 (8,10)&$
 y
 + y^2
 - y^4
 - y^5
 - y^6
 - y^7
 + y^{10}
 + 2\*y^{11}
 - y^{13}
 $\\
 \hline
 (9,1)&$
 1
 - y^4
 - y^6
 + 2\*y^{10}
 - 2\*y^{12}
 + y^{16}
 + y^{18}
 - y^{20}
 $\\
 \hline
 (9,3)&$
 -1
 + y^4
 + y^6
 - 2\*y^{10}
 + y^{12}
 + y^{14}
 - 2\*y^{18}
 + y^{20}
 $\\
 \hline
 (9,5)&$
 -1
 + y^4
 + y^6
 - y^{10}
 + 2\*y^{12}
 - y^{14}
 - 3\*y^{16}
 + y^{18}
 + y^{20}
 $\\
 \hline
 (9,7)&$
 y^2
 - y^6
 - y^8
 + 2\*y^{12}
 - 2\*y^{14}
 + y^{16}
 $\\
 \hline
 (9,9)&$
 1
 - y^2
 - y^4
 + y^8
 + y^{10}
 - 3\*y^{12}
 + 2\*y^{14}
 + y^{16}
 - y^{20}
 $\\
 \hline
(10,0)&$
-y
 - y^3
 + y^7
 + 2\*y^{10}
 + y^{13}
 - y^{17}
 - y^{19}
$\\
\hline
(10,2)&$
y
 + y^3
 - 2\*y^7
 + y^8
 - 2\*y^{10}
 - y^{11}
 - y^{13}
 + y^{14}
 + y^{16}
 + y^{17}
 + y^{19}
 - y^{20}
$\\
\hline
(10,4)&$
y
 + y^3
 - y^5
 + y^6
 - y^8
 - y^{10}
 - y^{13}
 - y^{14}
 - y^{16}
 + y^{17}
 + y^{19}
 + y^{20}
$\\
\hline
(10,6)&$
-y^3
 + y^4
 + y^5
 - y^6
 + y^7
 - y^8
 + y^{11}
 - y^{14}
 - y^{16}
 + y^{20}
$\\
\hline
(10,8)&$
-y
 + y^2
 - y^4
 - y^6
 + y^{10}
 + 2\*y^{11}
 + y^{13}
 - y^{17}
 - y^{19}
$\\
\hline
(10,10)&$
-y^2
 + y^6
 + y^8
 - 2\*y^{11}
 + y^{14}
 + y^{16}
 - y^{20}
$\\
\hline
\end{longtable}}
\end{center}
}

\section{Cohomological consequences}\label{se:cohomological-consequences}

The preceding sections give the homological Euler characteristic for
every irreducible highest-weight coefficient system. These formulas do
not determine the individual cohomology groups, but they still give
useful information about their nonvanishing and their dimensions. We
first recall the basic cohomological setting for
\(\mathrm{SL}_4(\mathbb Z)\). We then consider arbitrary highest
weights and finally specialize to symmetric powers, for which one of
the two \(\mathrm{GL}_4(\mathbb Z)\)-parts is already known degree by
degree.

All cohomology groups and dimensions in this section are taken over
\(\mathbb Q\). Except where an earlier theorem is cited, the arguments
below do not determine whether the cohomology classes are inner,
cuspidal, or Eisenstein.

\subsection{Cohomology of\texorpdfstring{$\mathrm{SL}_4(\mathbb Z)$}{SL4(Z)}}\label{subse:sl4-cohomological-setting}

We briefly recall the cohomological framework in which the Euler characteristics computed above are situated. The cohomology of an arithmetic group connects the geometry of the associated locally symmetric space with the representation theory of the ambient algebraic group and the theory of automorphic forms. A basic problem is to describe this cohomology degree by degree and, in particular, to distinguish its inner, cuspidal, boundary, and Eisenstein parts. We use the same Borel--Serre framework as in our earlier calculations for \(\mathrm{SL}_3(\mathbb Z)\), \(G_2(\mathbb Z)\), and \(\mathrm{Sp}_4(\mathbb Z)\); see~\cite{BHHMG2021,BG2022,BHMG2023}.

Let
\[
 \mathbf G=\mathrm{SL}_4,
 \qquad
 \Gamma=\mathbf G(\mathbb Z)=\mathrm{SL}_4(\mathbb Z),
 \qquad
 K_\infty=\mathrm{SO}(4),
\]
and let
\[
 X=\mathbf G(\mathbb R)/K_\infty,
 \qquad
 S_\Gamma=\Gamma\backslash X.
\]
The symmetric space \(X\) is contractible. Since \(\Gamma\) contains torsion, the quotient \(S_\Gamma\) is naturally regarded as an arithmetic orbifold. Let \(\mathcal M_\lambda\) be a finite-dimensional rational representation of \(\mathbf G\), and let \(\mathcal L_\lambda\) denote the corresponding rational local system on \(S_\Gamma\). Interpreting the cohomology of \(S_\Gamma\) in the orbifold sense, one has a canonical identification
\begin{equation}\label{eq:sl4-group-geometric-cohomology}
 H^q(\Gamma,\mathcal M_\lambda) \simeq H^q(S_\Gamma,\mathcal L_\lambda);
\end{equation}
see, for example, \cite[Chapter~VII]{BoWa}. Thus the group cohomology appearing in this paper may be studied through the geometry of the locally symmetric orbifold \(S_\Gamma\).

Let \(\overline{S}_\Gamma^{\mathrm{BS}}\) be the Borel--Serre compactification of \(S_\Gamma\)~\cite{BorelSerre73}. It is a compact orbifold with corners, and the inclusion
\[
 i:S_\Gamma\hookrightarrow\overline{S}_\Gamma^{\mathrm{BS}}
\]
is a homotopy equivalence. We denote the direct-image extension
\(i_*\mathcal L_\lambda\) again by \(\mathcal L_\lambda\). Then
\[
 H^q(\Gamma,\mathcal M_\lambda) \simeq H^q(\overline{S}_\Gamma^{\mathrm{BS}},\mathcal L_\lambda).
\]
Define
\[
 \partial\overline{S}_\Gamma^{\mathrm{BS}}
 =
 \overline{S}_\Gamma^{\mathrm{BS}}\setminus S_\Gamma.
\]
The long exact sequence associated with the Borel--Serre boundary is
\begin{equation}\label{eq:sl4-borel-serre-long-exact-sequence}
 \begin{aligned}
 \cdots&\longrightarrow
 H_c^q(S_\Gamma,\mathcal L_\lambda)
 \longrightarrow
 H^q(S_\Gamma,\mathcal L_\lambda)
 \xrightarrow{\,r^q\,}
 H^q(\partial\overline{S}_\Gamma^{\mathrm{BS}},
     \mathcal L_\lambda)\\
 &\longrightarrow
 H_c^{q+1}(S_\Gamma,\mathcal L_\lambda)
 \longrightarrow
 H^{q+1}(S_\Gamma,\mathcal L_\lambda)
 \longrightarrow\cdots .
 \end{aligned}
\end{equation}
The group on the right of the restriction map is called the
\emph{boundary cohomology}. The \emph{inner cohomology} is
\[
 H_!^q(S_\Gamma,\mathcal L_\lambda)
 :=
 \operatorname{im}\!\left(
 H_c^q(S_\Gamma,\mathcal L_\lambda)
 \longrightarrow H^q(S_\Gamma,\mathcal L_\lambda)
 \right)
 =
 \ker(r^q),
\]
whereas the \emph{Eisenstein cohomology} is defined here by
\[
 H_{\mathrm{Eis}}^q(S_\Gamma,\mathcal L_\lambda)
 :=
 \operatorname{im}(r^q).
\]
Consequently, there is a short exact sequence
\begin{equation}\label{eq:sl4-inner-total-eisenstein}
 0\longrightarrow
 H_!^q(S_\Gamma,\mathcal L_\lambda)
 \longrightarrow
 H^q(\Gamma,\mathcal M_\lambda)
 \longrightarrow
 H_{\mathrm{Eis}}^q(S_\Gamma,\mathcal L_\lambda)
 \longrightarrow0.
\end{equation}
The cuspidal cohomology, defined in terms of cuspidal automorphic
forms, maps into the inner cohomology. In general, determining the
image of \(r^q\), or deciding when inner and cuspidal cohomology agree,
requires additional automorphic input; see
\cite{Harder87,Harder91}. No such identification will be assumed in
this section.

The Borel--Serre boundary is assembled from faces indexed by the proper
rational parabolic subgroups of \(\mathbf G\). After fixing the
diagonal maximal \(\mathbb Q\)-split torus and the upper-triangular
Borel subgroup, its cohomology is computed by the spectral sequence
\begin{equation}\label{eq:sl4-boundary-spectral-sequence}
 E_1^{p,q}
 =
 \bigoplus_{\substack{P\text{ proper standard}\\
                      \operatorname{prk}(P)=p+1}}
 H^q(\partial_P,\mathcal L_\lambda)
 \quad\Longrightarrow\quad
 H^{p+q}(\partial\overline{S}_\Gamma^{\mathrm{BS}},
         \mathcal L_\lambda),
\end{equation}
where \(\partial_P\) is the boundary face attached to \(P\) and
\(\operatorname{prk}(P)\) is its parabolic rank. If
\(P=L_PN_P\) is a Levi decomposition, the cohomology of
\(\partial_P\) is reduced to the cohomology of an arithmetic subgroup
of \(L_P\), with coefficients coming from the cohomology of the
unipotent radical \(N_P\). The representation-theoretic input is
provided by Kostant's theorem:
\begin{equation}\label{eq:sl4-kostant-boundary-input}
 H^j(\mathfrak n_P,\mathcal M_\lambda)
 \simeq
 \bigoplus_{w\in W^P: \ell(w)=j}
 \mathcal M_{w\cdot\lambda}^{L_P},
 \qquad
 w\cdot\lambda=w(\lambda+\rho)-\rho;
\end{equation}
see~\cite{Kostant61}. Since
\(\rank_{\mathbb Q}\mathbf G=3\), the spectral sequence
\eqref{eq:sl4-boundary-spectral-sequence} has three columns,
\(p=0,1,2\). Thus, unlike the \(\mathbb Q\)-rank-two groups treated
in the three papers cited above, the boundary cohomology of
\(\mathrm{SL}_4(\mathbb Z)\) involves a genuine three-column spectral
sequence.

We finally recall the cohomological range. Since \(\Gamma\) has
torsion, the relevant finite invariant is its \emph{virtual
cohomological dimension}. If \(\Gamma'\subseteq\Gamma\) is any
torsion-free subgroup of finite index, then
\[
 \vcd(\Gamma):=\cd(\Gamma');
\]
this is independent of the choice of \(\Gamma'\). The theorem of
Borel and Serre gives
\begin{equation}\label{eq:sl4-vcd}
 \begin{aligned}
 \vcd(\Gamma)
  &=\dim X-\rank_{\mathbb Q}\mathbf G \\
  &=\dim\mathbf G-\dim K_\infty-\rank_{\mathbb Q}\mathbf G \\
  &=15-6-3=6.
 \end{aligned}
\end{equation}
Because \(\mathbf G=\mathrm{SL}_4\) is \(\mathbb Q\)-split, the same
number is
\[
 \dim N
 =
 \lvert\Phi^+\rvert
 =
 \ell(w_0)
 =
 6,
\]
where \(N\) is the unipotent radical of a Borel subgroup and \(w_0\)
is the longest element of the Weyl group
\(W\simeq\mathfrak S_4\). It follows that, for every finite-dimensional
rational coefficient system \(\mathcal M_\lambda\),
\begin{equation}\label{eq:sl4-vanishing-above-vcd}
 H^q(\Gamma,\mathcal M_\lambda)=0
 \qquad(q>6).
\end{equation}

Accordingly, the homological Euler characteristic considered in this
paper is the finite alternating sum
\[
 \chi_h(\Gamma,\mathcal M_\lambda)
 =
 \sum_{q=0}^{6}(-1)^q
 \dim H^q(\Gamma,\mathcal M_\lambda).
\]
It gives an exact numerical constraint on the full cohomology, but it
does not by itself determine the individual degrees or separate the
inner, cuspidal, and Eisenstein contributions. The remainder of this
section extracts the strongest degreewise consequences that follow
from the Euler-characteristic formulas without carrying out the full
boundary and Eisenstein-cohomology calculation.

\subsection{The two \(\mathrm{GL}_4\)-extensions}\label{subse:degreewise-index-two}
Let $\lambda$ be the dominant integral highest weight of $\mathrm{SL}_4$ with fundamental-weight coordinates $(\ell_1,\ell_2,\ell_3)$ with $\ell_1,\ell_2,\ell_3\in\mathbb Z_{\geq0},$ and let \(\mathcal M_\lambda\) be the corresponding irreducible coefficient system.  Its polynomial \(\mathrm{GL}_4\)-extension has partition coordinates
\[
 \widetilde\lambda =\bigl(\ell_1+\ell_2+\ell_3, \ell_2+\ell_3,\ell_3,0\bigr).
\]
We denote this extension by \(\widetilde{\mathcal M}_\lambda\).  The other extension of the same \(\Gamma\)-module is \(\widetilde{\mathcal M}_\lambda\otimes\det\).

The Euler-characteristic identity~\eqref{eq:gl4-sl4-coefficient-relation} has the following degreewise form.
\begin{prop}[Degreewise index-two decomposition]\label{prop:degreewise-index-two}
For every \(q\geq0\), there is an isomorphism of \(\mathbb Q\)-vector spaces
\begin{equation}\label{eq:degreewise-index-two}
 H^q(\Gamma,\mathcal M_\lambda) \simeq H^q(\widetilde{\Gamma},\widetilde{\mathcal M}_\lambda) \oplus H^q(\widetilde{\Gamma},\widetilde{\mathcal M}_\lambda\otimes\det).
\end{equation}
Consequently, if
\begin{align*}
 b_q(\lambda) =\dim H^q(\Gamma,\mathcal M_\lambda),\qquad
 b_q^\varepsilon(\lambda) =\dim H^q\!\left(\widetilde{\Gamma},\widetilde{\mathcal M}_\lambda\otimes\det^\varepsilon \right),\qquad \varepsilon\in\{0,1\},
\end{align*}
then
\begin{equation}\label{eq:betti-index-two}
 b_q(\lambda)=b_q^0(\lambda)+b_q^1(\lambda)
\end{equation}
in every degree.
\end{prop}

\begin{proof}
Since \(\Gamma\) is normal of index two in \(G\),
\[
 \operatorname{Ind}_{\Gamma}^{\widetilde{\Gamma}}\mathbf 1\simeq\mathbf 1\oplus\det.
\]
Shapiro's lemma~\cite[Chapter~III, Prop.~6.2]{Brown94}, followed by the projection formula, gives
\[
 H^q(\Gamma,\mathcal M_\lambda)\simeq H^q\!\left(\widetilde{\Gamma},\widetilde{\mathcal M}_\lambda\otimes\operatorname{Ind}_{\Gamma}^{\widetilde{\Gamma}}\mathbf 1\right),
\]
and the asserted decomposition follows.
\end{proof}

The first degreewise vanishing statement is supplied by the center and
is stronger than a vanishing of Euler characteristics.

\begin{prop}[Central-character vanishing]
\label{prop:central-degreewise-vanishing}
If $\ell_1+\ell_3\equiv1\pmod2,$ then, for every \(q\geq0\) and \(\varepsilon\in\{0,1\}\),
\[
 H^q(\Gamma,\mathcal M_\lambda)=0, \qquad H^q\left(\widetilde{\Gamma},\widetilde{\mathcal M}_\lambda\otimes\det^\varepsilon\right)=0.
\]
\end{prop}

\begin{proof}
The central element \(-I_4\) acts on \(\mathcal M_\lambda\) by
\[
 (-1)^{\ell_1+2\ell_2+3\ell_3} =(-1)^{\ell_1+\ell_3}.
\]
Moreover, \(\det(-I_4)=1\), so the determinant twist does not change
this action.  If \(\ell_1+\ell_3\) is odd, the invariant space for the
central subgroup \(\langle-I_4\rangle\) is zero.  Since this subgroup is
finite and the coefficient field has characteristic zero, its positive
degree cohomology also vanishes.  The Lyndon--Hochschild--Serre spectral
sequence therefore has zero \(E_2\)-page in each of the three cases.
\end{proof}

\begin{rmk}\label{rmk:all-odd-distinction}
Propositions~\ref{prop:sl4-parity-vanishing} and~\ref{prop:gl4-parity-vanishing} also show that the Euler characteristics vanish when $\ell_1\equiv\ell_2\equiv\ell_3\equiv1\pmod2$.  In that case \(-I_4\) acts trivially.  Thus the argument above does not give
degreewise vanishing, and no such conclusion is intended: an Euler characteristic equal to zero may conceal nonzero cohomology in degrees of opposite parity.
\end{rmk}

\subsection{Parity constraints on the Betti numbers}\label{subse:general-betti-constraints}

The virtual cohomological dimension of \(\Gamma\) and \(\widetilde{\Gamma}\) is six; hence their rational cohomology vanishes above degree six~\cite{BorelSerre73}.  We shall also use the standard low-degree vanishing for nontrivial irreducible algebraic coefficient systems:
\begin{equation}\label{eq:low-degree-vanishing-section-seven}
 H^q(\Gamma,\mathcal M_\lambda)=0,\qquad H^q\left(\widetilde{\Gamma},\widetilde{\mathcal M}_\lambda\otimes\det^\varepsilon\right)=0 \quad(q=0,1,2),
\end{equation}  
for $\lambda\neq(0,0,0)$. Indeed, since $\Gamma=\mathrm{SL}_4(\mathbb{Z})$ is Zariski dense in $\mathrm{SL}_4$, the nontrivial irreducible representation
$\mathcal{M}_\lambda$ has no $\Gamma$-invariant vectors:
\[
\bigl(\mathcal{M}_\lambda\otimes_{\mathbb{Q}}\mathbb{C}\bigr)^\Gamma=0.
\]
Moreover, $r_{\mathrm{SL}_4(\mathbb{R})}=3$. Hence Li--Sun~\cite[Thm~1.8 and Table~2]{LiSun19}, after extension of scalars to $\mathbb{C}$, gives
\[ H^q\left(\Gamma,\mathcal{M}_\lambda\otimes_{\mathbb{Q}}\mathbb{C}\right)=0 \qquad (q<3). \]
It follows that $H^q(\Gamma,\mathcal{M}_\lambda)=0$ for $q=0,1,2$. Finally, the degreewise index-two decomposition of Proposition~12 implies that both summands
\[
H^q\left(\widetilde{\Gamma},\widetilde{\mathcal{M}}_\lambda\right) \quad \text{and} \quad H^q\left(\widetilde{\Gamma}, \widetilde{\mathcal{M}}_\lambda\otimes\det\right)
\]
vanish in the same degrees. This proves~\eqref{eq:low-degree-vanishing-section-seven}. The trivial coefficient system will be treated separately below.

For a nontrivial \(\lambda\), define
\begin{align*}
 E_\lambda&=b_4(\lambda)+b_6(\lambda),
 &O_\lambda&=b_3(\lambda)+b_5(\lambda),\\
 E_\lambda^\varepsilon
 &=b_4^\varepsilon(\lambda)+b_6^\varepsilon(\lambda),
 &O_\lambda^\varepsilon
 &=b_3^\varepsilon(\lambda)+b_5^\varepsilon(\lambda).
\end{align*}
Recall that
\[
 \chi_\lambda=\chi_h(\Gamma,\mathcal M_\lambda),\qquad \chi_\lambda^\varepsilon =\chi_h\!\left(\widetilde{\Gamma},\widetilde{\mathcal M}_\lambda\otimes\det^\varepsilon \right).
\]

\begin{prop}[Degree-parity identities and bounds]\label{prop:degree-parity-identities}
Let \(\lambda\neq(0,0,0)\).  Then
\begin{equation}\label{eq:sl4-even-odd-difference}
 E_\lambda-O_\lambda=\chi_\lambda
\end{equation}
and, for \(\varepsilon\in\{0,1\}\),
\begin{equation}\label{eq:gl4-even-odd-difference}
 E_\lambda^\varepsilon-O_\lambda^\varepsilon
 =\chi_\lambda^\varepsilon.
\end{equation}
In particular,
\begin{align}
 \chi_\lambda>0
 &\quad\Longrightarrow\quad
 b_4(\lambda)+b_6(\lambda)\geq\chi_\lambda,
 \label{eq:positive-chi-even-bound}\\
 \chi_\lambda<0
 &\quad\Longrightarrow\quad
 b_3(\lambda)+b_5(\lambda)\geq-\chi_\lambda.
 \label{eq:negative-chi-odd-bound}
\end{align}
The analogous assertions hold separately with
\((b_q,\chi_\lambda)\) replaced by
\((b_q^\varepsilon,\chi_\lambda^\varepsilon)\).
\end{prop}

\begin{proof}
By~\eqref{eq:low-degree-vanishing-section-seven} and the virtual
cohomological dimension, the Euler characteristic is
\[
 \chi_\lambda
 =-b_3(\lambda)+b_4(\lambda)-b_5(\lambda)+b_6(\lambda),
\]
which is~\eqref{eq:sl4-even-odd-difference}.  The same argument gives
\eqref{eq:gl4-even-odd-difference}.  The inequalities follow from the
nonnegativity of all Betti numbers.
\end{proof}

The identities also give exact information about the total Betti
number.  Put
\[
 B_\lambda=\sum_qb_q(\lambda),
 \qquad
 B_\lambda^\varepsilon=\sum_qb_q^\varepsilon(\lambda).
\]
Then
\begin{equation}\label{eq:total-betti-exact-identity}
 B_\lambda=
 \begin{cases}
  \chi_\lambda+2O_\lambda,&\chi_\lambda\geq0,\\
  -\chi_\lambda+2E_\lambda,&\chi_\lambda\leq0,
 \end{cases}
\end{equation}
and the same formula holds for each determinant parity.  Consequently,
\begin{equation}\label{eq:total-betti-basic-bound}
 B_\lambda\geq|\chi_\lambda|,
 \qquad
 B_\lambda^\varepsilon\geq|\chi_\lambda^\varepsilon|.
\end{equation}
Moreover, the difference between either total Betti number and the
absolute value of its Euler characteristic is an even nonnegative
integer.

The separate \(\mathrm{GL}_4\)-formulas yield a refinement not visible
from the \(\mathrm{SL}_4\)-Euler characteristic alone.

\begin{coro}[Refined index-two lower bound]
\label{cor:refined-index-two-lower-bound}
For every nontrivial highest weight \(\lambda\),
\begin{equation}\label{eq:refined-index-two-lower-bound}
 \sum_q b_q(\lambda)
 \geq
 |\chi_\lambda^0|+|\chi_\lambda^1|
 \geq
 |\chi_\lambda|.
\end{equation}
\end{coro}

\begin{proof}
Equation~\eqref{eq:betti-index-two} gives
\(B_\lambda=B_\lambda^0+B_\lambda^1\).  Apply
\eqref{eq:total-betti-basic-bound} to the two summands and then use
\(\chi_\lambda=\chi_\lambda^0+\chi_\lambda^1\).
\end{proof}

The first inequality in~\eqref{eq:refined-index-two-lower-bound} can be
strictly stronger than the bound \(B_\lambda\geq|\chi_\lambda|\): if
\(\chi_\lambda^0\) and \(\chi_\lambda^1\) have opposite signs, part of
their contribution cancels after passage to \(H\), whereas the two
degreewise eigenspaces remain distinct.

\begin{rmk}\label{rmk:balanced-subspace-bound}
Suppose that independent information produces a subspace of dimension \(c_\lambda\) in each of degrees four and five.  These two contributions cancel in the Euler characteristic, but they contribute \(2c_\lambda\) to the total Betti number.  Thus
\[
 B_\lambda\geq|\chi_\lambda|+2c_\lambda.
\]
This observation is especially relevant for self-dual weights \(\ell_1=\ell_3\), but no assertion about the existence or origin of such a subspace is needed for the results above.
\end{rmk}

\subsection{Infinite families with growing cohomological lower bounds}\label{subse:forced-cohomology-families}

The residue polynomials of Theorems~\ref{thm:sl4-residue-formulas} and~\ref{thm:gl4-residue-formulas} can now be read as quantitative cohomological statements.  At every nontrivial highest weight for which the relevant residue polynomial is nonzero, the corresponding cohomology is nonzero.  Its sign specifies whether the forced contribution lies on the even side \(H^4\oplus H^6\) or on the odd side \(H^3\oplus H^5\).  The two \(\mathrm{GL}_4\)-polynomials give the same
information separately in the two determinant eigenspaces.

For example, the residue formula with
\(\lambda=(12t,12t,12t)\) gives
\[
 \chi_{(12t,12t,12t)}=8t^2-4t
 \qquad(t\geq1).
\]
Proposition~\ref{prop:degree-parity-identities} therefore yields the following stronger form of nonvanishing.

\begin{coro}\label{cor:quadratic-even-degree-growth}
For every \(t\geq1\),
\[
 \dim H^4\left(\Gamma,\mathcal M_{(12t,12t,12t)}\right) + \dim H^6\left(\Gamma,\mathcal M_{(12t,12t,12t)}\right) \geq8t^2-4t.
\]
In particular,
\[
 \sum_q\dim H^q\left(\Gamma,\mathcal M_{(12t,12t,12t)}\right)\geq8t^2-4t.
\]
\end{coro}

More generally, restricting any nonzero residue polynomial to a ray in the dominant cone gives a linear or quadratic lower bound whenever the absolute value of that specialization grows.  This turns the complete residue tables into an explicit collection of infinite cohomological nonvanishing and growth families, without requiring a degreewise calculation of the cohomology.

\subsection{Symmetric powers: exact Euler-characteristic identities} \label{subse:symmetric-power-identities}

We now specialize to
\[
 \lambda=(m,0,0),
 \qquad
 \mathcal M_\lambda=\operatorname{Sym}^mV,
\]
where \(V\) is the standard four-dimensional representation.  Let
\[
 s_k=\dim S_k\!\left(\mathrm{SL}_2(\mathbb Z)\right),
\]
with \(s_k=0\) for odd \(k\).  Recall the classical generating series
\begin{equation}\label{eq:cusp-dimension-generating-series}
 \sum_{k\geq0}s_kt^k
 =\frac{t^{12}}{(1-t^4)(1-t^6)}.
\end{equation}
For \(m\geq0\), put
\[
 \delta_{m,0}=
 \begin{cases}1,&m=0,\\0,&m>0,\end{cases}
 \qquad
 e_m=
 \begin{cases}1,&m\text{ is even},\\0,&m\text{ is odd}.
 \end{cases}
\]

\begin{prop}[Coefficient extraction for symmetric powers]
\label{prop:symmetric-power-coefficient-extraction}
For every \(m\geq0\),
\begin{align}
 \chi_h\left(\widetilde{\Gamma},\operatorname{Sym}^mV\right)&=\delta_{m,0}-s_{m+2},\label{eq:sym-untwisted-euler-cusp-dimension}\\
 \chi_h\left(\widetilde{\Gamma},\operatorname{Sym}^mV\otimes\det\right)&=-e_m-s_{m+4}.\label{eq:sym-twisted-euler-cusp-dimension}
\end{align}
Consequently,
\begin{equation}\label{eq:sym-sl4-euler-cusp-dimension}
 \chi_h\left(\Gamma,\operatorname{Sym}^mV\right) =\delta_{m,0}-e_m-s_{m+2}-s_{m+4}.
\end{equation}
In particular, for every even \(m>0\),
\begin{equation}\label{eq:sym-sl4-negative-euler}
 -\chi_h\!\left(\Gamma,\operatorname{Sym}^mV\right)=1+s_{m+2}+s_{m+4}.
\end{equation}
\end{prop}

\begin{proof}
Set \(D(t)=(1-t^4)(1-t^6)\).  From
\eqref{eq:cusp-dimension-generating-series},
\[
 \sum_{m\geq0}s_{m+2}t^m=\frac{t^{10}}{D(t)},
 \qquad
 \sum_{m\geq0}s_{m+4}t^m=\frac{t^8}{D(t)}.
\]
Therefore
\[
 \sum_{m\geq0}(\delta_{m,0}-s_{m+2})t^m
 =1-\frac{t^{10}}{D(t)}
 =\frac{1-t^4-t^6}{1-t^4-t^6+t^{10}},
\]
which is~\eqref{gengl4}.  Moreover,
\[
 \sum_{m\geq0}e_mt^m=\frac1{1-t^2},
\]
and hence
\[
 \sum_{m\geq0}(-e_m-s_{m+4})t^m
 =-\frac1{1-t^2}-\frac{t^8}{D(t)}
 =-\frac{1+t^2-t^6}{1-t^4-t^6+t^{10}},
\]
which is~\eqref{gengl4-odd}.  Adding the two identities and using
Proposition~\ref{prop:degreewise-index-two} proves
\eqref{eq:sym-sl4-euler-cusp-dimension}.
\end{proof}

For odd \(m\), Proposition~\ref{prop:central-degreewise-vanishing}
shows that all the corresponding cohomology groups vanish, not merely
their Euler characteristics.  The remaining discussion therefore
concerns even \(m\).

\subsection{The known determinant-twisted summand}
\label{subse:known-symmetric-power-summand}

The extra information that distinguishes symmetric powers from an
arbitrary highest-weight family is the following degreewise
calculation of Horozov~\cite{Horozov2014}.

\begin{thm}\cite[Thm.~1.1]{Horozov2014}\label{thm:horozov-symmetric-power-cohomology}
Let \(m\geq0\) be even, and let \(V_2\) denote the standard
two-dimensional representation.  Then
\[
 H^q\!\left(\widetilde{\Gamma},\operatorname{Sym}^mV\otimes\det\right)=
 \begin{cases}
  \mathbb Q\oplus H^1_{\mathrm{cusp}}\left(\mathrm{GL}_2(\mathbb Z), \operatorname{Sym}^{m+2}V_2\otimes\det \right),&q=3,\\
  0,&q\neq3.
 \end{cases}
\]
In particular,
\begin{equation}\label{eq:horozov-twisted-betti-number}
 b_3^1(m,0,0)=1+s_{m+4},
 \qquad
 b_q^1(m,0,0)=0\quad(q\neq3).
\end{equation}
\end{thm}

Thus Proposition~\ref{prop:symmetric-power-coefficient-extraction}
determines the Euler characteristic of the complementary, untwisted
summand exactly.  It does not, by itself, determine its individual
Betti numbers.

For brevity, write
\[
 b_q^0(m)=\dim H^q\!\left(\widetilde{\Gamma},\operatorname{Sym}^mV\right),\qquad E_m^0=b_4^0(m)+b_6^0(m).
\]

\begin{coro}[Exact constraints on the complementary summand]\label{cor:symmetric-power-complementary-constraints}
For every even \(m>0\),
\begin{equation}\label{eq:symmetric-power-odd-even-balance}
 b_3^0(m)+b_5^0(m)=s_{m+2}+b_4^0(m)+b_6^0(m).
\end{equation}
Consequently,
\begin{align}
  b_3^0(m)+b_5^0(m)&\geq s_{m+2},\label{eq:sym-complementary-odd-lower-bound}\\
  \sum_q b_q^0(m)&=s_{m+2}+2E_m^0,\label{eq:sym-complementary-total-identity}\\
  \sum_q\dim H^q\left(\Gamma,\operatorname{Sym}^mV\right)&=1+s_{m+4}+s_{m+2}+2E_m^0.\label{eq:sym-sl4-total-identity}
\end{align}
In particular,
\begin{equation}\label{eq:sym-sl4-total-lower-bound}
 \sum_q\dim H^q\left(\Gamma,\operatorname{Sym}^mV\right)\geq1+s_{m+4}+s_{m+2}.
\end{equation}
\end{coro}

\begin{proof}
For \(m>0\), the low-degree vanishing~\eqref{eq:low-degree-vanishing-section-seven} and~\eqref{eq:sym-untwisted-euler-cusp-dimension} give
\[
 -b_3^0(m)+b_4^0(m)-b_5^0(m)+b_6^0(m)=-s_{m+2},
\]
which is equivalent to~\eqref{eq:symmetric-power-odd-even-balance}.  Adding the even and odd sides gives~\eqref{eq:sym-complementary-total-identity}.  Finally,
Theorem~\ref{thm:horozov-symmetric-power-cohomology} and the degreewise decomposition~\eqref{eq:degreewise-index-two} give~\eqref{eq:sym-sl4-total-identity}.
\end{proof}

Identity~\eqref{eq:sym-sl4-total-identity} gives a precise measure of what remains invisible to the Euler characteristic.  The lower bound in~\eqref{eq:sym-sl4-total-lower-bound} is attained exactly when the complementary untwisted summand has no even-degree cohomology.  Even in that case, the Euler characteristic alone does not distinguish between degrees three and five.

\subsection{A degreewise symmetric-power conjecture} \label{subse:symmetric-power-conjecture}

The smallest degreewise table compatible with the preceding exact results places the entire complementary contribution in degree five. We record this as a conjecture, not as a consequence of the Euler characteristic calculation.

\begin{conj}[Degreewise symmetric-power conjecture] \label{conj:symmetric-power-degreewise}
For every even \(m>0\),
\begin{equation}\label{eq:conjectural-untwisted-symmetric-table}
 \dim H^q\!\left(\widetilde{\Gamma},\operatorname{Sym}^mV\right)=
 \begin{cases}
  s_{m+2},&q=5,\\
  0,&q\neq5.
 \end{cases}
\end{equation}
\end{conj}

Together with Theorem~\ref{thm:horozov-symmetric-power-cohomology}, Conjecture~\ref{conj:symmetric-power-degreewise} predicts
\begin{equation}\label{eq:conjectural-sl4-symmetric-table}
 \dim H^q\!\left(\Gamma,\operatorname{Sym}^mV\right)=
 \begin{cases}
  1+s_{m+4},&q=3,\\
  s_{m+2},&q=5,\\
  0,&q\neq3,5,
 \end{cases}
 \qquad(m>0\text{ even}).
\end{equation}
For odd \(m\), every group vanishes by
Proposition~\ref{prop:central-degreewise-vanishing}.  The case \(m=0\)
is separate: the rational trivial-coefficient calculation is
\[
 H^q(\Gamma,\mathbb Q)
 =
 \begin{cases}
  \mathbb Q,&q=0,3,\\
  0,&q\neq0,3;
 \end{cases}
\]
see~\cite{LS1976}.

For completeness, write \(m=12n+r\), where
\(r\in\{0,2,4,6,8,10\}\).  The conjectural nonzero Betti numbers in
\eqref{eq:conjectural-sl4-symmetric-table} are then as follows.

\begin{center}
\begin{tabular}{c|c|c}
\hline
\(r\) & predicted \(b_3=1+s_{m+4}\)
      & predicted \(b_5=s_{m+2}\)\\
\hline
0  & \(n+1\) & \(n-1\), for \(n\geq1\)\\
2  & \(n+1\) & \(n\)\\
4  & \(n+1\) & \(n\)\\
6  & \(n+1\) & \(n\)\\
8  & \(n+2\) & \(n\)\\
10 & \(n+1\) & \(n+1\)\\
\hline
\end{tabular}
\end{center}

\begin{rmk}\label{rmk:content-of-symmetric-conjecture}
Conjecture~\ref{conj:symmetric-power-degreewise} contains two assertions that do not follow from the present Euler-characteristic formulas: first, the untwisted summand has no cohomology in degrees four and six; second, its forced odd-degree contribution occurs in degree five rather than degree three.  Nor do the calculations in this section identify a parabolic source, an Eisenstein lift, or a Hecke-equivariant isomorphism for the predicted degree-five space.  Any such description requires additional cohomological input and lies beyond the scope of the present article.
\end{rmk}

\subsection{Summary of the unconditional conclusions}\label{subse:section-seven-summary}
For nontrivial highest weights, the results of this section give:
\begin{enumerate}
\item complete degreewise vanishing for $\Gamma$ and for both $\widetilde{\Gamma}$-extensions when $\ell_1+\ell_3$ is odd;
\item exact identities expressing the Euler characteristics in terms of the cohomological dimensions in degrees three through six, for $\Gamma$ and, separately, for the two $\widetilde{\Gamma}$-extensions $\widetilde{M}_\lambda$ and $\widetilde{M}_\lambda\otimes\det$;
\item parity-sensitive cohomological lower bounds obtained from the values and signs of the explicit residue polynomials;
\item the refined total lower bound
\[ B_\lambda\geq |\chi_\lambda^0|+|\chi_\lambda^1|;\]
and
\item explicit infinite families for which the resulting cohomological lower bounds grow linearly or quadratically.
\end{enumerate}

For even positive symmetric powers, combining the known determinant-twisted calculation with the Euler-characteristic identities and the degreewise
index-two decomposition gives the further exact identities~\eqref{eq:symmetric-power-odd-even-balance}--\eqref{eq:sym-sl4-total-identity}. These are the strongest conclusions available here without a separate degreewise analysis of the complementary untwisted summand. The conjectural formula~\eqref{eq:conjectural-sl4-symmetric-table} and the table following it are therefore retained as precise targets for subsequent work and are not used in proving any result of this paper.

\section{Concluding remarks and further directions}\label{se:conclusion}

In this paper, we have computed the homological Euler characteristics of \(\mathrm{SL}_4(\mathbb Z)\) and \(\mathrm{GL}_4(\mathbb Z)\) with
coefficients in arbitrary irreducible rational highest-weight representations.  The resulting formulas are indexed by the residue
classes of the highest-weight parameters modulo \(12\), and are encoded by explicit rational generating functions.  In particular, the Euler
characteristics are quasi-polynomial functions of total degree at most two and period \(12\).  For \(\mathrm{GL}_4(\mathbb Z)\), the untwisted and determinant-twisted formulas retain information that is lost after passing to \(\mathrm{SL}_4(\mathbb Z)\).

The results of Section~\ref{se:cohomological-consequences} show that these formulas already have degreewise consequences.  They give
vanishing when the central element \(-I_4\) acts nontrivially, parity-sensitive lower bounds for the Betti numbers, and explicit
families in which the resulting lower bounds grow linearly or quadratically.  They do not, however, determine the individual cohomology groups,
describe their restriction to the boundary, or determine their relation to inner, cuspidal, and Eisenstein cohomology.  We describe below four directions in which the present calculation can be developed further.

\subsection{Degreewise cohomology and the symmetric-power family} \label{subse:further-symmetric-powers}

An Euler characteristic records only the alternating sum of the Betti numbers.  Even when its value and sign are known, it does not determine the degrees in which the cohomology occurs.  The first natural problem is therefore to pass from the residue formulas of this paper to a degreewise calculation.

The symmetric-power family provides a concrete starting point.  Let \(V\) be the standard four-dimensional representation and consider
\(\operatorname{Sym}^mV\).  Horozov determined the cohomology of the determinant-twisted \(\mathrm{GL}_4(\mathbb Z)\)-summand degree by
degree~\cite{Horozov2014}.  Combining his result with our Euler-characteristic formulas gives the exact identities and lower
bounds in Corollary~\ref{cor:symmetric-power-complementary-constraints} for the complementary untwisted summand.  In particular, for every even
\(m>0\),
\[
 b_3^0(m)+b_5^0(m)
 =s_{m+2}+b_4^0(m)+b_6^0(m),
\]
where \(s_k=\dim S_k(\mathrm{SL}_2(\mathbb Z))\).  This identity is
exact, but it leaves open both the possible even-degree contribution
and the distribution of the odd-degree contribution between degrees
three and five.

Conjecture~\ref{conj:symmetric-power-degreewise} predicts the smallest
degreewise table compatible with these identities: for even \(m>0\),
the untwisted summand should be concentrated in degree five and have
dimension \(s_{m+2}\).  Thus a proof must establish two facts that do
not follow from the Euler characteristic: the vanishing in degrees four
and six, and the occurrence of the remaining odd cohomology in degree
five rather than degree three.  A finer version of the problem is to
identify the automorphic or boundary source of these classes and to
determine whether they belong to inner, cuspidal, or Eisenstein
cohomology.

\subsection{Boundary, Eisenstein, and cuspidal cohomology} \label{subse:further-boundary-eisenstein}

The Borel--Serre compactification gives the natural framework for a full degreewise calculation~\cite{BorelSerre73}. As recalled in Subsection~\ref{subse:sl4-cohomological-setting}, the boundary of the locally symmetric space for \(\mathrm{SL}_4(\mathbb Z)\) is assembled from faces indexed by proper rational parabolic subgroups. If \(P=L_PN_P\) is a Levi decomposition and \(\mathfrak n_P=\operatorname{Lie}(N_P)\), Kostant's theorem expresses
$$H^\bullet(\mathfrak n_P,M_\lambda)$$
as a sum of explicit highest-weight representations of \(L_P\)~\cite{Kostant61}. The remaining input for the cohomology of an individual boundary face is therefore the cohomology of arithmetic subgroups of the corresponding proper Levi factor.

For the three maximal standard parabolic subgroups, the Levi factors have, subject to the determinant-one condition, the block types
$$
\mathrm{GL}_1\times\mathrm{GL}_3,\qquad
\mathrm{GL}_2\times\mathrm{GL}_2,\qquad
\mathrm{GL}_3\times\mathrm{GL}_1.
$$

Thus much of the \(E_1\)-page of the boundary spectral sequence is governed by cohomology in ranks one, two, and three. The boundary and Eisenstein calculations for \(\mathrm{SL}_3(\mathbb Z)\), \(G_2(\mathbb Z)\), and \(\mathrm{Sp}_4(\mathbb Z)\) provide useful lower-rank models for this procedure; see~\cite{BHHMG2021,BG2022,BHMG2023}.

There is, however, a genuine new difficulty in rational rank three. The boundary spectral sequence for \(\mathrm{SL}_4\) has three columns. A complete calculation must determine the restriction maps between the faces, the differentials of the spectral sequence, and the extension problems that remain after passing to its limiting page. It must also account for possible ghost classes. These are boundary classes lying in the image of the global restriction map but becoming trivial after restriction to all maximal boundary faces.

Ghost classes were introduced by Borel~\cite{Borel84}. Harder subsequently discussed them in the context of Eisenstein cohomology and, in the case of \(\mathrm{GL}_3\), emphasized their relation to the vanishing of certain \(L\)-values~\cite{Harder90}. The surrounding boundary and Eisenstein-cohomology framework was developed further by Schwermer~\cite{Sch90,Schwermer1994}. Rohlfs treated ghost classes for \(\mathrm{SL}_4\)~\cite{Rohlfs96}, while Franke developed a general method for constructing them for \(\mathrm{SL}_n\)~\cite{Franke98}. More recent detailed studies include \(\mathrm{SL}_3\) and \(\mathrm{GL}_3\) with arbitrary highest-weight coefficients~\cite{BHHMG2021}, \(\mathbb Q\)-rank-two orthogonal Shimura varieties~\cite{BajpaiMoyaGhost}, and the Shimura varieties associated with \(\mathrm{GSp}_4\) and \(\mathrm{GU}(2,2)\)~\cite{MoyaGSp4,MoyaGU22}.

Even a complete description of the boundary cohomology does not by itself determine the Eisenstein cohomology. One must decide which boundary classes are obtained by restricting global classes constructed from Eisenstein series. This involves constant terms, intertwining operators, possible poles and residues, and, in some cases, special values of automorphic \(L\)-functions. Carrying out this program for arbitrary highest-weight coefficients of \(\mathrm{SL}_4(\mathbb Z)\) would identify the Eisenstein contribution, clarify the inner cohomology and its cuspidal part, and explain the degreewise origin of the numerical constraints obtained in this paper.

\subsection{Residue polynomials and cohomological nonvanishing} \label{subse:further-residue-polynomials}

For each residue class \(r=(r_1,r_2,r_3)\) modulo \(12\), write $$\ell_i=12m_i+r_i,\qquad 0\leq r_i<12.$$

The Euler-characteristic formulas for \(\mathrm{SL}_4(\mathbb Z)\) and for the two \(\mathrm{GL}_4(\mathbb Z)\)-extensions are then given by explicit polynomials in \((m_1,m_2,m_3)\) of total degree at most two. The integral points on the zero set of any one of these polynomials correspond to weights for which that Euler characteristic alone gives no nonvanishing information. Away from the zero set, a positive value gives a lower bound in degrees four and six, whereas a negative value gives a lower bound in degrees three and five.

Along a one-parameter arithmetic progression $$\lambda(t)=\lambda_0+12t\mu,\qquad t\in\mathbb Z_{\geq0},$$ where \(\lambda_0,\mu\in\mathbb Z_{\geq0}^3\) and \(\mu\neq0\), the relevant residue polynomial specializes to a polynomial in \(t\) of degree at most two. Unless this specialization vanishes identically, its leading nonzero term determines the asymptotic growth.

It would be useful to analyze this finite collection of residue polynomials systematically. One may ask for an explicit description of their zero sets and sign regions, a classification of the initial weights and directions for which the corresponding specializations exhibit quadratic, linear, or bounded growth, and a comparison of the untwisted and determinant-twisted polynomials. The last comparison is especially relevant because, for nontrivial \(\lambda\), the refined bound
$$\sum_q b_q(\lambda)\geq |\chi_\lambda^0|+|\chi_\lambda^1|$$

detects cohomology that the corresponding \(\mathrm{SL}_4(\mathbb Z)\)-Euler characteristic may fail to detect because of cancellation between \(\chi_\lambda^0\) and \(\chi_\lambda^1\).

The rational generating functions in this paper connect these questions with the general theory of vector-partition functions and quasi-polynomial behavior; see~\cite{Sturmfels95}. In the present rank-four case, however, the problem is completely explicit: it reduces to the arithmetic and real geometry of a finite family of polynomials of degree at most two. Such an analysis would turn the residue tables into a more conceptual description of the cohomological nonvanishing regions among the dominant integral weights.

\subsection{Generalizations to higher-rank arithmetic groups}\label{subse:further-higher-rank}
Wall's formula expresses the homological Euler characteristic as a sum over conjugacy classes of torsion elements, with each contribution involving the orbifold Euler characteristic of the corresponding centralizer and the trace of the torsion element on the coefficient system~\cite{Wall}. Horozov developed this method for arithmetic groups and carried out several explicit calculations~\cite{Horozov2005}. If a primitive \(d\)-th root of unity occurs as an eigenvalue of a torsion element of \(\mathrm{GL}_n(\mathbb Z)\), then \(\varphi(d)\leq n\). Thus, for fixed \(n\), the possible cyclotomic orders are bounded. This suggests that the quasi-polynomial dependence on the highest weight found here is part of a more general phenomenon.

For fixed \(n\) and \(\varepsilon\in\{0,1\}\), consider

$$
\chi_{n,\lambda}^{\varepsilon}
=
\chi_h\!\left(
\mathrm{GL}_n(\mathbb Z),
\widetilde{\mathcal M}_\lambda\otimes\det^{\varepsilon}
\right),
$$

where \(\widetilde{\mathcal M}_\lambda\) is a chosen extension of the
\(\mathrm{SL}_n\)-coefficient system \(\mathcal M_\lambda\). Since the determinant takes values in \(\{\pm1\}\) on \(\mathrm{GL}_n(\mathbb Z)\), determinant twists on this group depend only on the parity of the exponent. One may ask whether the functions

$$
\lambda\longmapsto\chi_{n,\lambda}^{\varepsilon}
$$

are quasi-polynomial or piecewise quasi-polynomial in the nonnegative fundamental-weight coordinates of \(\lambda\). If so, can their periods and degrees be determined directly from the cyclotomic types, eigenvalue multiplicities, and centralizer structures of the contributing torsion elements, without computing all the individual residue formulas? Finally, does keeping \(\chi_{n,\lambda}^{0}\) and \(\chi_{n,\lambda}^{1}\) separate yield strictly stronger cohomological lower bounds than those obtained from their sum

$$
\chi_h\!\left(\mathrm{SL}_n(\mathbb Z),\mathcal M_\lambda\right)
=
\chi_{n,\lambda}^{0}+\chi_{n,\lambda}^{1}\,?
$$

Analogous questions concerning quasi-polynomiality, periods, and degrees may be considered for symplectic and orthogonal arithmetic groups. In each case, the main tasks are to organize the torsion-centralizer contributions, compute the traces on highest-weight representations uniformly, and determine the resulting periods and degrees. The formulas established here give complete Euler-characteristic calculations in the \(n=4\) case for both \(\mathrm{SL}_4(\mathbb Z)\) and \(\mathrm{GL}_4(\mathbb Z)\), and thereby provide a concrete test case for formulating and testing
analogous statements in higher rank.

\subsection{The next cohomological step} The present work provides explicit Euler-characteristic formulas for both \(\mathrm{SL}_4(\mathbb Z)\) and \(\mathrm{GL}_4(\mathbb Z)\), together with exact numerical constraints on the corresponding cohomology.  A major next step is a full calculation of the boundary and Eisenstein cohomology of \(\mathrm{SL}_4(\mathbb Z)\), beginning with symmetric-power coefficient systems and then extending to arbitrary irreducible highest-weight coefficient systems.
\clearpage
\appendix
\section{Euler Characteristic Formula Tables for \texorpdfstring{$\mathrm{SL}_4(\mathbb{Z})$}{SL4(Z)}}\label{app:sl4-residue-tables}

This appendix records all $648$ potentially nonzero residue formulas used in Theorem~\ref{thm:sl4-residue-formulas}.  Throughout, $\ell_i=12m_i+r_i$ with $0\leq r_i<12$.  The tables are organized by $r_3$; their row and column labels give $r_1$ and $r_2$, respectively. Residue triples excluded by Proposition~\ref{prop:sl4-parity-vanishing} have value zero and are not repeated here.

\subsection{The Euler characteristic of $\SL_4(\Z)$ with $\ell_1\equiv 0\pmod{2}$}
\subsubsection{Case 1\,: $r_3 = 0$}

{\begin{center}
\tiny\renewcommand{\arraystretch}{2.4}
{
\begin{longtable}{|c|c|c|c|c|}
\hline
\diagbox[width=3em]
{$r_1 $}{$r_2 $}&0&1&2&3\\
\hline
0&$\substack{\left(2\*m_{2}
 + \left(2\*m_{3}
 - 2\right) \right) \*m_{1}\\
 + \left(2\*m_{2}^2
 + 2\*m_{3}\*m_{2}
 - 2\*m_{3}\right) }$&
 $\substack{\left(-2\*m_{2}
 - 2\right) \*m_{1}\\
 + \left(-2\*m_{2}^2
 + \left(-2\*m_{3}
 - 2\right) \*m_{2}
 - 2\*m_{3}\right) }$&$\substack{
  \left(2\*m_{2}
 + 2\*m_{3}\right) \*m_{1}\\
 + \left(2\*m_{2}^2
 + \left(2\*m_{3}
 + 2\right) \*m_{2}
 + 1\right) 
 }
 $&$
 \substack{
\left(-2\*m_{2}
 - 2\right) \*m_{1}\\
 + \left(-2\*m_{2}^2
 + \left(-2\*m_{3}
 - 2\right) \*m_{2}
 - 2\*m_{3}\right)
 }$
\\
\hline
2&$\substack{
 \left(2\*m_{2}
 + \left(2\*m_{3}
 - 2\right) \right) \*m_{1}\\
 + \left(2\*m_{2}^2
 + 2\*m_{3}\*m_{2}
 - 1\right) 
} $&$
\substack{
\left(-2\*m_{2}
 - 2\right) \*m_{1}\\
 + \left(-2\*m_{2}^2
 + \left(-2\*m_{3}
 - 2\right) \*m_{2}
 - 1\right)
}$
&
$\substack{
 \left(2\*m_{2}
 + 2\*m_{3}\right) \*m_{1}\\
 + \left(2\*m_{2}^2
 + \left(2\*m_{3}
 + 2\right) \*m_{2}
 + 2\*m_{3}\right) 
}
$
&$
\substack{
\left(-2\*m_{2}
 - 2\right) \*m_{1}\\
 + \left(-2\*m_{2}^2
 + \left(-2\*m_{3}
 - 2\right) \*m_{2}
 - 1\right) 
}
$\\
\hline 
4& $
\substack{  \left(2\*m_{2}
 + \left(2\*m_{3}
 - 2\right) \right) \*m_{1}\\
 + \left(2\*m_{2}^2
 + 2\*m_{3}\*m_{2}
 - 1\right)}
$&
$\substack{
 \left(-2\*m_{2}
 - 2\right) \*m_{1}\\
 + \left(-2\*m_{2}^2
 + \left(-2\*m_{3}
 - 2\right) \*m_{2}
 - 1\right)  
}$
&$
\substack{\left(2\*m_{2}
 + 2\*m_{3}\right) \*m_{1}\\
 + \left(2\*m_{2}^2
 + \left(2\*m_{3}
 + 2\right) \*m_{2}
 + 2\*m_{3}\right)  }
$&
$\substack{  \left(-2\*m_{2}
 - 2\right) \*m_{1}\\
 + \left(-2\*m_{2}^2
 + \left(-2\*m_{3}
 - 2\right) \*m_{2}
 - 1\right)  }$
\\
\hline
\end{longtable}}
\end{center}
}

{ \begin{center}
\tiny\renewcommand{\arraystretch}{2.4}
{\begin{longtable}{|c|c|c|c|c|}
\hline
\diagbox[width=3em]
{$r_1 $}{$r_2 $}&4&5&6&7\\
\hline
0&$\substack{ \left(2\*m_{2}
 + 2\*m_{3}\right) \*m_{1}\\
 + \left(2\*m_{2}^2
 + \left(2\*m_{3}
 + 2\right) \*m_{2}
 + 1\right) }$&
 $\substack{\left(-2\*m_{2}
 - 2\right) \*m_{1}\\
 + \left(-2\*m_{2}^2
 + \left(-2\*m_{3}
 - 2\right) \*m_{2}
 - 2\*m_{3}\right) }$
 &$\substack{
  \left(2\*m_{2}
 + 2\*m_{3}\right) \*m_{1}\\
 + \left(2\*m_{2}^2
 + \left(2\*m_{3}
 + 2\right) \*m_{2}
 + 1\right) 
 }
 $&$
 \substack{
 \left(-2\*m_{2}
 - 2\right) \*m_{1}\\
 + \left(-2\*m_{2}^2
 + \left(-2\*m_{3}
 - 2\right) \*m_{2}
 - 2\*m_{3}\right) 
 }$
\\
\hline
2&
$\substack{
\left(2\*m_{2}
 + 2\*m_{3}\right) \*m_{1}\\
 + \left(2\*m_{2}^2
 + \left(2\*m_{3}
 + 2\right) \*m_{2}
 + 2\*m_{3}\right) 
}$
&
$\substack{
 \left(-2\*m_{2}
 - 2\right) \*m_{1}\\
 + \left(-2\*m_{2}^2
 + \left(-2\*m_{3}
 - 2\right) \*m_{2}
 - 1\right) 
}$
&
$\substack{
 \left(2\*m_{2}
 + 2\*m_{3}\right) \*m_{1}\\
 + \left(2\*m_{2}^2
 + \left(2\*m_{3}
 + 4\right) \*m_{2}
 + \left(2\*m_{3}
 + 1\right) \right)
}$
&
$\substack{\left(-2\*m_{2}
 - 2\right) \*m_{1}\\
 + \left(-2\*m_{2}^2
 + \left(-2\*m_{3}
 - 4\right) \*m_{2}
 + \left(-2\*m_{3}
 - 2\right) \right)  }$
\\
\hline 
4& $\substack{  \left(2\*m_{2}
 + 2\*m_{3}\right) \*m_{1}\\
 + \left(2\*m_{2}^2
 + \left(2\*m_{3}
 + 4\right) \*m_{2}
 + \left(2\*m_{3}
 + 1\right) \right)   }$&
 $\substack{  \left(-2\*m_{2}
 - 2\right) \*m_{1}\\
 + \left(-2\*m_{2}^2
 + \left(-2\*m_{3}
 - 4\right) \*m_{2}
 + \left(-2\*m_{3}
 - 2\right) \right)  }$
 &
 $\substack{  \left(2\*m_{2}
 + 2\*m_{3}\right) \*m_{1}\\
 + \left(2\*m_{2}^2
 + \left(2\*m_{3}
 + 2\right) \*m_{2}
 + 2\*m_{3}\right)  }$
 &$\substack{  \left(-2\*m_{2}
 - 2\right) \*m_{1}\\
 + \left(-2\*m_{2}^2
 + \left(-2\*m_{3}
 - 2\right) \*m_{2}
 - 1\right)  }$\\
\end{longtable}}
\end{center}
}
{ \begin{center}
\tiny\renewcommand{\arraystretch}{2.4}
{\begin{longtable}{|c|c|c|c|c|}
\hline
\diagbox[width=3em]
{$r_1 $}{$r_2 $}&8&9&10&11\\
\hline
0&$\substack{ \left(2\*m_{2}
 + 2\*m_{3}\right) \*m_{1}\\
 + \left(2\*m_{2}^2
 + \left(2\*m_{3}
 + 4\right) \*m_{2}
 + 2\right)   }$&
 $\substack{ \left(-2\*m_{2}
 - 4\right) \*m_{1}\\
 + \left(-2\*m_{2}^2
 + \left(-2\*m_{3}
 - 6\right) \*m_{2}
 + \left(-4\*m_{3}
 - 4\right) \right)    }$
 &$\substack{  \left(2\*m_{2}
 + \left(2\*m_{3}
 + 2\right) \right) \*m_{1}\\
 + \left(2\*m_{2}^2
 + \left(2\*m_{3}
 + 4\right) \*m_{2}
 + \left(2\*m_{3}
 + 2\right) \right) 
 }
 $&$
 \substack{
 \left(-2\*m_{2}
 - 2\right) \*m_{1}\\
 + \left(-2\*m_{2}^2
 + \left(-2\*m_{3}
 - 2\right) \*m_{2}
 - 2\*m_{3}\right) 
 }$
\\
\hline
2&
$\substack{
\left(2\*m_{2}
 + 2\*m_{3}\right) \*m_{1}\\
 + \left(2\*m_{2}^2
 + \left(2\*m_{3}
 + 2\right) \*m_{2}
 + 2\*m_{3}\right) 
}$
&
$\substack{  \left(-2\*m_{2}
 - 4\right) \*m_{1}\\
 + \left(-2\*m_{2}^2
 + \left(-2\*m_{3}
 - 4\right) \*m_{2}
 - 2\right)  }$
&$\substack{ \left(2\*m_{2}
 + \left(2\*m_{3}
 + 2\right) \right) \*m_{1}\\
 + \left(2\*m_{2}^2
 + \left(2\*m_{3}
 + 6\right) \*m_{2}
 + \left(4\*m_{3}
 + 4\right) \right)  }$&$
 \substack{  \left(-2\*m_{2}
 - 2\right) \*m_{1}\\
 + \left(-2\*m_{2}^2
 + \left(-2\*m_{3}
 - 4\right) \*m_{2}
 + \left(-2\*m_{3}
 - 2\right) \right)  }
 $\\
\hline 
4& $\substack{ \left(2\*m_{2}
 + 2\*m_{3}\right) \*m_{1}\\
 + \left(2\*m_{2}^2
 + \left(2\*m_{3}
 + 4\right) \*m_{2}
 + \left(2\*m_{3}
 + 1\right) \right)  }$&
 $\substack{ \left(-2\*m_{2}
 - 4\right) \*m_{1}\\
 + \left(-2\*m_{2}^2
 + \left(-2\*m_{3}
 - 6\right) \*m_{2}
 + \left(-2\*m_{3}
 - 5\right) \right)   }$&
 $\substack{ \left(2\*m_{2}
 + \left(2\*m_{3}
 + 2\right) \right) \*m_{1}\\
 + \left(2\*m_{2}^2
 + \left(2\*m_{3}
 + 6\right) \*m_{2}
 + \left(4\*m_{3}
 + 4\right) \right) }$
 &$\substack{  \left(-2\*m_{2}
 - 2\right) \*m_{1}\\
 + \left(-2\*m_{2}^2
 + \left(-2\*m_{3}
 - 4\right) \*m_{2}
 + \left(-2\*m_{3}
 - 2\right) \right)  }$\\
\end{longtable}}
\end{center}
}

\subsubsection{Case 2\,: $r_3 = 2$}

{ \begin{center}
\tiny\renewcommand{\arraystretch}{2.4}
{\begin{longtable}{|c|c|c|c|c|}
\hline
\diagbox[width=3em]
{$r_1 $}{$r_2 $}&0&1&2&3\\
\hline
0&$\substack{
 \left(2\*m_{2}
 + 2\*m_{3}\right) \*m_{1}\\
 + \left(2\*m_{2}^2
 + 2\*m_{3}\*m_{2}
 + \left(-2\*m_{3}
 - 1\right) \right)
} $&$
\substack{
  -2\*m_{2}\*m_{1}\\
 + \left(-2\*m_{2}^2
 + \left(-2\*m_{3}
 - 2\right) \*m_{2}
 + \left(-2\*m_{3}
 - 1\right) \right) 
}$
&
$\substack{
\left(2\*m_{2}
 + \left(2\*m_{3}
 + 2\right) \right) \*m_{1}\\
 + \left(2\*m_{2}^2
 + \left(2\*m_{3}
 + 2\right) \*m_{2}\right) 
}
$
&$
\substack{
-2\*m_{2}\*m_{1}\\
 + \left(-2\*m_{2}^2
 + \left(-2\*m_{3}
 - 2\right) \*m_{2}
 + \left(-2\*m_{3}
 - 1\right) \right) 
}
$\\
\hline
2&$\substack{
 \left(2\*m_{2}
 + 2\*m_{3}\right) \*m_{1}\\
 + \left(2\*m_{2}^2
 + 2\*m_{3}\*m_{2}\right) 
} $&$
\substack{
 -2\*m_{2}\*m_{1}\\
 + \left(-2\*m_{2}^2
 + \left(-2\*m_{3}
 - 2\right) \*m_{2}\right) 
}$
&
$\substack{
  \left(2\*m_{2}
 + \left(2\*m_{3}
 + 2\right) \right) \*m_{1}\\
 + \left(2\*m_{2}^2
 + \left(2\*m_{3}
 + 2\right) \*m_{2}
 + \left(2\*m_{3}
 + 1\right) \right) 
}
$
&$
\substack{
-2\*m_{2}\*m_{1}\\
 + \left(-2\*m_{2}^2
 + \left(-2\*m_{3}
 - 2\right) \*m_{2}\right) 
}
$\\
\hline 
4& $
\substack{ 
\left(2\*m_{2}
 + 2\*m_{3}\right) \*m_{1}\\
 + \left(2\*m_{2}^2
 + 2\*m_{3}\*m_{2}\right)
}
$&
$\substack{
-2\*m_{2}\*m_{1}\\
 + \left(-2\*m_{2}^2
 + \left(-2\*m_{3}
 - 2\right) \*m_{2}\right) 
}$
&$
\substack{
  \left(2\*m_{2}
 + \left(2\*m_{3}
 + 2\right) \right) \*m_{1}\\
 + \left(2\*m_{2}^2
 + \left(2\*m_{3}
 + 4\right) \*m_{2}
 + \left(2\*m_{3}
 + 2\right) \right) 
}
$&
$\substack{
-2\*m_{2}\*m_{1}\\
 + \left(-2\*m_{2}^2
 + \left(-2\*m_{3}
 - 2\right) \*m_{2}\right) 
}$
\\
\end{longtable}}
\end{center}
}

{ \begin{center}
\tiny\renewcommand{\arraystretch}{2.4}
{\begin{longtable}{|c|c|c|c|c|}
\hline
\diagbox[width=3em]
{$r_1 $}{$r_2 $}&4&5&6&7\\
\hline
0&$\substack{
 \left(2\*m_{2}
 + \left(2\*m_{3}
 + 2\right) \right) \*m_{1}\\
 + \left(2\*m_{2}^2
 + \left(2\*m_{3}
 + 2\right) \*m_{2}\right) 
} $&$
\substack{
 -2\*m_{2}\*m_{1}\\
 + \left(-2\*m_{2}^2
 + \left(-2\*m_{3}
 - 2\right) \*m_{2}
 + \left(-2\*m_{3}
 - 1\right) \right) 
}$
&
$\substack{
 \left(2\*m_{2}
 + \left(2\*m_{3}
 + 2\right) \right) \*m_{1}\\
 + \left(2\*m_{2}^2
 + \left(2\*m_{3}
 + 4\right) \*m_{2}
 + 1\right) 
}
$
&$
\substack{
 \left(-2\*m_{2}
 - 2\right) \*m_{1}\\
 + \left(-2\*m_{2}^2
 + \left(-2\*m_{3}
 - 4\right) \*m_{2}
 + \left(-2\*m_{3}
 - 2\right) \right) 
}
$\\
\hline
2&$\substack{
 \left(2\*m_{2}
 + \left(2\*m_{3}
 + 2\right) \right) \*m_{1}\\
 + \left(2\*m_{2}^2
 + \left(2\*m_{3}
 + 4\right) \*m_{2}
 + \left(2\*m_{3}
 + 2\right) \right) 
} $&$
\substack{
 -2\*m_{2}\*m_{1}\\
 + \left(-2\*m_{2}^2
 + \left(-2\*m_{3}
 - 2\right) \*m_{2}\right) 
}$
&
$\substack{
  \left(2\*m_{2}
 + \left(2\*m_{3}
 + 2\right) \right) \*m_{1}\\
 + \left(2\*m_{2}^2
 + \left(2\*m_{3}
 + 2\right) \*m_{2}
 + \left(2\*m_{3}
 + 1\right) \right)
}
$
&$
\substack{
\left(-2\*m_{2}
 - 2\right) \*m_{1}\\
 + \left(-2\*m_{2}^2
 + \left(-2\*m_{3}
 - 6\right) \*m_{2}
 + \left(-2\*m_{3}
 - 4\right) \right) 
}
$\\
\hline 
4& $
\substack{ 
  \left(2\*m_{2}
 + \left(2\*m_{3}
 + 2\right) \right) \*m_{1}\\
 + \left(2\*m_{2}^2
 + \left(2\*m_{3}
 + 2\right) \*m_{2}
 + \left(2\*m_{3}
 + 1\right) \right)
}
$&
$\substack{
 -2\*m_{2}\*m_{1}\\
 + \left(-2\*m_{2}^2
 + \left(-2\*m_{3}
 - 4\right) \*m_{2}
 + \left(-2\*m_{3}
 - 1\right) \right) 
}$
&$
\substack{
 \left(2\*m_{2}
 + \left(2\*m_{3}
 + 2\right) \right) \*m_{1}\\
 + \left(2\*m_{2}^2
 + \left(2\*m_{3}
 + 4\right) \*m_{2}
 + \left(2\*m_{3}
 + 2\right) \right)
}
$&
$\substack{
 \left(-2\*m_{2}
 - 2\right) \*m_{1}\\
 + \left(-2\*m_{2}^2
 + \left(-2\*m_{3}
 - 4\right) \*m_{2}
 - 1\right) 
}$
\\
\end{longtable}}
\end{center}
}

{ \begin{center}
\tiny\renewcommand{\arraystretch}{2.4}
{\begin{longtable}{|c|c|c|c|c|}
\hline
\diagbox[width=3em]
{$r_1 $}{$r_2 $}&8&9&10&11\\
\hline
0&$\substack{
 \left(2\*m_{2}
 + \left(2\*m_{3}
 + 2\right) \right) \*m_{1}\\
 + \left(2\*m_{2}^2
 + \left(2\*m_{3}
 + 2\right) \*m_{2}\right) 
} $&$
\substack{
-2\*m_{2}\*m_{1}\\
 + \left(-2\*m_{2}^2
 + \left(-2\*m_{3}
 - 4\right) \*m_{2}
 + \left(-4\*m_{3}
 - 2\right) \right) 
}$
&
$\substack{
\left(2\*m_{2}
 + \left(2\*m_{3}
 + 4\right) \right) \*m_{1}\\
 + \left(2\*m_{2}^2
 + \left(2\*m_{3}
 + 6\right) \*m_{2}
 + \left(2\*m_{3}
 + 4\right) \right) 
}
$
&$
\substack{
\left(-2\*m_{2}
 - 2\right) \*m_{1}\\
 + \left(-2\*m_{2}^2
 + \left(-2\*m_{3}
 - 4\right) \*m_{2}
 + \left(-2\*m_{3}
 - 2\right) \right) 
}
$\\
\hline
2&$\substack{
 \left(2\*m_{2}
 + \left(2\*m_{3}
 + 2\right) \right) \*m_{1}\\
 + \left(2\*m_{2}^2
 + \left(2\*m_{3}
 + 4\right) \*m_{2}
 + \left(2\*m_{3}
 + 2\right) \right)
} $&$
\substack{
-2\*m_{2}\*m_{1}\\
 + \left(-2\*m_{2}^2
 + \left(-2\*m_{3}
 - 2\right) \*m_{2}\right)  
}$
&
$\substack{
 \left(2\*m_{2}
 + \left(2\*m_{3}
 + 4\right) \right) \*m_{1}\\
 + \left(2\*m_{2}^2
 + \left(2\*m_{3}
 + 6\right) \*m_{2}
 + \left(4\*m_{3}
 + 5\right) \right) 
}
$
&$
\substack{
 \left(-2\*m_{2}
 - 2\right) \*m_{1}\\
 + \left(-2\*m_{2}^2
 + \left(-2\*m_{3}
 - 6\right) \*m_{2}
 + \left(-2\*m_{3}
 - 4\right) \right) 
}
$\\
\hline 
4& $
\substack{ 
 \left(2\*m_{2}
 + \left(2\*m_{3}
 + 2\right) \right) \*m_{1}\\
 + \left(2\*m_{2}^2
 + \left(2\*m_{3}
 + 4\right) \*m_{2}
 + \left(2\*m_{3}
 + 2\right) \right) 
}
$&
$\substack{
-2\*m_{2}\*m_{1}\\
 + \left(-2\*m_{2}^2
 + \left(-2\*m_{3}
 - 4\right) \*m_{2}
 + \left(-2\*m_{3}
 - 1\right) \right) 
}$
&$
\substack{
 \left(2\*m_{2}
 + \left(2\*m_{3}
 + 4\right) \right) \*m_{1}\\
 + \left(2\*m_{2}^2
 + \left(2\*m_{3}
 + 6\right) \*m_{2}
 + \left(4\*m_{3}
 + 5\right) \right) 
}
$&
$\substack{
 \left(-2\*m_{2}
 - 2\right) \*m_{1}\\
 + \left(-2\*m_{2}^2
 + \left(-2\*m_{3}
 - 6\right) \*m_{2}
 + \left(-2\*m_{3}
 - 4\right) \right) 
}$
\\
\end{longtable}}
\end{center}
}

\subsubsection{Case 3\,:  $r_3 = 4$ }

{ \begin{center}
\tiny\renewcommand{\arraystretch}{2.4}
{\begin{longtable}{|c|c|c|c|c|}
\hline
\diagbox[width=3em]
{$r_1 $}{$r_2 $}&0&1&2&3\\
\hline
0&$\substack{
  \left(2\*m_{2}
 + 2\*m_{3}\right) \*m_{1}\\
 + \left(2\*m_{2}^2
 + 2\*m_{3}\*m_{2}
 + \left(-2\*m_{3}
 - 1\right) \right) 
} $&$
\substack{
-2\*m_{2}\*m_{1}\\
 + \left(-2\*m_{2}^2
 + \left(-2\*m_{3}
 - 2\right) \*m_{2}
 + \left(-2\*m_{3}
 - 1\right) \right) 
}$
&
$\substack{
 \left(2\*m_{2}
 + \left(2\*m_{3}
 + 2\right) \right) \*m_{1}\\
 + \left(2\*m_{2}^2
 + \left(2\*m_{3}
 + 2\right) \*m_{2}\right)
}
$
&$
\substack{
-2\*m_{2}\*m_{1}\\
 + \left(-2\*m_{2}^2
 + \left(-2\*m_{3}
 - 2\right) \*m_{2}
 + \left(-2\*m_{3}
 - 1\right) \right) 
}
$\\
\hline
2&$\substack{
 \left(2\*m_{2}
 + 2\*m_{3}\right) \*m_{1}\\
 + \left(2\*m_{2}^2
 + 2\*m_{3}\*m_{2}\right) 
} $&$
\substack{
 -2\*m_{2}\*m_{1}\\
 + \left(-2\*m_{2}^2
 + \left(-2\*m_{3}
 - 2\right) \*m_{2}\right) 
}$
&
$\substack{
\left(2\*m_{2}
 + \left(2\*m_{3}
 + 2\right) \right) \*m_{1}\\
 + \left(2\*m_{2}^2
 + \left(2\*m_{3}
 + 4\right) \*m_{2}
 + \left(2\*m_{3}
 + 2\right) \right) 
}
$
&$
\substack{
-2\*m_{2}\*m_{1}\\
 + \left(-2\*m_{2}^2
 + \left(-2\*m_{3}
 - 2\right) \*m_{2}\right) 
}
$\\
\hline 
4& $
\substack{ 
 \left(2\*m_{2}
 + 2\*m_{3}\right) \*m_{1}\\
 + \left(2\*m_{2}^2
 + \left(2\*m_{3}
 + 2\right) \*m_{2}
 - 1\right) 
}
$&
$\substack{
-2\*m_{2}\*m_{1}\\
 + \left(-2\*m_{2}^2
 + \left(-2\*m_{3}
 - 2\right) \*m_{2}\right) 
}$
&$
\substack{
  \left(2\*m_{2}
 + \left(2\*m_{3}
 + 2\right) \right) \*m_{1}\\
 + \left(2\*m_{2}^2
 + \left(2\*m_{3}
 + 2\right) \*m_{2}
 + \left(2\*m_{3}
 + 1\right) \right) 
}
$&
$\substack{
-2\*m_{2}\*m_{1}\\
 + \left(-2\*m_{2}^2
 + \left(-2\*m_{3}
 - 2\right) \*m_{2}\right)
}$
\\
\end{longtable}}
\end{center}
}

{ \begin{center}
\tiny\renewcommand{\arraystretch}{2.4}
{\begin{longtable}{|c|c|c|c|c|}
\hline
\diagbox[width=3em]
{$r_1 $}{$r_2 $}&4&5&6&7\\
\hline
0&$\substack{
  \left(2\*m_{2}
 + \left(2\*m_{3}
 + 2\right) \right) \*m_{1}\\
 + \left(2\*m_{2}^2
 + \left(2\*m_{3}
 + 4\right) \*m_{2}
 + 1\right) 
} $&$
\substack{
 \left(-2\*m_{2}
 - 2\right) \*m_{1}\\
 + \left(-2\*m_{2}^2
 + \left(-2\*m_{3}
 - 4\right) \*m_{2}
 + \left(-2\*m_{3}
 - 2\right) \right) 
}$
&
$\substack{
 \left(2\*m_{2}
 + \left(2\*m_{3}
 + 2\right) \right) \*m_{1}\\
 + \left(2\*m_{2}^2
 + \left(2\*m_{3}
 + 2\right) \*m_{2}\right)
}
$
&$
\substack{
-2\*m_{2}\*m_{1}\\
 + \left(-2\*m_{2}^2
 + \left(-2\*m_{3}
 - 2\right) \*m_{2}
 + \left(-2\*m_{3}
 - 1\right) \right)
}
$\\
\hline
2&$\substack{
 \left(2\*m_{2}
 + \left(2\*m_{3}
 + 2\right) \right) \*m_{1}\\
 + \left(2\*m_{2}^2
 + \left(2\*m_{3}
 + 2\right) \*m_{2}
 + \left(2\*m_{3}
 + 1\right) \right) 
} $&$
\substack{
 \left(-2\*m_{2}
 - 2\right) \*m_{1}\\
 + \left(-2\*m_{2}^2
 + \left(-2\*m_{3}
 - 4\right) \*m_{2}
 - 1\right) 
}$
&
$\substack{
 \left(2\*m_{2}
 + \left(2\*m_{3}
 + 2\right) \right) \*m_{1}\\
 + \left(2\*m_{2}^2
 + \left(2\*m_{3}
 + 4\right) \*m_{2}
 + \left(2\*m_{3}
 + 2\right) \right) 
}
$
&$
\substack{
-2\*m_{2}\*m_{1}\\
 + \left(-2\*m_{2}^2
 + \left(-2\*m_{3}
 - 4\right) \*m_{2}
 + \left(-2\*m_{3}
 - 1\right) \right)
}
$\\
\hline 
4& $
\substack{ 
  \left(2\*m_{2}
 + \left(2\*m_{3}
 + 2\right) \right) \*m_{1}\\
 + \left(2\*m_{2}^2
 + \left(2\*m_{3}
 + 4\right) \*m_{2}
 + \left(2\*m_{3}
 + 2\right) \right)
}
$&
$\substack{
\left(-2\*m_{2}
 - 2\right) \*m_{1}\\
 + \left(-2\*m_{2}^2
 + \left(-2\*m_{3}
 - 6\right) \*m_{2}
 + \left(-2\*m_{3}
 - 4\right) \right) 
}$
&$
\substack{
\left(2\*m_{2}
 + \left(2\*m_{3}
 + 2\right) \right) \*m_{1}\\
 + \left(2\*m_{2}^2
 + \left(2\*m_{3}
 + 4\right) \*m_{2}
 + \left(2\*m_{3}
 + 2\right) \right)  
}
$&
$\substack{
-2\*m_{2}\*m_{1}\\
 + \left(-2\*m_{2}^2
 + \left(-2\*m_{3}
 - 2\right) \*m_{2}\right) 
}$
\\
\end{longtable}}
\end{center}
}

{ \begin{center}
\tiny\renewcommand{\arraystretch}{2.4}
{\begin{longtable}{|c|c|c|c|c|}
\hline
\diagbox[width=3em]
{$r_1 $}{$r_2 $}&8&9&10&11\\
\hline
0&$\substack{
 \left(2\*m_{2}
 + \left(2\*m_{3}
 + 2\right) \right) \*m_{1}\\
 + \left(2\*m_{2}^2
 + \left(2\*m_{3}
 + 4\right) \*m_{2}
 + 1\right) 

} $&$
\substack{
 \left(-2\*m_{2}
 - 2\right) \*m_{1}\\
 + \left(-2\*m_{2}^2
 + \left(-2\*m_{3}
 - 6\right) \*m_{2}
 + \left(-4\*m_{3}
 - 5\right) \right)
}$
&
$\substack{
 \left(2\*m_{2}
 + \left(2\*m_{3}
 + 4\right) \right) \*m_{1}\\
 + \left(2\*m_{2}^2
 + \left(2\*m_{3}
 + 6\right) \*m_{2}
 + \left(2\*m_{3}
 + 4\right) \right) 
}
$
&$
\substack{
  \left(-2\*m_{2}
 - 2\right) \*m_{1}\\
 + \left(-2\*m_{2}^2
 + \left(-2\*m_{3}
 - 4\right) \*m_{2}
 + \left(-2\*m_{3}
 - 2\right) \right) 
}
$\\
\hline
2&$\substack{
 \left(2\*m_{2}
 + \left(2\*m_{3}
 + 2\right) \right) \*m_{1}\\
 + \left(2\*m_{2}^2
 + \left(2\*m_{3}
 + 4\right) \*m_{2}
 + \left(2\*m_{3}
 + 2\right) \right) 
} $&$
\substack{
  \left(-2\*m_{2}
 - 2\right) \*m_{1}\\
 + \left(-2\*m_{2}^2
 + \left(-2\*m_{3}
 - 4\right) \*m_{2}
 - 1\right) 
}$
&
$\substack{
 \left(2\*m_{2}
 + \left(2\*m_{3}
 + 4\right) \right) \*m_{1}\\
 + \left(2\*m_{2}^2
 + \left(2\*m_{3}
 + 6\right) \*m_{2}
 + \left(4\*m_{3}
 + 5\right) \right) 
}
$
&$
\substack{
 \left(-2\*m_{2}
 - 2\right) \*m_{1}\\
 + \left(-2\*m_{2}^2
 + \left(-2\*m_{3}
 - 6\right) \*m_{2}
 + \left(-2\*m_{3}
 - 4\right) \right) 
}
$\\
\end{longtable}}
\end{center}
}

\subsubsection{Case 4\,:  $r_3 = 6$ }

{ \begin{center}
\tiny\renewcommand{\arraystretch}{2.4}
{\begin{longtable}{|c|c|c|c|c|}
\hline
\diagbox[width=3em]
{$r_1 $}{$r_2 $}&0&1&2&3\\
\hline
0&$\substack{
 \left(2\*m_{2}
 + 2\*m_{3}\right) \*m_{1}\\
 + \left(2\*m_{2}^2
 + 2\*m_{3}\*m_{2}
 + \left(-2\*m_{3}
 - 1\right) \right)
} $&$
\substack{
-2\*m_{2}\*m_{1}\\
 + \left(-2\*m_{2}^2
 + \left(-2\*m_{3}
 - 2\right) \*m_{2}
 + \left(-2\*m_{3}
 - 1\right) \right) 
}$
&
$\substack{
 \left(2\*m_{2}
 + \left(2\*m_{3}
 + 2\right) \right) \*m_{1}\\
 + \left(2\*m_{2}^2
 + \left(2\*m_{3}
 + 4\right) \*m_{2}
 + 1\right) 
}
$
&$
\substack{
 \left(-2\*m_{2}
 - 2\right) \*m_{1}\\
 + \left(-2\*m_{2}^2
 + \left(-2\*m_{3}
 - 4\right) \*m_{2}
 + \left(-2\*m_{3}
 - 2\right) \right) 
}
$\\
\hline
2&$\substack{
 \left(2\*m_{2}
 + 2\*m_{3}\right) \*m_{1}\\
 + \left(2\*m_{2}^2
 + \left(2\*m_{3}
 + 2\right) \*m_{2}
 - 1\right) 
} $&$
\substack{
 -2\*m_{2}\*m_{1}\\
 + \left(-2\*m_{2}^2
 + \left(-2\*m_{3}
 - 2\right) \*m_{2}\right)
}$
&
$\substack{
  \left(2\*m_{2}
 + \left(2\*m_{3}
 + 2\right) \right) \*m_{1}\\
 + \left(2\*m_{2}^2
 + \left(2\*m_{3}
 + 2\right) \*m_{2}
 + \left(2\*m_{3}
 + 1\right) \right) 
}
$
&$
\substack{
 \left(-2\*m_{2}
 - 2\right) \*m_{1}\\
 + \left(-2\*m_{2}^2
 + \left(-2\*m_{3}
 - 4\right) \*m_{2}
 - 1\right) 
}
$\\
\hline 
4& $
\substack{ 
 \left(2\*m_{2}
 + 2\*m_{3}\right) \*m_{1}\\
 + \left(2\*m_{2}^2
 + 2\*m_{3}\*m_{2}\right) 
}
$&
$\substack{
-2\*m_{2}\*m_{1}\\
 + \left(-2\*m_{2}^2
 + \left(-2\*m_{3}
 - 2\right) \*m_{2}\right)
}$
&$
\substack{
\left(2\*m_{2}
 + \left(2\*m_{3}
 + 2\right) \right) \*m_{1}\\
 + \left(2\*m_{2}^2
 + \left(2\*m_{3}
 + 4\right) \*m_{2}
 + \left(2\*m_{3}
 + 2\right) \right) 
}
$&
$\substack{
\left(-2\*m_{2}
 - 2\right) \*m_{1}\\
 + \left(-2\*m_{2}^2
 + \left(-2\*m_{3}
 - 4\right) \*m_{2}
 - 1\right) 
}$
\\
\end{longtable}}
\end{center}
}

{ \begin{center}
\tiny\renewcommand{\arraystretch}{2.4}
{\begin{longtable}{|c|c|c|c|c|}
\hline
\diagbox[width=3em]
{$r_1 $}{$r_2 $}&4&5&6&7\\
\hline
0&$\substack{
  \left(2\*m_{2}
 + \left(2\*m_{3}
 + 2\right) \right) \*m_{1}\\
 + \left(2\*m_{2}^2
 + \left(2\*m_{3}
 + 2\right) \*m_{2}\right) 
} $&$
\substack{
 -2\*m_{2}\*m_{1}\\
 + \left(-2\*m_{2}^2
 + \left(-2\*m_{3}
 - 2\right) \*m_{2}
 + \left(-2\*m_{3}
 - 1\right) \right) 
}$
&
$\substack{
\left(2\*m_{2}
 + \left(2\*m_{3}
 + 2\right) \right) \*m_{1}\\
 + \left(2\*m_{2}^2
 + \left(2\*m_{3}
 + 4\right) \*m_{2}
 + 1\right) 
}
$
&$
\substack{
 \left(-2\*m_{2}
 - 2\right) \*m_{1}\\
 + \left(-2\*m_{2}^2
 + \left(-2\*m_{3}
 - 4\right) \*m_{2}
 + \left(-2\*m_{3}
 - 2\right) \right) 
}
$\\
\hline
2&$\substack{
\left(2\*m_{2}
 + \left(2\*m_{3}
 + 2\right) \right) \*m_{1}\\
 + \left(2\*m_{2}^2
 + \left(2\*m_{3}
 + 4\right) \*m_{2}
 + \left(2\*m_{3}
 + 2\right) \right)
} $&$
\substack{
-2\*m_{2}\*m_{1}\\
 + \left(-2\*m_{2}^2
 + \left(-2\*m_{3}
 - 2\right) \*m_{2}\right)
}$
&
$\substack{
\left(2\*m_{2}
 + \left(2\*m_{3}
 + 2\right) \right) \*m_{1}\\
 + \left(2\*m_{2}^2
 + \left(2\*m_{3}
 + 4\right) \*m_{2}
 + \left(2\*m_{3}
 + 2\right) \right)  
}
$
&$
\substack{
 \left(-2\*m_{2}
 - 2\right) \*m_{1}\\
 + \left(-2\*m_{2}^2
 + \left(-2\*m_{3}
 - 6\right) \*m_{2}
 + \left(-2\*m_{3}
 - 4\right) \right) 
}
$\\
\hline 
4& $
\substack{ 
\left(2\*m_{2}
 + \left(2\*m_{3}
 + 2\right) \right) \*m_{1}\\
 + \left(2\*m_{2}^2
 + \left(2\*m_{3}
 + 4\right) \*m_{2}
 + \left(2\*m_{3}
 + 2\right) \right)
}
$&
$\substack{
 -2\*m_{2}\*m_{1}\\
 + \left(-2\*m_{2}^2
 + \left(-2\*m_{3}
 - 4\right) \*m_{2}
 + \left(-2\*m_{3}
 - 1\right) \right)
}$
&$
\substack{
  \left(2\*m_{2}
 + \left(2\*m_{3}
 + 2\right) \right) \*m_{1}\\
 + \left(2\*m_{2}^2
 + \left(2\*m_{3}
 + 4\right) \*m_{2}
 + \left(2\*m_{3}
 + 2\right) \right) 
}
$&
$\substack{
 \left(-2\*m_{2}
 - 2\right) \*m_{1}\\
 + \left(-2\*m_{2}^2
 + \left(-2\*m_{3}
 - 4\right) \*m_{2}
 - 1\right) 
}$
\\
\end{longtable}}
\end{center}
}

{ \begin{center}
\tiny\renewcommand{\arraystretch}{2.4}
{\begin{longtable}{|c|c|c|c|c|}
\hline
\diagbox[width=3em]
{$r_1 $}{$r_2 $}&8&9&10&11\\
\hline
0&$\substack{
\left(2\*m_{2}
 + \left(2\*m_{3}
 + 2\right) \right) \*m_{1}\\
 + \left(2\*m_{2}^2
 + \left(2\*m_{3}
 + 4\right) \*m_{2}
 + 1\right) 
} $&$
\substack{
\left(-2\*m_{2}
 - 2\right) \*m_{1}\\
 + \left(-2\*m_{2}^2
 + \left(-2\*m_{3}
 - 6\right) \*m_{2}
 + \left(-4\*m_{3}
 - 5\right) \right)
}$
&
$\substack{
 \left(2\*m_{2}
 + \left(2\*m_{3}
 + 4\right) \right) \*m_{1}\\
 + \left(2\*m_{2}^2
 + \left(2\*m_{3}
 + 6\right) \*m_{2}
 + \left(2\*m_{3}
 + 4\right) \right)
}
$
&$
\substack{
 \left(-2\*m_{2}
 - 2\right) \*m_{1}\\
 + \left(-2\*m_{2}^2
 + \left(-2\*m_{3}
 - 4\right) \*m_{2}
 + \left(-2\*m_{3}
 - 2\right) \right)
}
$\\
\hline
2&$\substack{
 \left(2\*m_{2}
 + \left(2\*m_{3}
 + 2\right) \right) \*m_{1}\\
 + \left(2\*m_{2}^2
 + \left(2\*m_{3}
 + 4\right) \*m_{2}
 + \left(2\*m_{3}
 + 2\right) \right) 
} $&$
\substack{
 \left(-2\*m_{2}
 - 2\right) \*m_{1}\\
 + \left(-2\*m_{2}^2
 + \left(-2\*m_{3}
 - 4\right) \*m_{2}
 - 1\right) 
}$
&
$\substack{
\left(2\*m_{2}
 + \left(2\*m_{3}
 + 4\right) \right) \*m_{1}\\
 + \left(2\*m_{2}^2
 + \left(2\*m_{3}
 + 6\right) \*m_{2}
 + \left(4\*m_{3}
 + 5\right) \right)
}
$
&$
\substack{
  \left(-2\*m_{2}
 - 2\right) \*m_{1}\\
 + \left(-2\*m_{2}^2
 + \left(-2\*m_{3}
 - 6\right) \*m_{2}
 + \left(-2\*m_{3}
 - 4\right) \right) 
}
$\\
\end{longtable}}
\end{center}
}

\subsubsection{Case 5\,:  $r_3 = 8$ }

{ \begin{center}
\tiny\renewcommand{\arraystretch}{2.4}
{\begin{longtable}{|c|c|c|c|c|}
\hline
\diagbox[width=3em]
{$r_1 $}{$r_2 $}&0&1&2&3\\
\hline
0&$\substack{
  \left(2\*m_{2}
 + 2\*m_{3}\right) \*m_{1}\\
 + \left(2\*m_{2}^2
 + \left(2\*m_{3}
 + 2\right) \*m_{2}
 + \left(-2\*m_{3}
 - 2\right) \right) 
} $&$
\substack{
  \left(-2\*m_{2}
 - 2\right) \*m_{1}\\
 + \left(-2\*m_{2}^2
 + \left(-2\*m_{3}
 - 4\right) \*m_{2}
 + \left(-2\*m_{3}
 - 2\right) \right) 
}$
&
$\substack{
\left(2\*m_{2}
 + \left(2\*m_{3}
 + 2\right) \right) \*m_{1}\\
 + \left(2\*m_{2}^2
 + \left(2\*m_{3}
 + 2\right) \*m_{2}\right) 
 
}
$
&$
\substack{
 -2\*m_{2}\*m_{1}\\
 + \left(-2\*m_{2}^2
 + \left(-2\*m_{3}
 - 2\right) \*m_{2}
 + \left(-2\*m_{3}
 - 1\right) \right)  
}
$\\
\hline
2&$\substack{
 \left(2\*m_{2}
 + 2\*m_{3}\right) \*m_{1}\\
 + \left(2\*m_{2}^2
 + 2\*m_{3}\*m_{2}\right) 
} $&$
\substack{
\left(-2\*m_{2}
 - 2\right) \*m_{1}\\
 + \left(-2\*m_{2}^2
 + \left(-2\*m_{3}
 - 4\right) \*m_{2}
 - 1\right) 

}$
&
$\substack{
\left(2\*m_{2}
 + \left(2\*m_{3}
 + 2\right) \right) \*m_{1}\\
 + \left(2\*m_{2}^2
 + \left(2\*m_{3}
 + 4\right) \*m_{2}
 + \left(2\*m_{3}
 + 2\right) \right) 
}
$
&$
\substack{
-2\*m_{2}\*m_{1}\\
 + \left(-2\*m_{2}^2
 + \left(-2\*m_{3}
 - 2\right) \*m_{2}\right) 
}
$\\
\hline 
4& $
\substack{ 
  \left(2\*m_{2}
 + 2\*m_{3}\right) \*m_{1}\\
 + \left(2\*m_{2}^2
 + \left(2\*m_{3}
 + 2\right) \*m_{2}
 - 1\right) 
}
$&
$\substack{
 \left(-2\*m_{2}
 - 2\right) \*m_{1}\\
 + \left(-2\*m_{2}^2
 + \left(-2\*m_{3}
 - 4\right) \*m_{2}
 - 1\right) 

}$
&$
\substack{
  \left(2\*m_{2}
 + \left(2\*m_{3}
 + 2\right) \right) \*m_{1}\\
 + \left(2\*m_{2}^2
 + \left(2\*m_{3}
 + 4\right) \*m_{2}
 + \left(2\*m_{3}
 + 2\right) \right) 
}
$&
$\substack{
-2\*m_{2}\*m_{1}\\
 + \left(-2\*m_{2}^2
 + \left(-2\*m_{3}
 - 2\right) \*m_{2}\right)
}$
\\
\end{longtable}}
\end{center}
}

{ \begin{center}
\tiny\renewcommand{\arraystretch}{2.4}
{\begin{longtable}{|c|c|c|c|c|}
\hline
\diagbox[width=3em]
{$r_1 $}{$r_2 $}&4&5&6&7\\
\hline
0&$\substack{
 \left(2\*m_{2}
 + \left(2\*m_{3}
 + 2\right) \right) \*m_{1}\\
 + \left(2\*m_{2}^2
 + \left(2\*m_{3}
 + 4\right) \*m_{2}
 + 1\right) 
} $&$
\substack{
  \left(-2\*m_{2}
 - 2\right) \*m_{1}\\
 + \left(-2\*m_{2}^2
 + \left(-2\*m_{3}
 - 4\right) \*m_{2}
 + \left(-2\*m_{3}
 - 2\right) \right) 
}$
&
$\substack{
\left(2\*m_{2}
 + \left(2\*m_{3}
 + 2\right) \right) \*m_{1}\\
 + \left(2\*m_{2}^2
 + \left(2\*m_{3}
 + 4\right) \*m_{2}
 + 1\right) 
}
$
&$
\substack{
 \left(-2\*m_{2}
 - 2\right) \*m_{1}\\
 + \left(-2\*m_{2}^2
 + \left(-2\*m_{3}
 - 4\right) \*m_{2}
 + \left(-2\*m_{3}
 - 2\right) \right)  
}
$\\
\hline
2&$\substack{
\left(2\*m_{2}
 + \left(2\*m_{3}
 + 2\right) \right) \*m_{1}\\
 + \left(2\*m_{2}^2
 + \left(2\*m_{3}
 + 4\right) \*m_{2}
 + \left(2\*m_{3}
 + 2\right) \right)  
} $&$
\substack{
 \left(-2\*m_{2}
 - 2\right) \*m_{1}\\
 + \left(-2\*m_{2}^2
 + \left(-2\*m_{3}
 - 4\right) \*m_{2}
 - 1\right) 
}$
&
$\substack{
 \left(2\*m_{2}
 + \left(2\*m_{3}
 + 2\right) \right) \*m_{1}\\
 + \left(2\*m_{2}^2
 + \left(2\*m_{3}
 + 4\right) \*m_{2}
 + \left(2\*m_{3}
 + 2\right) \right)
 
}
$
&$
\substack{
 \left(-2\*m_{2}
 - 2\right) \*m_{1}\\
 + \left(-2\*m_{2}^2
 + \left(-2\*m_{3}
 - 6\right) \*m_{2}
 + \left(-2\*m_{3}
 - 4\right) \right) 
}
$\\
\end{longtable}}
\end{center}
}

{ \begin{center}
\tiny\renewcommand{\arraystretch}{2.4}
{\begin{longtable}{|c|c|c|c|c|}
\hline
\diagbox[width=3em]
{$r_1 $}{$r_2 $}&8&9&10&11\\
\hline
0&$\substack{
 \left(2\*m_{2}
 + \left(2\*m_{3}
 + 2\right) \right) \*m_{1}\\
 + \left(2\*m_{2}^2
 + \left(2\*m_{3}
 + 4\right) \*m_{2}
 + 1\right) 
} $&$
\substack{
 \left(-2\*m_{2}
 - 2\right) \*m_{1}\\
 + \left(-2\*m_{2}^2
 + \left(-2\*m_{3}
 - 6\right) \*m_{2}
 + \left(-4\*m_{3}
 - 5\right) \right)  
}$
&
$\substack{
\left(2\*m_{2}
 + \left(2\*m_{3}
 + 4\right) \right) \*m_{1}\\
 + \left(2\*m_{2}^2
 + \left(2\*m_{3}
 + 6\right) \*m_{2}
 + \left(2\*m_{3}
 + 4\right) \right) 
 
}
$
&$
\substack{ \left(-2\*m_{2}
 - 2\right) \*m_{1}\\
 + \left(-2\*m_{2}^2
 + \left(-2\*m_{3}
 - 4\right) \*m_{2}
 + \left(-2\*m_{3}
 - 2\right) \right) 

}
$\\
\hline
2&$\substack{
\left(2\*m_{2}
 + \left(2\*m_{3}
 + 2\right) \right) \*m_{1}\\
 + \left(2\*m_{2}^2
 + \left(2\*m_{3}
 + 4\right) \*m_{2}
 + \left(2\*m_{3}
 + 2\right) \right)
} $&$
\substack{
 \left(-2\*m_{2}
 - 2\right) \*m_{1}\\
 + \left(-2\*m_{2}^2
 + \left(-2\*m_{3}
 - 4\right) \*m_{2}
 - 1\right) 
}$
&
$\substack{
\left(2\*m_{2}
 + \left(2\*m_{3}
 + 4\right) \right) \*m_{1}\\
 + \left(2\*m_{2}^2
 + \left(2\*m_{3}
 + 8\right) \*m_{2}
 + \left(4\*m_{3}
 + 8\right) \right) 
}
$
&$
\substack{
 \left(-2\*m_{2}
 - 2\right) \*m_{1}\\
 + \left(-2\*m_{2}^2
 + \left(-2\*m_{3}
 - 6\right) \*m_{2}
 + \left(-2\*m_{3}
 - 4\right) \right)
}
$\\
\end{longtable}}
\end{center}
}

\subsubsection{Case 6\,:  $r_3 = 10$ }

{ \begin{center}
\tiny\renewcommand{\arraystretch}{2.4}
{\begin{longtable}{|c|c|c|c|c|}
\hline
\diagbox[width=3em]
{$r_1 $}{$r_2 $}&0&1&2&3\\
\hline
0&$\substack{
  \left(2\*m_{2}
 + \left(2\*m_{3}
 + 2\right) \right) \*m_{1}\\
 + \left(2\*m_{2}^2
 + 2\*m_{3}\*m_{2}
 + \left(-2\*m_{3}
 - 2\right) \right)  
} $&$
\substack{
 \left(-2\*m_{2}
 + 2\right) \*m_{1}\\
 + \left(-2\*m_{2}^2
 + \left(-2\*m_{3}
 - 2\right) \*m_{2}
 + \left(-2\*m_{3}
 - 2\right) \right) 
}$
&
$\substack{
\left(2\*m_{2}
 + \left(2\*m_{3}
 + 4\right) \right) \*m_{1}\\
 + \left(2\*m_{2}^2
 + \left(2\*m_{3}
 + 4\right) \*m_{2}\right) 
}
$
&$
\substack{
-2\*m_{2}\*m_{1}\\
 + \left(-2\*m_{2}^2
 + \left(-2\*m_{3}
 - 4\right) \*m_{2}
 + \left(-2\*m_{3}
 - 3\right) \right)
}
$\\
\hline
2&$\substack{
\left(2\*m_{2}
 + \left(2\*m_{3}
 + 2\right) \right) \*m_{1}\\
 + \left(2\*m_{2}^2
 + \left(2\*m_{3}
 + 2\right) \*m_{2}\right) 
} $&$
\substack{
 \left(-2\*m_{2}
 + 2\right) \*m_{1}\\
 + \left(-2\*m_{2}^2
 + \left(-2\*m_{3}
 - 2\right) \*m_{2}
 + 1\right) 
}$
&
$\substack{
\left(2\*m_{2}
 + \left(2\*m_{3}
 + 4\right) \right) \*m_{1}\\
 + \left(2\*m_{2}^2
 + \left(2\*m_{3}
 + 4\right) \*m_{2}
 + \left(2\*m_{3}
 + 3\right) \right)  
 
}
$
&$
\substack{
-2\*m_{2}\*m_{1}\\
 + \left(-2\*m_{2}^2
 + \left(-2\*m_{3}
 - 4\right) \*m_{2}\right) 
}
$\\
\hline 
4& $
\substack{ 
\left(2\*m_{2}
 + \left(2\*m_{3}
 + 2\right) \right) \*m_{1}\\
 + \left(2\*m_{2}^2
 + \left(2\*m_{3}
 + 2\right) \*m_{2}\right) 
}
$&
$\substack{
 \left(-2\*m_{2}
 + 2\right) \*m_{1}\\
 + \left(-2\*m_{2}^2
 + \left(-2\*m_{3}
 - 2\right) \*m_{2}
 + 1\right) 
}$
&$
\substack{
 \left(2\*m_{2}
 + \left(2\*m_{3}
 + 4\right) \right) \*m_{1}\\
 + \left(2\*m_{2}^2
 + \left(2\*m_{3}
 + 4\right) \*m_{2}
 + \left(2\*m_{3}
 + 3\right) \right) 
}
$&
$\substack{
-2\*m_{2}\*m_{1}\\
 + \left(-2\*m_{2}^2
 + \left(-2\*m_{3}
 - 4\right) \*m_{2}\right) 
}$
\\
\end{longtable}}
\end{center}
}

{ \begin{center}
\tiny\renewcommand{\arraystretch}{2.4}
{\begin{longtable}{|c|c|c|c|c|}
\hline
\diagbox[width=3em]
{$r_1 $}{$r_2 $}&4&5&6&7\\
\hline
0&$\substack{
 \left(2\*m_{2}
 + \left(2\*m_{3}
 + 4\right) \right) \*m_{1}\\
 + \left(2\*m_{2}^2
 + \left(2\*m_{3}
 + 4\right) \*m_{2}\right) 
} $&$
\substack{
  -2\*m_{2}\*m_{1}\\
 + \left(-2\*m_{2}^2
 + \left(-2\*m_{3}
 - 4\right) \*m_{2}
 + \left(-2\*m_{3}
 - 3\right) \right) 
}$
&
$\substack{
 \left(2\*m_{2}
 + \left(2\*m_{3}
 + 4\right) \right) \*m_{1}\\
 + \left(2\*m_{2}^2
 + \left(2\*m_{3}
 + 4\right) \*m_{2}\right) 
 
}
$
&$
\substack{
-2\*m_{2}\*m_{1}\\
 + \left(-2\*m_{2}^2
 + \left(-2\*m_{3}
 - 4\right) \*m_{2}
 + \left(-2\*m_{3}
 - 3\right) \right) 
}
$\\
\hline
2&$\substack{
 \left(2\*m_{2}
 + \left(2\*m_{3}
 + 4\right) \right) \*m_{1}\\
 + \left(2\*m_{2}^2
 + \left(2\*m_{3}
 + 4\right) \*m_{2}
 + \left(2\*m_{3}
 + 3\right) \right) 
} $&$
\substack{
 -2\*m_{2}\*m_{1}\\
 + \left(-2\*m_{2}^2
 + \left(-2\*m_{3}
 - 4\right) \*m_{2}\right)  
}$
&
$\substack{
\left(2\*m_{2}
 + \left(2\*m_{3}
 + 4\right) \right) \*m_{1}\\
 + \left(2\*m_{2}^2
 + \left(2\*m_{3}
 + 4\right) \*m_{2}
 + \left(2\*m_{3}
 + 3\right) \right)
}
$
&$
\substack{
-2\*m_{2}\*m_{1}\\
 + \left(-2\*m_{2}^2
 + \left(-2\*m_{3}
 - 6\right) \*m_{2}
 + \left(-2\*m_{3}
 - 3\right) \right) 
}
$\\
\hline 
4& $
\substack{ 
\left(2\*m_{2}
 + \left(2\*m_{3}
 + 4\right) \right) \*m_{1}\\
 + \left(2\*m_{2}^2
 + \left(2\*m_{3}
 + 4\right) \*m_{2}
 + \left(2\*m_{3}
 + 3\right) \right) 
}
$&
$\substack{
-2\*m_{2}\*m_{1}\\
 + \left(-2\*m_{2}^2
 + \left(-2\*m_{3}
 - 6\right) \*m_{2}
 + \left(-2\*m_{3}
 - 3\right) \right) 
}$
&$
\substack{
 \left(2\*m_{2}
 + \left(2\*m_{3}
 + 4\right) \right) \*m_{1}\\
 + \left(2\*m_{2}^2
 + \left(2\*m_{3}
 + 6\right) \*m_{2}
 + \left(2\*m_{3}
 + 4\right) \right) 
}
$&
$\substack{
 -2\*m_{2}\*m_{1}\\
 + \left(-2\*m_{2}^2
 + \left(-2\*m_{3}
 - 4\right) \*m_{2}\right)  
}$
\\
\end{longtable}}
\end{center}
}

{ \begin{center}
\tiny\renewcommand{\arraystretch}{2.4}
{\begin{longtable}{|c|c|c|c|c|}
\hline
\diagbox[width=3em]
{$r_1 $}{$r_2 $}&8&9&10&11\\
\hline
0&$\substack{
\left(2\*m_{2}
 + \left(2\*m_{3}
 + 4\right) \right) \*m_{1}\\
 + \left(2\*m_{2}^2
 + \left(2\*m_{3}
 + 4\right) \*m_{2}\right)
} $&$
\substack{
  -2\*m_{2}\*m_{1}\\
 + \left(-2\*m_{2}^2
 + \left(-2\*m_{3}
 - 6\right) \*m_{2}
 + \left(-4\*m_{3}
 - 6\right) \right) 
}$
&
$\substack{
\left(2\*m_{2}
 + \left(2\*m_{3}
 + 6\right) \right) \*m_{1}\\
 + \left(2\*m_{2}^2
 + \left(2\*m_{3}
 + 8\right) \*m_{2}
 + \left(2\*m_{3}
 + 6\right) \right) 
}
$
&$
\substack{
\left(-2\*m_{2}
 - 2\right) \*m_{1}\\
 + \left(-2\*m_{2}^2
 + \left(-2\*m_{3}
 - 6\right) \*m_{2}
 + \left(-2\*m_{3}
 - 4\right) \right) 
}
$\\
\hline
2&$\substack{
\left(2\*m_{2}
 + \left(2\*m_{3}
 + 4\right) \right) \*m_{1}\\
 + \left(2\*m_{2}^2
 + \left(2\*m_{3}
 + 6\right) \*m_{2}
 + \left(2\*m_{3}
 + 4\right) \right) 
} $&$
\substack{
  -2\*m_{2}\*m_{1}\\
 + \left(-2\*m_{2}^2
 + \left(-2\*m_{3}
 - 4\right) \*m_{2}\right) 
}$
&
$\substack{
\left(2\*m_{2}
 + \left(2\*m_{3}
 + 6\right) \right) \*m_{1}\\
 + \left(2\*m_{2}^2
 + \left(2\*m_{3}
 + 6\right) \*m_{2}
 + \left(4\*m_{3}
 + 6\right) \right) 
}
$
&$
\substack{
\left(-2\*m_{2}
 - 2\right) \*m_{1}\\
 + \left(-2\*m_{2}^2
 + \left(-2\*m_{3}
 - 8\right) \*m_{2}
 + \left(-2\*m_{3}
 - 6\right) \right) 
}
$\\
\hline 
4& $
\substack{ 
\left(2\*m_{2}
 + \left(2\*m_{3}
 + 4\right) \right) \*m_{1}\\
 + \left(2\*m_{2}^2
 + \left(2\*m_{3}
 + 4\right) \*m_{2}
 + \left(2\*m_{3}
 + 3\right) \right)

}
$&
$\substack{
-2\*m_{2}\*m_{1}\\
 + \left(-2\*m_{2}^2
 + \left(-2\*m_{3}
 - 6\right) \*m_{2}
 + \left(-2\*m_{3}
 - 3\right) \right) 
}$
&$
\substack{
 \left(2\*m_{2}
 + \left(2\*m_{3}
 + 6\right) \right) \*m_{1}\\
 + \left(2\*m_{2}^2
 + \left(2\*m_{3}
 + 8\right) \*m_{2}
 + \left(4\*m_{3}
 + 9\right) \right)
}
$&
$\substack{
 \left(-2\*m_{2}
 - 2\right) \*m_{1}\\
 + \left(-2\*m_{2}^2
 + \left(-2\*m_{3}
 - 8\right) \*m_{2}
 + \left(-2\*m_{3}
 - 6\right) \right) 
}$
\\
\end{longtable}}
\end{center}
}
\subsection{The Euler characteristic of $\mathrm{SL}_4(\Z)$ with $\ell_1\equiv 1\pmod{2}$}

\subsubsection{Case 1\,: $r_3 = 1$}

{\begin{center}
\tiny\renewcommand{\arraystretch}{2.4}
{\begin{longtable}{|c|c|c|c|c|c|c|}
\hline
\diagbox[width=3em]
{$r_1 $}{$r_2 $}&0&2&4&6&8&10\\
\hline
1&$\substack{\left(-2\*m_{3}
 - 2\right) \*m_{1}\\
 + \left(-2\*m_{2}
 + \left(-2\*m_{3}
 - 1\right) \right)}$& $\substack{-2\*m_{3}\*m_{1}}$&$\substack{-2\*m_{3}\*m_{1}}$& $ \substack{ -2\*m_{3}\*m_{1}\\
 + \left(2\*m_{2}
 + 1\right)}$& $\substack{ \left(-2\*m_{3}
 - 2\right) \*m_{1}\\
 + \left(-4\*m_{2}
 + \left(-2\*m_{3}
 - 4\right) \right) }$& $\substack{ \left(-2\*m_{3}
 + 2\right) \*m_{1}\\
 + \left(4\*m_{2}
 + \left(2\*m_{3}
 + 4\right) \right) }$\\
\hline
3 &$\substack{ \left(-2\*m_{3}
 - 2\right) \*m_{1}\\
 + \left(-2\*m_{2}
 + \left(-2\*m_{3}
 - 1\right) \right)} $&$\substack{-2\*m_{3}\*m_{1} }$&$\substack{-2\*m_{3}\*m_{1}\\
 + \left(2\*m_{2}
 + 1\right)  }$&$\substack{-2\*m_{3}\*m_{1}\\
 + \left(-2\*m_{2}
 + \left(-2\*m_{3}
 - 1\right) \right)  }$&$\substack{ \left(-2\*m_{3}
 - 2\right) \*m_{1} }$&$\substack{ \left(-2\*m_{3}
 + 2\right) \*m_{1}\\
 + \left(2\*m_{2}
 + 3\right)  }$\\
\hline 
5&  $ \substack{ \left(-2\*m_{3}
 - 2\right) \*m_{1}\\
 + \left(-2\*m_{2}
 + \left(-2\*m_{3}
 - 1\right) \right) }$& $ \substack{-2\*m_{3}\*m_{1}
 + \\ \left(2\*m_{2}
 + 1\right)  }$ & $ \substack{ -2\*m_{3}\*m_{1}\\
 + \left(-2\*m_{2}
 + \left(-2\*m_{3}
 - 1\right) \right) }$ &  $ \substack{ -2\*m_{3}\*m_{1}\\
 + \left(2\*m_{2}
 + 1\right)}$&$\substack{\left(-2\*m_{3}
 - 2\right) \*m_{1}\\
 + \left(-2\*m_{2}
 + \left(-2\*m_{3}
 - 3\right) \right)  }$&$\substack{  \left(-2\*m_{3}
 + 2\right) \*m_{1}\\
 + \left(2\*m_{2}
 + 3\right) }$\\
\end{longtable}}
\end{center}
}

\subsubsection{Case 2\,: $r_3 = 3$}

{\begin{center}
\tiny\renewcommand{\arraystretch}{2.4}
{\begin{longtable}{|c|c|c|c|c|c|c|}
\hline
\diagbox[width=3em]
{$r_1 $}{$r_2 $}&0&2&4&6&8&10\\
\hline
1&$\substack{ \left(-2\*m_{3}
 - 2\right) \*m_{1}\\
 + \left(-2\*m_{2}
 + \left(-2\*m_{3}
 - 1\right) \right) }$& $\substack{-2\*m_{3}\*m_{1}}$&$\substack{-2\*m_{3}\*m_{1}\\
 + \left(2\*m_{2}
 + 1\right)}$& $ \substack{ \left(-2\*m_{3}
 - 2\right) \*m_{1}\\
 + \left(-2\*m_{2}
 - 1\right) }$&$\substack{-2\*m_{3}\*m_{1}
 - 2\*m_{3}}$& $ \substack{-2\*m_{3}\*m_{1}\\
 + \left(2\*m_{2}
 + \left(2\*m_{3}
 + 3\right) \right) }$\\
\hline
3 &$\substack{ \left(-2\*m_{3}
 - 2\right) \*m_{1}\\
 + \left(-2\*m_{2}
 + \left(-2\*m_{3}
 - 1\right) \right)} $&$\substack{ -2\*m_{3}\*m_{1}\\
 + \left(2\*m_{2}
 + 1\right)  }$&$\substack{-2\*m_{3}\*m_{1}}$&$\substack{ \left(-2\*m_{3}
 - 2\right) \*m_{1}\\
 + \left(-2\*m_{2}
 + \left(-2\*m_{3}
 - 3\right) \right)}$&$\substack{ -2\*m_{3}\*m_{1}\\
 + \left(2\*m_{2}
 + 1\right)}$& $ \substack{ -2\*m_{3}\*m_{1} }$\\
\hline 
5&  $ \substack{\left(-2\*m_{3}
 - 2\right) \*m_{1}\\
 +  \left(-2\*m_{3}
 - 2\right)  }$& $ \substack{-2\*m_{3}\*m_{1} }$ & $ \substack{-2\*m_{3}\*m_{1}
 - 2\*m_{3} }$ &  $ \substack{\left(-2\*m_{3}
 - 2\right) \*m_{1} }$&$\substack{-2\*m_{3}\*m_{1}
 - 2\*m_{3}}$& $ \substack{ -2\*m_{3}\*m_{1}}$\\
\hline
7&$ \substack{ \left(-2\*m_{3}
 - 2\right) \*m_{1}\\
 + \left(-2\*m_{2}
 + \left(-2\*m_{3}
 - 1\right) \right) }$& $ \substack{-2\*m_{3}\*m_{1}
 - 2\*m_{3} }$ & $ \substack{ -2\*m_{3}\*m_{1}\\
 + \left(2\*m_{2}
 + 1\right) }$ &  $ \substack{ \left(-2\*m_{3}
 - 2\right) \*m_{1}\\
 + \left(-2\*m_{2}
 + \left(-2\*m_{3}
 - 3\right) \right)  }$ &$\substack{-2\*m_{3}\*m_{1}
 - 2\*m_{3}}$& $ \substack{ -2\*m_{3}\*m_{1}\\
 + \left(2\*m_{2}
 + 3\right) }$\\
\end{longtable}}
\end{center}
}

\subsubsection{Case 3\,: $r_3 = 5$}

{\begin{center}
\tiny\renewcommand{\arraystretch}{2.4}
{\begin{longtable}{|c|c|c|c|c|c|c|}
\hline
\diagbox[width=3em]
{$r_1 $}{$r_2 $}&0&2&4&6&8&10\\
\hline
1&$\substack{\left(-2\*m_{3}
 - 2\right) \*m_{1}\\
 + \left(-2\*m_{2}
 + \left(-2\*m_{3}
 - 1\right) \right) }$& $\substack{-2\*m_{3}\*m_{1}\\
 + \left(2\*m_{2}
 + 1\right)}$&$\substack{\left(-2\*m_{3}
 - 2\right) \*m_{1}\\
 + \left(-2\*m_{2}
 - 1\right) }$& $ \substack{-2\*m_{3}\*m_{1}\\
 + \left(2\*m_{2}
 + 1\right) }$&$\substack{ \left(-2\*m_{3}
 - 2\right) \*m_{1}\\
 + \left(-2\*m_{2}
 + \left(-2\*m_{3}
 - 3\right) \right)}$& $ \substack{ -2\*m_{3}\*m_{1}\\
 + \left(2\*m_{2}
 + \left(2\*m_{3}
 + 3\right) \right)}$\\
\hline
3 &$\substack{ \left(-2\*m_{3}
 - 2\right) \*m_{1}\\
 +  \left(-2\*m_{3}
 - 2\right) } $&$\substack{ -2\*m_{3}\*m_{1}}$&$\substack{\left(-2\*m_{3}
 - 2\right) \*m_{1} }$&$\substack{-2\*m_{3}\*m_{1}
 - 2\*m_{3} }$&$\substack{\left(-2\*m_{3}
 - 2\right) \*m_{1}}$& $ \substack{-2\*m_{3}\*m_{1} }$\\
\hline 
5&  $ \substack{\left(-2\*m_{3}
 - 2\right) \*m_{1}\\
 + \left(-2\*m_{2}
 + \left(-2\*m_{3}
 - 1\right) \right)  }$& $ \substack{-2\*m_{3}\*m_{1}\\
 + \left(2\*m_{2}
 + 1\right) }$ & $ \substack{ \left(-2\*m_{3}
 - 2\right) \*m_{1}\\
 + \left(-2\*m_{2}
 + \left(-2\*m_{3}
 - 3\right) \right) }$ &  $ \substack{ -2\*m_{3}\*m_{1}\\
 + \left(2\*m_{2}
 + 1\right) }$&$\substack{ \left(-2\*m_{3}
 - 2\right) \*m_{1}\\
 + \left(-2\*m_{2}
 + \left(-2\*m_{3}
 - 3\right) \right) }$& $ \substack{-2\*m_{3}\*m_{1}\\
 + \left(2\*m_{2}
 + 3\right)  }$\\
\hline
7&$ \substack{ \left(-2\*m_{3}
 - 2\right) \*m_{1}\\
 +  \left(-2\*m_{3}
 - 2\right)  }$& $ \substack{ -2\*m_{3}\*m_{1}
 - 2\*m_{3}}$ & $ \substack{\left(-2\*m_{3}
 - 2\right) \*m_{1} }$ &  $ \substack{-2\*m_{3}\*m_{1}
 - 2\*m_{3} }$ &$\substack{ \left(-2\*m_{3}
 - 2\right) \*m_{1}\\
 +  \left(-2\*m_{3}
 - 2\right) }$& $ \substack{-2\*m_{3}\*m_{1} }$\\
\end{longtable}}
\end{center}
}

\subsubsection{Case 4\,: $r_3 = 7$}

{\begin{center}
\tiny\renewcommand{\arraystretch}{2.4}
{\begin{longtable}{|c|c|c|c|c|c|c|}
\hline
\diagbox[width=3em]
{$r_1 $}{$r_2 $}&0&2&4&6&8&10\\
\hline
1&$\substack{ \left(-2\*m_{3}
 - 2\right) \*m_{1}\\
 +  \left(-2\*m_{3}
 - 2\right)  }$& $\substack{\left(-2\*m_{3}
 - 2\right) \*m_{1}\\
 + \left(-2\*m_{2}
 - 1\right) }$&$\substack{-2\*m_{3}\*m_{1}\\
 + \left(2\*m_{2}
 + 1\right) }$& $ \substack{ \left(-2\*m_{3}
 - 2\right) \*m_{1} }$&$\substack{ \left(-2\*m_{3}
 - 2\right) \*m_{1}\\
 + \left(-2\*m_{2}
 + \left(-2\*m_{3}
 - 3\right) \right) }$& $ \substack{-2\*m_{3}\*m_{1}\\
 + \left(2\*m_{2}
 + \left(2\*m_{3}
 + 3\right) \right) }$\\
\hline
3 &$\substack{ \left(-2\*m_{3}
 - 2\right) \*m_{1}\\
 + \left(-2\*m_{2}
 + \left(-2\*m_{3}
 - 1\right) \right)} $&$\substack{ \left(-2\*m_{3}
 - 2\right) \*m_{1}}$&$\substack{ -2\*m_{3}\*m_{1}\\
 + \left(2\*m_{2}
 + 1\right) }$&$\substack{\left(-2\*m_{3}
 - 2\right) \*m_{1}\\
 + \left(-2\*m_{2}
 + \left(-2\*m_{3}
 - 3\right) \right)  }$&$\substack{ \left(-2\*m_{3}
 - 2\right) \*m_{1}}$& $ \substack{ -2\*m_{3}\*m_{1}\\
 + \left(2\*m_{2}
 + 3\right)}$\\
\hline 
5&  $ \substack{\left(-2\*m_{3}
 - 2\right) \*m_{1}\\
 +  \left(-2\*m_{3}
 - 2\right)  }$& $ \substack{ \left(-2\*m_{3}
 - 2\right) \*m_{1}}$ & $ \substack{ -2\*m_{3}\*m_{1}
 - 2\*m_{3}}$ &  $ \substack{ \left(-2\*m_{3}
 - 2\right) \*m_{1} }$&$\substack{ \left(-2\*m_{3}
 - 2\right) \*m_{1}\\
 +  \left(-2\*m_{3}
 - 2\right) }$& $ \substack{-2\*m_{3}\*m_{1} }$\\
\hline
7&$ \substack{\left(-2\*m_{3}
 - 2\right) \*m_{1}\\
 +  \left(-2\*m_{3}
 - 2\right)  }$& $ \substack{ \left(-2\*m_{3}
 - 2\right) \*m_{1}\\
 + \left(-2\*m_{2}
 + \left(-2\*m_{3}
 - 3\right) \right)}$ & $ \substack{-2\*m_{3}\*m_{1}\\
 + \left(2\*m_{2}
 + 1\right)  }$ &  $ \substack{\left(-2\*m_{3}
 - 2\right) \*m_{1}\\
 +  \left(-2\*m_{3}
 - 2\right)  }$ &$\substack{ \left(-2\*m_{3}
 - 2\right) \*m_{1}\\
 + \left(-2\*m_{2}
 + \left(-2\*m_{3}
 - 3\right) \right) }$& $ \substack{-2\*m_{3}\*m_{1}\\
 + \left(2\*m_{2}
 + 3\right)  }$\\
\end{longtable}}
\end{center}
}

\subsubsection{Case 5\,: $r_3 = 9$}

{\begin{center}
\tiny\renewcommand{\arraystretch}{2.4}
{\begin{longtable}{|c|c|c|c|c|c|c|}
\hline
\diagbox[width=3em]
{$r_1 $}{$r_2 $}&0&2&4&6&8&10\\
\hline
1&$\substack{\left(-2\*m_{3}
 - 4\right) \*m_{1}\\
 + \left(-4\*m_{2}
 + \left(-2\*m_{3}
 - 2\right) \right) }$& $\substack{-2\*m_{3}\*m_{1}\\
 + \left(2\*m_{2}
 + 1\right) }$&$\substack{\left(-2\*m_{3}
 - 2\right) \*m_{1}}$& $ \substack{ \left(-2\*m_{3}
 - 2\right) \*m_{1}}$&$\substack{\left(-2\*m_{3}
 - 2\right) \*m_{1}\\
 + \left(-2\*m_{2}
 + \left(-2\*m_{3}
 - 3\right) \right) }$& $ \substack{-2\*m_{3}\*m_{1}\\
 + \left(4\*m_{2}
 + \left(2\*m_{3}
 + 6\right) \right)  }$\\
\hline
3 &$\substack{\left(-2\*m_{3}
 - 4\right) \*m_{1}\\
 + \left(-2\*m_{2}
 + \left(-2\*m_{3}
 - 3\right) \right)} $&$\substack{ -2\*m_{3}\*m_{1}\\
 + \left(2\*m_{2}
 + 1\right) }$&$\substack{ \left(-2\*m_{3}
 - 2\right) \*m_{1} }$&$\substack{ \left(-2\*m_{3}
 - 2\right) \*m_{1}\\
 + \left(-2\*m_{2}
 + \left(-2\*m_{3}
 - 3\right) \right)}$&$\substack{\left(-2\*m_{3}
 - 2\right) \*m_{1}\\
 + \left(2\*m_{2}
 + 1\right) }$& $ \substack{ -2\*m_{3}\*m_{1}}$\\
\hline 
5&  $ \substack{ \left(-2\*m_{3}
 - 4\right) \*m_{1}\\
 + \left(-2\*m_{2}
 + \left(-2\*m_{3}
 - 3\right) \right) }$& $ \substack{-2\*m_{3}\*m_{1}\\
 + \left(2\*m_{2}
 + 1\right)  }$ & $ \substack{\left(-2\*m_{3}
 - 2\right) \*m_{1}\\
 + \left(-2\*m_{2}
 + \left(-2\*m_{3}
 - 3\right) \right)  }$ &  $ \substack{  \left(-2\*m_{3}
 - 2\right) \*m_{1}\\
 + \left(2\*m_{2}
 + 1\right) }$&$\substack{\left(-2\*m_{3}
 - 2\right) \*m_{1}\\
 + \left(-2\*m_{2}
 + \left(-2\*m_{3}
 - 3\right) \right) }$& $ \substack{ -2\*m_{3}\*m_{1}\\
 + \left(2\*m_{2}
 + 3\right) }$\\
\end{longtable}}
\end{center}
}

\subsubsection{Case 6\,: $r_3 = 11$}

{\begin{center}
\tiny\renewcommand{\arraystretch}{2.4}
{\begin{longtable}{|c|c|c|c|c|c|c|}
\hline
\diagbox[width=3em]
{$r_1 $}{$r_2 $}&0&2&4&6&8&10\\
\hline
1&$\substack{ \left(-2\*m_{3}
 - 2\right) \*m_{1}\\
 +  \left(-2\*m_{3}
 - 2\right) }$& $\substack{ \left(-2\*m_{3}
 - 2\right) \*m_{1}}$&$\substack{\left(-2\*m_{3}
 - 2\right) \*m_{1}}$& $ \substack{  \left(-2\*m_{3}
 - 2\right) \*m_{1}}$&$\substack{\left(-2\*m_{3}
 - 2\right) \*m_{1}\\
 +  \left(-2\*m_{3}
 - 2\right) 
}$& $ \substack{\left(-2\*m_{3}
 - 2\right) \*m_{1}\\
 +  \left(2\*m_{3}
 + 2\right)  }$\\
\hline
3 &$\substack{\left(-2\*m_{3}
 - 2\right) \*m_{1}\\
 +  \left(-2\*m_{3}
 - 2\right)} $&$\substack{ \left(-2\*m_{3}
 - 2\right) \*m_{1} }$&$\substack{ \left(-2\*m_{3}
 - 2\right) \*m_{1}}$&$\substack{ \left(-2\*m_{3}
 - 2\right) \*m_{1}\\
 +  \left(-2\*m_{3}
 - 2\right) }$&$\substack{ \left(-2\*m_{3}
 - 2\right) \*m_{1}}$& $ \substack{ \left(-2\*m_{3}
 - 2\right) \*m_{1}
 }$\\
\hline 
5&  $ \substack{\left(-2\*m_{3}
 - 2\right) \*m_{1}\\
 +  \left(-2\*m_{3}
 - 2\right)  }$& $ \substack{ \left(-2\*m_{3}
 - 2\right) \*m_{1}}$ & $ \substack{\left(-2\*m_{3}
 - 2\right) \*m_{1}\\
 +  \left(-2\*m_{3}
 - 2\right)  }$ &  $ \substack{ \left(-2\*m_{3}
 - 2\right) \*m_{1}}$&$\substack{ \left(-2\*m_{3}
 - 2\right) \*m_{1}\\
 +  \left(-2\*m_{3}
 - 2\right) }$& $ \substack{\left(-2\*m_{3}
 - 2\right) \*m_{1} }$\\
\hline
7&$ \substack{  \left(-2\*m_{3}
 - 2\right) \*m_{1}\\
 +  \left(-2\*m_{3}
 - 2\right) }$& $ \substack{\left(-2\*m_{3}
 - 2\right) \*m_{1}\\
 +  \left(-2\*m_{3}
 - 2\right) }$ & $ \substack{\left(-2\*m_{3}
 - 2\right) \*m_{1} }$ &  $ \substack{ \left(-2\*m_{3}
 - 2\right) \*m_{1}\\
 +  \left(-2\*m_{3}
 - 2\right)  }$ &$\substack{\left(-2\*m_{3}
 - 2\right) \*m_{1}\\
 +  \left(-2\*m_{3}
 - 2\right)}$& $ \substack{\left(-2\*m_{3}
 - 2\right) \*m_{1} }$\\
\end{longtable}}
\end{center}
}

\clearpage
\section{Euler Characteristic Formula Tables for \texorpdfstring{$\mathrm{GL}_4(\mathbb{Z})$}{GL4(Z)}}\label{app:gl4-residue-tables}

This appendix preserves the complete computational data used in Section~\ref{se:Euler-GL4}.  For the residue tables, $\ell_i=12m_i+r_i$ with $0\leq r_i<12$ for $i=1,2,3$, while the determinant exponent is recorded only by $\varepsilon\equiv\ell_4\pmod2$.  The four table families contain all $1296$ potentially nonzero formulas from Theorem~\ref{thm:gl4-residue-formulas}.  

\subsection{The Euler characteristic of $\GL_4(\Z)$ with $ \ell_1\equiv 0\pmod{2}$  and $\ell_4\equiv 0\pmod{2}$}

\subsubsection{Case 1\,: $r_3 = 0$}

{\begin{center}
\tiny\renewcommand{\arraystretch}{2.4}
{
\begin{longtable}{|c|c|c|c|c|}
\hline
\diagbox[width=3em]
{$r_1 $}{$r_2 $}&0&1&2&3\\
\hline
0&$\substack{\left(m_{2}
 + \left(m_{3}
 - 1\right) \right) \*m_{1}\\
 + \left(m_{2}^2
 + m_{3}\*m_{2}
 + \left(-m_{3}
 + 1\right) \right)  }$& $\substack{\left(-m_{2}
 - 1\right) \*m_{1}\\
 + \left(-m_{2}^2
 + \left(-m_{3}
 - 1\right) \*m_{2}
 - m_{3}\right) }$&$\substack{\left(m_{2}
 + m_{3}\right) \*m_{1}\\
 + \left(m_{2}^2
 + \left(m_{3}
 + 1\right) \*m_{2}
 + 1\right)  }$& $ \substack{\left(-m_{2}
 - 1\right) \*m_{1}\\
 + \left(-m_{2}^2
 + \left(-m_{3}
 - 1\right) \*m_{2}
 - m_{3}\right)  }$\\
\hline
2 &$\substack{\left(m_{2}
 + \left(m_{3}
 - 1\right) \right) \*m_{1}\\
 + \left(m_{2}^2
 + m_{3}\*m_{2}\right) } $&$\substack{\left(-m_{2}
 - 1\right) \*m_{1}\\
 + \left(-m_{2}^2
 + \left(-m_{3}
 - 1\right) \*m_{2}
 - 1\right)  }$&$\substack{\left(m_{2}
 + m_{3}\right) \*m_{1}\\
 + \left(m_{2}^2
 + \left(m_{3}
 + 1\right) \*m_{2}
 + m_{3}\right)  }$&$\substack{\left(-m_{2}
 - 1\right) \*m_{1}\\
 + \left(-m_{2}^2
 + \left(-m_{3}
 - 1\right) \*m_{2}
 - 1\right) }$\\
\hline 
4&  $ \substack{\left(m_{2}
 + \left(m_{3}
 - 1\right) \right) \*m_{1}\\
 + \left(m_{2}^2
 + m_{3}\*m_{2}\right)  }$& $ \substack{ \left(-m_{2}
 - 1\right) \*m_{1}\\
 + \left(-m_{2}^2
 + \left(-m_{3}
 - 1\right) \*m_{2}
 - 1\right) }$ & $ \substack{ \left(m_{2}
 + m_{3}\right) \*m_{1}\\
 + \left(m_{2}^2
 + \left(m_{3}
 + 1\right) \*m_{2}
 + m_{3}\right)  }$ &  $ \substack{\left(-m_{2}
 - 1\right) \*m_{1}\\
 + \left(-m_{2}^2
 + \left(-m_{3}
 - 1\right) \*m_{2}
 - 1\right) }$\\
\end{longtable}}
\end{center}
}

{\begin{center}
\tiny\renewcommand{\arraystretch}{2.4}
{\begin{longtable}{|c|c|c|c|c|}
\hline
\diagbox[width=3em]
{$r_1 $}{$r_2 $}&4&5&6&7\\
\hline
0&$\substack{\left(m_{2}
 + m_{3}\right) \*m_{1}\\
 + \left(m_{2}^2
 + \left(m_{3}
 + 1\right) \*m_{2}
 + 1\right)  }$& $\substack{\left(-m_{2}
 - 1\right) \*m_{1}\\
 + \left(-m_{2}^2
 + \left(-m_{3}
 - 1\right) \*m_{2}
 - m_{3}\right)  }$&$\substack{ \left(m_{2}
 + m_{3}\right) \*m_{1}\\
 + \left(m_{2}^2
 + \left(m_{3}
 + 1\right) \*m_{2}
 + 1\right)  }$& $ \substack{\left(-m_{2}
 - 1\right) \*m_{1}\\
 + \left(-m_{2}^2
 + \left(-m_{3}
 - 1\right) \*m_{2}
 - m_{3}\right)  }$\\
\hline
2 &$\substack{\left(m_{2}
 + m_{3}\right) \*m_{1}\\
 + \left(m_{2}^2
 + \left(m_{3}
 + 1\right) \*m_{2}
 + m_{3}\right) } $&$\substack{ \left(-m_{2}
 - 1\right) \*m_{1}\\
 + \left(-m_{2}^2
 + \left(-m_{3}
 - 1\right) \*m_{2}
 - 1\right) }$&$\substack{\left(m_{2}
 + m_{3}\right) \*m_{1}\\
 + \left(m_{2}^2
 + \left(m_{3}
 + 2\right) \*m_{2}
 + \left(m_{3}
 + 1\right) \right)  }$&$\substack{\left(-m_{2}
 - 1\right) \*m_{1}\\
 + \left(-m_{2}^2
 + \left(-m_{3}
 - 2\right) \*m_{2}
 + \left(-m_{3}
 - 1\right) \right)  }$\\
\hline 
4&  $ \substack{\left(m_{2}
 + m_{3}\right) \*m_{1}\\
 + \left(m_{2}^2
 + \left(m_{3}
 + 2\right) \*m_{2}
 + \left(m_{3}
 + 1\right) \right)  }$& $ \substack{\left(-m_{2}
 - 1\right) \*m_{1}\\
 + \left(-m_{2}^2
 + \left(-m_{3}
 - 2\right) \*m_{2}
 + \left(-m_{3}
 - 1\right) \right) }$ & $ \substack{\left(m_{2}
 + m_{3}\right) \*m_{1}\\
 + \left(m_{2}^2
 + \left(m_{3}
 + 1\right) \*m_{2}
 + m_{3}\right)  }$ &  $ \substack{\left(-m_{2}
 - 1\right) \*m_{1}\\
 + \left(-m_{2}^2
 + \left(-m_{3}
 - 1\right) \*m_{2}
 - 1\right)  }$\\
\end{longtable}}
\end{center}
}

{\begin{center}
\tiny\renewcommand{\arraystretch}{2.4}
{\begin{longtable}{|c|c|c|c|c|}
\hline
\diagbox[width=3em]
{$r_1 $}{$r_2 $}&8&9&10&11\\
\hline
0&$\substack{\left(m_{2}
 + m_{3}\right) \*m_{1}\\
 + \left(m_{2}^2
 + \left(m_{3}
 + 2\right) \*m_{2}
 + 2\right)  }$& $\substack{\left(-m_{2}
 - 2\right) \*m_{1}\\
 + \left(-m_{2}^2
 + \left(-m_{3}
 - 3\right) \*m_{2}
 + \left(-2\*m_{3}
 - 2\right) \right)  }$&$\substack{\left(m_{2}
 + \left(m_{3}
 + 1\right) \right) \*m_{1}\\
 + \left(m_{2}^2
 + \left(m_{3}
 + 2\right) \*m_{2}
 + \left(m_{3}
 + 1\right) \right)  }$& $ \substack{\left(-m_{2}
 - 1\right) \*m_{1}\\
 + \left(-m_{2}^2
 + \left(-m_{3}
 - 1\right) \*m_{2}
 - m_{3}\right)  }$\\
\hline
2 &$\substack{\left(m_{2}
 + m_{3}\right) \*m_{1}\\
 + \left(m_{2}^2
 + \left(m_{3}
 + 1\right) \*m_{2}
 + m_{3}\right) } $&$\substack{ \left(-m_{2}
 - 2\right) \*m_{1}\\
 + \left(-m_{2}^2
 + \left(-m_{3}
 - 2\right) \*m_{2}
 - 2\right) }$&$\substack{\left(m_{2}
 + \left(m_{3}
 + 1\right) \right) \*m_{1}\\
 + \left(m_{2}^2
 + \left(m_{3}
 + 3\right) \*m_{2}
 + \left(2\*m_{3}
 + 2\right) \right) }$&$\substack{ \left(-m_{2}
 - 1\right) \*m_{1}\\
 + \left(-m_{2}^2
 + \left(-m_{3}
 - 2\right) \*m_{2}
 + \left(-m_{3}
 - 1\right) \right) }$\\
\hline 
4&  $ \substack{\left(m_{2}
 + m_{3}\right) \*m_{1}\\
 + \left(m_{2}^2
 + \left(m_{3}
 + 2\right) \*m_{2}
 + \left(m_{3}
 + 1\right) \right)  }$& $ \substack{\left(-m_{2}
 - 2\right) \*m_{1}\\
 + \left(-m_{2}^2
 + \left(-m_{3}
 - 3\right) \*m_{2}
 + \left(-m_{3}
 - 3\right) \right)  }$ & $ \substack{\left(m_{2}
 + \left(m_{3}
 + 1\right) \right) \*m_{1}\\
 + \left(m_{2}^2
 + \left(m_{3}
 + 3\right) \*m_{2}
 + \left(2\*m_{3}
 + 2\right) \right)  }$ &  $ \substack{\left(-m_{2}
 - 1\right) \*m_{1}\\
 + \left(-m_{2}^2
 + \left(-m_{3}
 - 2\right) \*m_{2}
 + \left(-m_{3}
 - 1\right) \right)  }$\\
\end{longtable}}
\end{center}
}

\subsubsection{Case 2\,: $r_3 = 2$}

{\begin{center}
\tiny\renewcommand{\arraystretch}{2.4}
{%
}
\end{center}
}

\subsubsection{Case 3\,: $r_3 = 4$}

{\begin{center}
\tiny\renewcommand{\arraystretch}{2.4}
{%
}
\end{center}
}

\subsubsection{Case 4\,: $r_3 = 6$}

{\begin{center}
\tiny\renewcommand{\arraystretch}{2.4}
{%
}
\end{center}
}

\subsubsection{Case 5\,: $r_3 = 8$}

{\begin{center}
\tiny\renewcommand{\arraystretch}{2.4}
{%
}
\end{center}
}

\subsubsection{Case 6\,: $r_3 = 10$}

{\begin{center}
\tiny\renewcommand{\arraystretch}{2.4}
{%
}
\end{center}
}

\subsection{The Euler characteristic of $\GL_4(\Z)$ with $\ell_1\equiv 1\pmod{2}$  and $\ell_4\equiv 0\pmod{2}$}

\subsubsection{Case 1\,: $r_3 = 1$}

{\begin{center}
\tiny\renewcommand{\arraystretch}{2.4}
{%
}
\end{center}
}

\subsubsection{Case 2\,: $r_3 = 3$}

{\begin{center}
\tiny\renewcommand{\arraystretch}{2.4}
{%
}
\end{center}
}

\subsubsection{Case 3\,: $r_3 = 5$}

{\begin{center}
\tiny\renewcommand{\arraystretch}{2.4}
{%
}
\end{center}
}

\subsubsection{Case 4\,: $r_3 = 7$}

{\begin{center}
\tiny\renewcommand{\arraystretch}{2.4}
{%
}
\end{center}
}

\subsubsection{Case 5\,: $r_3 = 9$}

{\begin{center}
\tiny\renewcommand{\arraystretch}{2.4}
{%
}
\end{center}
}

\subsubsection{Case 6\,: $r_3 = 11$}

{\begin{center}
\tiny\renewcommand{\arraystretch}{2.4}
{%
}
\end{center}
}

\subsection{The Euler characteristic of $\GL_4(\Z)$ with $ \ell_1\equiv 0\pmod{2}$ and $\ell_4\equiv 1\pmod{2}$}

\subsubsection{Case 1\,: $r_3 = 0$}

{\begin{center}
\tiny\renewcommand{\arraystretch}{2.4}
{%
}
\end{center}
}

\subsubsection{Case 2\,: $r_3 = 2$}

{\begin{center}
\tiny\renewcommand{\arraystretch}{2.4}
{%
}
\end{center}
}

\subsubsection{Case 3\,: $r_3 = 4$}

{\begin{center}
\tiny\renewcommand{\arraystretch}{2.4}
{%
}
\end{center}
}
\subsubsection{Case 4\,: $r_3 = 6$}

{\begin{center}
\tiny\renewcommand{\arraystretch}{2.4}
{%
}
\end{center}
}

\subsubsection{Case 5\,: $r_3 = 8$}

{\begin{center}
\tiny\renewcommand{\arraystretch}{2.4}
{%
}
\end{center}
}

\subsubsection{Case 6\,: $r_3 = 10$}

{\begin{center}
\tiny\renewcommand{\arraystretch}{2.4}
{%
}
\end{center}
}
\subsection{The Euler characteristic of $\GL_4(\Z)$ with $\ell_1\equiv 1\pmod{2}$  and $\ell_4\equiv 1\pmod{2}$}

\subsubsection{Case 1\,: $r_3 = 1$}

{\begin{center}
\tiny\renewcommand{\arraystretch}{2.4}
{%
}
\end{center}
}

\subsubsection{Case 2\,: $r_3 = 3$}

{\begin{center}
\tiny\renewcommand{\arraystretch}{2.4}
{%
}
\end{center}
}

\subsubsection{Case 3\,: $r_3 = 5$}

{\begin{center}
\tiny\renewcommand{\arraystretch}{2.4}
{%
}
\end{center}
}

\subsubsection{Case 4\,: $r_3 = 7$}

{\begin{center}
\tiny\renewcommand{\arraystretch}{2.4}
{%
}
\end{center}
}

\subsubsection{Case 5\,: $r_3 = 9$}

{\begin{center}
\tiny\renewcommand{\arraystretch}{2.4}
{%
}
\end{center}
}

\subsubsection{Case 6\,: $r_3 = 11$}

{\begin{center}
\tiny\renewcommand{\arraystretch}{2.4}
{%
}
\end{center}
}

\section*{Acknowledgements}
JB would like to thank the Department of Mathematics of  Christian-Albrechts-Universit\"at zu Kiel and the Hebrew University of Jerusalem (HUJI) for their support and excellent working conditions. In addition, JB was supported by the European Research Council (ERC) under the European Union’s Horizon Europe research and innovation programme (grant agreement No. 101163794, GroupHype).

The computations in this article were carried out with the aid of PARI/GP.

\nocite{}
\bibliographystyle{abbrv}
\bibliography{BD}

\end{document}